\documentclass[12pt,amsfonts]{article}
\usepackage{dsfont}
\usepackage{graphicx}
\usepackage{latexsym}
\usepackage{amssymb}
\usepackage{amsmath}
\usepackage{amsthm}
\usepackage{enumerate}
\usepackage{stmaryrd}
\usepackage{mathrsfs}
\usepackage{layout}
\usepackage{cases}
\usepackage{euscript}
\usepackage{accents}
\usepackage{cancel}

\newtheorem{theorem}{Theorem}[section]
\newtheorem{prop}[theorem]{Proposition}
\newtheorem{lemma}[theorem]{Lemma}
\newtheorem{definition}[theorem]{Definition}

\newtheorem{remark}[theorem]{Remark}

\RequirePackage[colorlinks,citecolor=blue,urlcolor=blue,backref=page]{hyperref}

\usepackage[english]{babel}
\usepackage[T1]{fontenc}
\usepackage[latin1]{inputenc}

\numberwithin{equation}{section}

\newcommand{\E}{\mathbb{E}}

\newcommand{\dd}{{\rm d}}
\newcommand{\SA}{\mathbf{S}}
\newcommand{\bbX}{\mathbb{X}}

\newcommand{\R}{\mathbb{R}}

\newcommand{\N}{\mathbb{N}}

\newcommand{\ba}{\begin{array}}
	\newcommand{\ea}{\end{array}}
\newcommand{\be}{\begin{equation}}
	\newcommand{\ee}{\end{equation}}
\newcommand{\bea}{\begin{eqnarray}}
	\newcommand{\eea}{\end{eqnarray}}
\newcommand{\beaa}{\begin{eqnarray*}}
	\newcommand{\eeaa}{\end{eqnarray*}}

\newcommand{\eps}{\varepsilon}

\def \M{\mathbb{M}}
\def \R{\mathbb{R}}
\def \X{\mathbb{X}}

\def\P{\mathbb{P}}
\def\Occ{\mathcal{O}}
\def\gaman{\gamma^{an}}
\def\gampr{\gamma^{pr}}
\def\cA{\mathcal{A}}
\def\cB{\mathcal{B}}

\def\cF{\mathcal{F}}

\def\cN{\mathcal{N}}

\def\cP{\mathcal{P}}

\def\cX{\mathcal{X}}

\def \tx{\tilde{x}}
\def \ty{\tilde{y}}

\def \md{\mathrm{d}}
\def\qed{ \hfill \vrule width.25cm height.25cm depth0cm\smallskip}
\def\bN{{\mathbf{N}}}

\usepackage{tikz}
\usetikzlibrary{arrows, arrows.meta}

\newcommand{\dbtilde}[1]{\accentset{\approx}{#1}}

\newcommand{\remove}[1]{}

\usepackage{color}

\makeatletter
\def\namedlabel#1#2{\begingroup
    #2%
    \def\@currentlabel{#2}%
    \phantomsection\label{#1}\endgroup
}
\makeatother

\usepackage{authblk}

\title{\bf Spectra of Poisson functionals \\ and applications in continuum percolation}
\author[1]{C. Bhattacharjee}
\affil[1]{Department of Mathematics, University of Hamburg}
\author[2]{G. Peccati}
\affil[2]{Department of Mathematics, University of Luxembourg}
\author[3]{D. Yogeshwaran} 
\affil[3]{Theoretical Statistics and Mathematics Unit, Indian Statistical Institute, Bangalore}

\date{\vspace{-2.5ex}}
\begin{document}
	
	\maketitle
\begin{abstract}  Let $\eta$ be a Poisson random measure (defined on some Polish space), and let $F(\eta)$ be a square-integrable functional of $\eta$. In this paper we define and study a new notion of {\it spectral point process} associated with $F(\eta)$, and use such an object to study sharp noise instability and sensitivity properties of planar critical continuum percolation models under spatial birth-death (Ornstein-Uhlenbeck) dynamics --- the notion of sharp noise instability being a natural strengthening of the absence of noise stability. The concept of spectral point process is defined by exploiting the Wiener-It\^o chaos expansion of $F$, and represents a natural continuum counterpart to the notion of {\it spectral sample}, as introduced in Garban, Pete and Schramm (2010), in the context of discrete percolation models. In the particular case where $\eta$ is a marked Poisson measure, we use Hoeffding-ANOVA decompositions to establish an explicit connection with the notion of {\it annealed spectral sample}, introduced in Vanneuville (2021) in the context of Poisson-Voronoi percolation. We also relate spectral processes with an appropriate notion of {\it pivotal processes}. As applications, we show sharp noise instability of crossing events in the critical Poisson Boolean model with unit-radius balls and, using an observation of Vanneuville, we obtain sharp noise sensitivity (as well as sharp noise instability) for crossing events in the Poisson Voronoi percolation model. As an important ingredient, we prove quasi-multiplicativity of the $4$-arm probabilities in the critical Poisson Boolean percolation model.
\end{abstract}
\\
\noindent{\bf AMS 2020 Classification:} 
  60D05, % Geometric probability and stochastic geometry
  60G55, % Point processes (e.g., Poisson, Cox, Hawkes processes)
  60J25,  %Continuous-time  Markov  processes  on  general state spaces
  60H99,  %	None of the above, but in this section  - stochastic analysis
  82B43.  %	Percolation  
  82C43 % Time-dependent percolation in statistical mechanics
  60K35 %Interacting random processes; statistical mechanics type models; percolation theory

	{
		\hypersetup{linkcolor=black}
		\tableofcontents
	}
	
	\section{Introduction}

\subsection{Overview}\label{ss:overintro}

In this paper we introduce and study a new notion of {\bf spectral point process} for generic (square-integrable) functionals of Poisson random measures, and use it to analyze noise stability and noise sensitivity properties of continuum percolation models. We will see in Section \ref{sec:comparison} that our definition of spectral point process is markedly different from that of {\bf annealed spectral sample} introduced in \cite{Vann2021} -- for which one needs the underlying structure of a marked point process -- and rather provides a direct counterpart to the notion of {\bf discrete spectral sample} for Boolean functions on discrete cubes (see \cite{Garban2011, GPS10, GS}), an object which is defined as follows.

\smallskip

For $n \in \N$, endow the discrete cube $\{-1,1\}^n$ with the uniform probability measure,  
and consider a Boolean mapping $f : \{-1,1\}^n \to \{-1,1\}$ with its Fourier-Walsh expansion given by
	$$
	f(x) = \sum_{S\subseteq [n]} \hat{f}(S) \chi_S, \quad x=(x_1,...,x_n)\in \{-1,1\}^n,
	$$ 
	where $\chi_S := \prod_{i\in S} x_i$ and $\hat{f}(S) := \langle f,\chi_S \rangle = \E [f \chi_S]$ (see e.g. \cite[Ch.\ 4]{GS}).  According to \cite{GPS10} (see also \cite[Chs.\ 9-10]{GS}), the {\bf spectral probability measure} $\hat{\mathbb{P}}_f$ on $2^{[n]}$ (the class of all subsets of $[n]:=\{1,2,\hdots,n\}$) is defined as
	$$
	\hat{\mathbb{P}}_f(S) := \hat{f}^2(S), \quad S\subseteq [n],$$
 where the fact that the total mass of $\hat{\mathbb{P}}_f$ is one follows from the orthogonality of the Fourier-Walsh components. 
 The canonical random element distributed according to $\hat{\mathbb{P}}_f$, which is therefore a random subset of $[n]$, is called the {\bf spectral sample} of $f$, and it is typically denoted by the symbol $\mathscr{S}_f$. The spectral sample $\mathscr{S}_f$, which is an analytic object, is tightly related to the so-called {\bf pivotal set} of $f$, which corresponds to the (random) collection of those integers $i\in [n]$ such that changing $x_i$ to $-x_i$ yields a change of sign for $f(x)$; see e.g.\ \cite[Section 1.3]{GS}. Due to its geometric nature, the study of the pivotal set is often much easier than that of the spectral sample. It is important to notice, however, that these two random sets have in general different distributions, although their one- and two-dimensional marginals do coincide, see \cite[Sections 9.1-9.2]{GS}.

\smallskip  

Spectral samples and pivotal sets are convenient tools for encoding the sensitivity of a given Boolean mapping to variations in its argument. As such, they are particularly well-adapted for addressing questions of sharp {\bf noise sensitivity} and {\bf noise stability} in critical percolation models, see for instance \cite{BGS, Garban2011, GPS10, GPS13, GPS18, GS}. As discussed e.g.\ in \cite{Garban2011, GS}, even though techniques related to hypercontractivity, influence inequalities and randomized algorithms can be successful in showing noise sensitivity \cite{BKS99,SS10}, they are in general not able to provide strong enough bounds for deducing sharp noise sensitivity; see \cite[Section 4.3]{Garban2011}.
The definition of spectral sample given in \cite{GPS10} has been the crucial element for deducing the first full proof of sharp noise sensitivity for a critical percolation model. The motivation for studying questions on noise sensitivity stems naturally from dynamical percolation \cite{Haggstrom1997,vandenberg1997,Steif2009}, as well as from computer science and statistical physics -- see e.g. \cite{BKS99,Garban2011,GS}. A powerful new strategy for deducing sharp noise sensitivity via differential inequalities has been recently inaugurated in \cite{Tassion2023}; see the end of the forthcoming Section \ref{ss:appintro} for further remarks on this approach.

\smallskip

The study of noise sensitivity in continuum percolation models is more recent.  In this context, the prevailing strategy has been based on discretization techniques (see e.g. \cite{Ahlberg18, ABGM14, GR22, HJM22, HV22}), while an intrinsic continuum approach via stopping sets (continuum counterpart of randomized algorithms) was initiated in \cite{LPY2021}. Similarly to the discrete case, neither discretization nor stopping set techniques seem to be enough in order to prove sharp noise sensitivity/stability in continuum percolation models. To overcome such a difficulty, the seminal work \cite{Vann2021} introduced the already mentioned notion of annealed spectral sample (relying upon the discrete spectral sample approach, see Definition \ref{d:annealed} below) to deduce sharp noise sensitivity for the Voronoi percolation model.\footnote{Note that \cite{ABGM14,LPY2021} consider the Ornstein-Uhlenbeck (OU) or spatial birth-death dynamics, as in our paper, but \cite{Ahlberg18,Vann2021} consider frozen and other dynamics as well} As already mentioned, the construction of annealed spectral samples from \cite{Vann2021} requires one to work within the framework of marked point processes, and does not naturally extend e.g.\ to Boolean percolation.

\smallskip

In this paper we introduce a notion of (continuum) spectral sample that applies to a generic (possibly marked) Poisson point process $\eta$, defined on a generic Polish space $(\mathbb{X},\cX)$ and with a $\sigma$-finite intensity measure $\lambda$. The starting point of our construction is the fact that the well-known {\bf Wiener-It\^o chaos expansion} (see \cite[Section 18.4]{LastPenrose17}) yields an isomorphism between $L^2(\sigma(\eta))$ (the space of square-integrable functionals of $\eta$) and the symmetric Fock space ${\bf H} = \bigoplus_{k = 0}^\infty H^{\odot k}$ associated with the Hilbert space $H = L^2(\mathbb{X},\cX, \lambda)$. Using such a result, one can attach a random point configuration to any Boolean random variable $F(\eta)\in \{-1,1\}$, according to the following procedure:

\begin{enumerate}[(i)]
\item Consider the image $(u_0, u_1,u_2, ...) \in {\bf H}$ of $F$ through the Wiener-It\^o isomorphism;
\item Select a random number $N\in \mathbb{N} = \{0,1,2,...\}$ by using the probability measure on $\mathbb{N}$ having mass function $k\mapsto \|u_k\|^2_{H^{\odot k}} = k!\|u_k\|^2_{H^{\otimes k}} $;
\item Sample a random element of $\mathbb{X}^N$ by using the probability measure
$$
A \mapsto\frac{\langle u_N, u_N \mathds{1}_A \rangle_{H^{\otimes N}}}{\| u_N\|_{H^{\otimes N}}^2}, \quad A\in \mathcal{X}^N. 
$$
\end{enumerate}
The resulting point process on $(\mathbb{X},\cX)$, that we will denote by $\gamma_F$, coincides with the notion of {spectral point process} formally introduced in Definition \ref{def:spproc} below (in more generality). We will see in Remark \ref{r:quantumtsirelson} that such a notion is in resonance with Tsirelson's definition of the {\bf spectral measure} associated with a given continuous factorization \cite{Tsir1, Tsir2}, as well as with a certain interpretation of Bosonic Fock spaces emerging in quantum field theory \cite{glimm, meyer}. Among the contributions of the forthcoming Section \ref{s:specpp}, we mention the following:

\begin{itemize}
\item[--] Proposition \ref{p:ell} provides an explicit characterization of the density of the factorial moment measures of $\gamma_F$ in terms of {\bf add-one cost} and {\bf remove-one cost} operators.

\item[--] Proposition \ref{p:ell2} shows that -- up to a multiplicative constant -- the intensity of $\gamma_F$ coincides with that of the {\bf pivotal point process} associated with $F$ (see Definition \ref{def:pivproc}). Our analysis also shows that -- in contrast to discrete spectral samples -- such a relation does not extend to factorial moment measures of higher orders, and in particular it does not hold for second-order factorial moment measures.

\item[--] When $\eta$ is a marked Poisson point process with a mark space given by the two-point set $\{-1,1\}$, Theorem \ref{t:kgampran} yields a full description of the factorial measures associated with {\bf the annealed spectral sample} of $F$ (as defined in \cite{Vann2021}, as well as in Definition \ref{d:annealed} below) and those of its {\bf projected spectral sample} (see Definition \ref{d:gammastella}). In particular, and somehow surprisingly, our results show that such factorial moment measures {\it do not} coincide and are explicitly related through the notion of {\bf Hoeffding-ANOVA} decomposition \cite{efronstein, serfling}. 

\item[--] Further, our explicit characterizations allow us to derive an identity for intensities of annealed and projected spectral samples of increasing functions; see \eqref{e:onepointcomparison}. In contrast, we show in Proposition \ref{p:comp_Pivotal} that the $k$-th factorial moment measures of pivotal point process and {\bf quenched pivotal process} differ only by a factor of $2^k$. 

\end{itemize}

\subsection{Applications to continuum percolation}\label{ss:appintro}
    
As announced, our second goal in this paper is to develop a general framework -- based on the notions of spectral and pivotal point processes -- allowing one to study sharp noise instability and noise sensitivity in continuum percolation models under the Ornstein-Uhlenbeck (OU) (or spatial birth-death) dynamics. As discussed below, the notion of ``sharp noise instability'' describes a natural strengthening of the property of being {\it not} stable with respect to noise -- see e.g.\ \cite{Garban2011,GS}. Towards this goal, in Section~\ref{ss.gen_sns_cont} we provide four geometric conditions for the so-called {\bf arm probabilities} and spectral/pivotal point processes that are sufficient to prove sharp noise instability in a dynamical continuum percolation model. The content of Section~\ref{ss.gen_sns_cont} can be regarded as a continuum percolation counterpart to the strategy of proof outlined in \cite[Section 5.2]{Garban2011}. 

 The {\bf Poisson Boolean model} and the {\bf Voronoi percolation model} are arguably the two most commonly studied continuum percolation models; we will now briefly introduce them. Consider a stationary Poisson process $\eta$ of intensity $\lambda > 0$ in $\R^2$. In the Boolean model, one defines $\Occ(\eta) = \cup_{x \in \eta}B_1(x)$ to be the {\bf occupied region}, where $B_1(x)$ denotes the closed ball of unit radius around $x \in \R^2$. For Voronoi percolation we take $\lambda=1$ and, for each $x \in \eta$, we independently colour the corresponding Voronoi cell in black, with probability $p$, or in white, with probability $1-p$. The occupied region $\Occ(\eta)$ is then defined as the union of the black cells. The occupied regions in both models are increasing (in distribution), when regarded as functionals of the parameters $\lambda$ and $p$, respectively.  As a consequence, one can show the existence of a critical parameter $\lambda_c>0$ (resp.\ $p_c = \frac{1}{2}$), such that the occupied region in the Boolean model (resp.\ Voronoi percolation) percolates (i.e., it contains an unbounded component with probability one) whenever $\lambda > \lambda_c$ (resp $p > p_c$), and does not percolate otherwise -- see e.g.\ \cite{hall85} and \cite{BO06} for classical proofs of these facts. We will consider both models at criticality i.e., the Boolean model at $\lambda=\lambda_c$ and the Voronoi percolation at $p = 1/2$, and study a {\bf spatial birth-death} (also called {\bf Ornstein-Uhlenbeck} (OU)) dynamical version of both, according to which points appear independently at rate $\lambda_c$ (and new points are coloured black with probability $p_c$ in the Voronoi model) and existing points disappear independently at rate $1$ -- see Section \ref{ss.gen_sns_cont} for further details. We denote the resulting collection of dynamically evolving random point configurations as $\{\eta^t : t \geq 0\}$, with $\eta^0 = \eta$. Observe that $\eta^t=_d \eta$ at any time $t$, that is, the law of $\eta$ is a stationary measure for the OU dynamics.

As anticipated, given a critical continuum percolation model, one is typically interested in gauging its sensitivity to small perturbations of $\eta$ (the question becomes easier away from criticality). To this end, one commonly considered observable is the event that there is a left-right (L-R) crossing of the box $W_L: =  [-L,L]^2$, $L>0$, through its occupied region $\Occ(\eta)$. Let $f_L$ denote the $\pm 1$-indicator that there is such an L-R occupied crossing of $W_L$ in $\Occ(\eta)$. The collection $\{ f_L : L \ge 1\}$ 
is said to be {\bf noise sensitive} if,  $\forall \, t>0$,
			$$
			\lim_{L \to \infty} {\rm Cov}(f_L(\eta), f_L(\eta^t)) =0.
			$$
Since $f_L$ is a Boolean function, this indeed implies asymptotic independence. Such a definition represents a natural counterpart to the classical notion of noise sensitive Boolean functions on discrete cubes discussed e.g.\ in \cite{Garban2011, GS}, and appears for instance in \cite{Ahlberg18, ABGM14, LPY2021, Vann2021}. The collection $\{ f_L : L \ge 1\}$ is said to exhibit {\bf sharp noise sensitivity} at the time-scale $1/A_L$, for some $\{A_L\}$ verifying $A_L \to \infty$ as $L \to \infty$, if
		\begin{equation}\label{eq:sNSens}
		\lim_{L \to \infty} {\rm Cov}(f_L(\eta), f_L(\eta^{t_L})) =0, \quad \text{when $t_L A_L \to \infty$},
		\end{equation}
and there exists a constant $c > 0$ such that the above limit is at least $c$ for any sequence $t_L$ verifying $t_L A_L \to 0$ as $L \to \infty$. See \cite[Cor.\ 1.2]{GPS10} and \cite[Ch.\ 11]{GS} for early appearances of analogous notions of sharp noise sensitivity in a discrete setting, as well as \cite{Tassion2023} for a comprehensive discussion of the existing literature.

Instead of noise sensitivity, a second and complementary approach to study sensitivity to noise is via noise stability. The collection $\{ f_L : L \ge 1\}$ is {\bf noise stable} if
		\begin{equation*}\label{e:noisestable}
		\lim_{t \to 0} \sup_{L} \mathbb{P}\{f_L(\eta) \neq f_L(\eta^{t})\} =0.
				\end{equation*}
Such a definition is a continuum percolation counterpart to the classical definition of noise stable Boolean function; see \cite{Garban2011, GS} for a discussion of noise stability in a discrete setting, as well as \cite{LPY2021} for characterizations of noise stability in a continuum framework. 

Note that, a non-degenerate sequence of $\pm 1$-indicators cannot simultaneously be both noise sensitive and noise stable: in this sense, one can regard the absence of noise stability as a weaker form of noise sensitivity. Similarly as for noise sensitivity, we will be interested in characterizing those sequences of $\pm 1$-indicators that escape noise stability according to a sharp mechanism. To this effect, we will say that that $\{ f_L : L \ge 1\}$ exhibits {\bf sharp noise instability} at time-scale $1/A_L$, for some $A_L \to \infty$, if
	\begin{equation}\label{eq:sNStab}
	    \lim_{L \to \infty} \mathbb{P}\{f_L(\eta) \neq f_L(\eta^{t_L})\}  =0 \quad \text{when $t_L A_L \to 0$},
	\end{equation}
		while the limit is uniformly bounded away from $0$ when $t_L A_L \to \infty$ as $L \to \infty$. 
  
  \begin{remark}{\rm Let $\{f_L\}$ be a non-degenerate sequence of $\pm 1$-indicators as above. We immediately emphasize the following points:
\begin{enumerate}
\item[(i)] As sharp noise sensitivity implies noise sensitivity, one has that the sharp noise instability of $\{f_L\}$ implies that $\{f_L\}$ is {\it not} noise stable.\footnote{Indeed, choosing $t_L=g(L)$ for some monotonically decreasing and continuous function $g: \R_+ \to \R_+$ with $\lim_{L \to \infty} g(L)=0$ (in particular, $g$ is invertible in some neighbourhood of $0$), such that $t_L A_L \to \infty$ (we can always choose such a function since $1/A_L \to 0$ as $L \to \infty$), sharp noise instability implies that
  $
  \liminf_{L \to \infty} \mathbb{P}\{f_L(\eta) \neq f_L(\eta^{g(L)})\}  \ge c
  $
  for some $c>0$, which implies
  $$
  \liminf_{t \to 0} \sup_{L} \mathbb{P}\{f_L(\eta) \neq f_L(\eta^{t})\} \ge \liminf_{t \to 0} \mathbb{P}\{f_{g^{-1}(t)}(\eta) \neq f_{g^{-1}(t)}(\eta^{t})\} \ge c.
  $$
  %is also bounded below by $c$, yielding no noise stability.
  } In particular the sequence $\{f_L\}$ can enjoy simultaneously sharp noise instability and sharp noise sensitivity -- see e.g.\ Theorem \ref{t:sharpNS_voronoi}, where the two properties are proved to hold for indicators of L-R occupied crossings in Voronoi percolation. Such a behaviour is expected in many other percolation models as well.
\item[(ii)] In general, sharp noise sensitivity and sharp noise instability are not comparable, the two provable implications being the following: (a) relation \eqref{eq:sNSens} (as $t_L A_L \to \infty$) implies that $\liminf_{L \to \infty} \mathbb{P}\{f_L(\eta) \neq f_L(\eta^{t_L})\} >0$ (yielding noise instability above the time-scale $1/A_L$), and (b) as $\lim_{L \to \infty} t_LA_L=0$, relation \eqref{eq:sNStab} implies that $$\liminf_{L \to \infty} {\rm Cov}(f_L(\eta), f_L(\eta^{t_L})) >0$$ (yielding no noise sensitivity below the time-scale $1/A_L$).
  \end{enumerate}
  }
  \end{remark}
  
  In applications, for both sharp noise sensitivity and instability, one typically chooses the sequence $\{ A_L : L \ge 1\}$ to be such that $A_L \sim L^2\alpha_4(L)$ where $\alpha_4(L)\equiv \alpha_4(1,L)$ is the so-called \textit{$4$-arm probability} i.e., the probability that there are $4$ arms of alternating types (occupied and vacant) from $\partial W_1$ to $\partial W_L$ inside the annulus $W_L \setminus W_1$ (see Section \ref{ss.gen_sns_cont} for a more detailed definition).
\smallskip 

As a main application of spectral and pivotal processes, we show sharp noise instability for the crossing functionals in the Boolean percolation model under the OU dynamics. 
 \begin{theorem}\label{t:sharpNS_Boolean}
		Let $\Occ(\eta)$ be the occupied region in a critical Poisson Boolean percolation model with intensity $\lambda_c$ and unit radius balls. Let $f_L$ be the $\pm 1$-indicator of L-R crossing in $\Occ(\eta) \cap W_L$, $L \geq 1$. Then $f_L$ exhibits sharp noise instability at time-scale $\frac{1}{L^2\alpha_4(L)}$, as $L \to \infty$. More precisely,  we have that $\lim_{L \to \infty} \mathbb{P}\{f_L(\eta) \neq f_L(\eta^{t_L})\}  =0$ if $t_L L^2\alpha_4(L) \to 0$, and $\liminf_{L \to \infty} \mathbb{P}\{f_L(\eta) \neq f_L(\eta^{t_L})\}  >0$, if $t_L L^2\alpha_4(L) \to \infty$. 
	\end{theorem}
The content of Theorem \ref{t:sharpNS_Boolean} complements the well-known (non-sharp) noise sensitivity results from \cite{ABGM14,LPY2021}. We remark here that it should be possible to extend Theorem~\ref{t:sharpNS_Boolean} to the case of balls with random but bounded radius. The value of $\lambda_c$ will be different in this case, but with some extra work one could verify the general sufficient conditions \ref{A1} -- \ref{A3} for sharp noise instability of continuum percolation models, stated in Section \ref{ss.gen_sns_cont}. On the other hand, extending our findings to Boolean percolation with unbounded radii with suitable tail decay (as e.g. in \cite{Ahlbergsharp18, hall85}) represents a significant challenge. Incidentally, it is known that Boolean models with finite $(2+\epsilon)$-moment assumption on the random radius of balls (which is close to the optimal condition for the existence of a critical intensity -- see \cite{hall85}) exhibit (non-sharp) noise sensitivity, see \cite[Theorem 10.1]{LPY2021}.

\smallskip

Similarly to the definition of $\alpha_4$, for $0<r<R$, denote by $\alpha_3^+(r,R)$ the so-called \textit{3-arm probability in the half-plane,} that is, the probability that there are $3$ arms of alternating types from $\partial W_r$ to $\partial W_R$ inside the annulus $W_R \setminus W_r$ restricted to the upper half-plane. A fundamental tool in the proof of Theorem \ref{t:sharpNS_Boolean} is the {\bf quasi-multiplicativity property} of the four arm probabilities, as well as the of the three arm probabilities in the half-plane, which are of independent interest. Our proof, which is detailed in Section \ref{sec:quasim_boolean} below, is based on a highly non-trivial adaptation of arguments first presented by H.\ Kesten in \cite{Kes87} in the framework of $2$-dimensional discrete percolation.
\begin{theorem}[Quasi-multiplicativity of $\alpha_4$ and $\alpha_3^+$] 
\label{t:quasimPBM}
Consider the critical Boolean percolation as above. Then for $1 \leq r_1 \leq r_2 \leq r_3$,
	$$
	\alpha_4(r_1,r_3) \asymp \alpha_4(r_1,r_2) \alpha_4(r_2,r_3),
	$$
and
 $$
	\alpha_3^+(r_1,r_3) \asymp \alpha_3^+(r_1,r_2) \alpha_3^+(r_2,r_3).
	$$
\end{theorem}

Finally, we prove sharp noise sensitivity, as well as sharp noise instability for the critical Voronoi percolation model under the OU dynamics by comparing our spectral point process with the annealed spectral sample\footnote{Such a comparison was indicated to us by H.\ Vanneuville \cite{Vann2022}.} introduced in \cite{Vann2021} and using our general comparison of annealed and projected spectral samples from Section~\ref{sec:comparison}. 
\begin{theorem}
		\label{t:sharpNS_voronoi}
		Let $f_L$ be the $\pm 1$-indicator of LR crossing of $\Occ(\eta) \cap W_L$, $L \ge 1$, in a critical Voronoi percolation model as described above. Then, $f_L$ exhibits both sharp noise instability and sharp noise sensitivity at time-scale $\frac{1}{L^2\alpha_4(L)}$ as $L \to \infty$. More precisely,  we have that $\lim_{L \to \infty} \mathbb{P}\{f_L(\eta) \neq f_L(\eta^{t_L})\}  =0$ if $t_L L^2\alpha_4(L) \to 0$, and $\lim_{L \to \infty} {\rm Cov}(f_L(\eta), f_L(\eta^{t_L})) =0$ when $t_L L^2\alpha_4(L) \to \infty$.
	\end{theorem}
A consequence of the arguments used in our proof is that $f_L$ is also noise stable under frozen dynamics at time-scale $t_L$ if $t_L L^2\alpha_4(L) \to 0$. Hence, it also follows that it is not noise sensitive at this time-scale, which was shown in \cite[Theorem 1.7]{Vann2021}.

\smallskip 

Although in this paper we obtain many analogues of the results on discrete spectral sample (including certain conditional probability results) in \cite{GPS10}, these still fall short of establishing sharp noise sensitivity for general continuum percolation models, such as the Boolean model. The reason is that sharp noise sensitivity typically requires one to prove very fine bounds on lower tail probabilities associated with the spectral sample (see e.g. \cite[Proposition 5.12]{Garban2011} and \cite[Theorem 2.5]{Vann2021}, as well as STEP 4 in Remark \ref{rem:outlineNS}): for the time being, these estimates are outside the scope of the spectral bounds deduced in the present work, which only allow us to prove the weaker concentration inequality stated in Proposition \ref{p:noNStabgen}; see also the outline of the proof approach to sharp noise sensitivity in \cite[Section 5]{Garban2011}. It is important to mention here that the new approach to sharp noise sensitivity in Bernoulli percolation recently introduced in \cite{Tassion2023}, via differential inequalities and using a dynamical version of quasi-multiplicativity, is likely to apply to continuum percolation models as well, and this might yield an alternate route for extending the results of the present paper.

The rest of the paper is organized as follows. After some point process preliminaries in Section \ref{ss:prem}, in Section \ref{sec:Spec} we introduce the notion of spectral and pivotal point processes and explicitly characterize their factorial moment measures. In Section \ref{sec:comparison}, we recall the notion of annealed spectral sample from \cite{Vann2021}, we characterize its factorial moment measure, and compare it with the so-called projected spectral sample (see Definition \ref{d:gammastella}). Additionally, we also compare factorial moment measures of pivotal and quenched pivotal processes. In Section \ref{s:sharp_Nst_cont_perc}, we develop applications to continuum percolation models. Firstly, we give four general geometric conditions to verify sharp noise instability in continuum percolation models in Section \ref{ss.gen_sns_cont}, and exploit them to show the same for the Poisson Boolean model, as stated in Theorem \ref{t:sharpNS_Boolean}. The conditions are verified in Sections \ref{sec:quasim_boolean}--\ref{ss:proofBoolean}, including the proof  of the already evoked quasi-multiplicativity property in Theorem \ref{t:quasimPBM}. Finally, in Section \ref{s:sharpNSvor} we show sharp noise sensitivity of Voronoi percolation (Theorem \ref{t:sharpNS_voronoi}). For convenience, some technical proofs are deferred to the Appendices.

	\section{Spectral point processes}
 \label{s:specpp}

The aim of this section is to introduce and study the notion of {\bf spectral point process} associated with a given functional of a Poisson point process. We start by presenting some useful preliminary material in Section \ref{ss:prem}. Section \ref{sec:Spec} discusses the definitions and basic properties of spectral samples and pivotal point processes. Finally, Section \ref{sec:comparison} establishes an explicit connection (together with several comparisons) between a projection of the spectral point process and the {\bf annealed spectral sample} introduced in \cite{Vann2021}, and between the {\bf pivotal} and {\bf quenched pivotal processes}.

For the rest of the paper, every random element is assumed to be defined on a common probability space $(\Omega,\cF,\mathbb{P})$, with $\mathbb{E}$ denoting the expectation with respect to $\mathbb{P}$.

 \subsection{Preliminaries on point processes}\label{ss:prem}
 Throughout the paper, we let $(\mathbb{X},\cX)$ denote a Polish (that is, metrizable, separable and complete) space endowed with its Borel $\sigma$-field. A measure on $(\mathbb{X}, \cX)$ is called {\bf locally finite} if it is finite on bounded sets (and therefore $\sigma$-finite). Write $\N_0 := \{0,1,2,\hdots\}$ and denote by $\bN\equiv \bN(\mathbb{X})$ the class of all measures on $(\mathbb{X},\cX)$ that are $\N_0$-valued when restricted to bounded sets (that is, $\bN$ is the collection of positive, integer-valued locally finite measures). For the rest of the paper, we endow $\bN$ with the smallest $\sigma$-field $\cN \equiv \mathcal{N}(\mathbb{X})$ such that $\mu\mapsto \mu(B)$ is measurable for all $B\in\cX$.\footnote{In some parts of the paper we will be led to work simultaneously with two distinct Polish spaces $\mathbb{X}$ and $\mathbb{X}'$ (for instance, in Section \ref{sec:comparison}, $\mathbb{X} = \R^d \times \{-1,1\}$ and $\mathbb{X}' = \R^d$); in this case, it will always be clear from the context whether the notation $(\bN, \mathcal{N})$ refers to $(\bN(\mathbb{X}), \mathcal{N}(\mathbb{X}))$ or $(\bN(\mathbb{X} '), \mathcal{N}(\mathbb{X}'))$.}

 A {\bf point process} is a random element $\eta$ taking values in $\bN$.
The {\bf intensity measure} of $\eta$ is the
measure $\mathbb{E}[\eta]$ defined by $\mathbb{E}[\eta](B):=\mathbb{E}[\eta(B)]$, $B\in\cX$. Since $\mathbb{X}$ is Polish, every point process $\eta$ is {\bf proper}, that is it admits the representation
$ \eta = \sum_{n=1}^{\eta(\mathbb{X})}\delta_{Y_n}$, where $\delta_y$ is the Dirac mass at $y$, and $\{Y_n : n\geq 1\}$ is a collection of random elements with values in $\mathbb{X}$ -- see e.g.\ \cite[Section 6.1]{LastPenrose17}.

\smallskip

Fix a locally finite measure $\lambda$ on $(\mathbb{X}, \cX)$ with no atoms. A {\bf  Poisson process with intensity} $\lambda$ is a point process $\eta$
satisfying the following properties: (i) for every
$A\in \mathcal{X}$, the random variable $\eta(A)$ has a Poisson distribution with parameter $\lambda(A)$ (with obvious conventions whenever $\lambda(A)\in \{0, \infty\}$), and (ii) given $r\in\N$ and
disjoint sets $A_1,\hdots,A_r\in \mathcal{X}$, the random variables
$\eta(A_1),\hdots,\eta(A_r)$ are stochastically independent. It is a well-known fact (see e.g.\ \cite[Ch.\ 2]{LastPenrose17}) that, 
since $\lambda$ is non-atomic, then $\mathbb{P}(\eta\in\bN_s)=1$, where $\bN_s \equiv \bN_s(\mathbb{X})$ is the set of all $\mu\in\bN$ such that $\mu(\{x\})\le 1$ for each $x\in\mathbb{X}$ (a point measure of this type is called {\bf simple}). Given a measurable mapping $F: \bN \rightarrow \R$, we let
	\begin{equation}\label{e:add1}
	D_x F (\eta) := F(\eta + \delta_x)  - F(\eta), \quad x \in \X,
	\end{equation}
	and
	\begin{equation}\label{e:remove1}
	D_x^{-} F (\eta) := (F(\eta) - F(\eta - \delta_x))\mathds{1}_{\{x\in \eta\}}, \quad x\in \mathbb{X}
	\end{equation}
	denote, respectively, the {\bf add-one cost} and the {\bf remove-one cost} operators associated with $\eta$, where we write $x\in \eta$ to indicate that $x$ is in the support of $\eta$ (that is, $\eta(\{x\})>0$). Such a notational convention will be maintained throughout the rest of the paper.

\smallskip

Given $\mu\in \bN$ and $n\geq 1$, we write $\mu^{(n)}$ to indicate the $n$-th {\bf factorial measure} associated with $\mu$, as defined e.g.\ in \cite[Ch.\ 4]{LastPenrose17}. In the special case where $\mu\in\bN_s$, one has that
$\mu^{(n)}$ is the measure on $(\mathbb{X}^n, \mathcal{X}^{ n})$ obtained by
removing from the support of $\mu^n$ every point $(x_1,\hdots,x_n)$ such
that $x_i = x_j$ for some integers $i\neq j$ (with $\mu^{(1)} = \mu$). If $\eta$ is a point process on $(\mathbb{X}, \cX)$, the (symmetric and deterministic) measure on $(\mathbb{X}^n, \mathcal{X}^n)$ given by $A\mapsto \mathbb{E}[\eta^{(n)}(A)]$, $A\in \mathcal{X}^n$ is called the $n$-th {\bf factorial moment measure} of $\eta$ -- see e.g.\ \cite[Section 4.4]{LastPenrose17} or \cite[Section 5.4]{daleyverejones}. If the factorial moment measure has a density $f$ with respect to $\lambda^n$, then $f$ is called  the $n$-\textbf{point correlation function} associated with $\eta$. Note that the $1$-point correlation function is simply the {\bf intensity function} of $\eta$. 

\smallskip

We will often
apply the following {\bf multivariate Mecke formula}
(see \cite[Theorem 4.4]{LastPenrose17}). If $\eta$ is a Poisson
process with intensity $\lambda$ as above, then, for all $n\geq 1$ and all
measurable mappings $f\colon\mathbb{X}^n \times \bN \to [0,\infty)$,
\begin{align} \label{e:mecke}
 &\mathbb{E}\left[ \int_{\mathbb{X}^n}  f(x_1,\ldots,x_n, \eta) \, \eta^{(n)}(\md x_1,\ldots,\md x_n) \right] \nonumber\\ 
& \qquad \qquad=\mathbb{E}\left[ \int_{\mathbb{X}^n} f(x_1,\ldots,x_n, \eta+\delta_{x_1} +\cdots +\delta_{x_n} ) 
\, \lambda^{n}(\md x_1,\ldots,\md x_n)  \right].
\end{align}
This in particular shows that $\lambda^n$ is the $n$-th factorial moment measure of $\eta$.
\smallskip

One of the staples of our approach is the following {\bf chaotic representation property} of a Poisson process $\eta$ with intensity $\lambda$: every measurable function
$F\colon \bN \to \R$ such that $\mathbb{E}[F(\eta)^2] < \infty$ admits a unique
{\bf chaotic decomposition} of the type
\begin{equation}\label{e:chaos}
F = F(\eta) = \sum_{k=0}^\infty I_k(u_k),
\end{equation}
where $I_0(u_0) = u_0: = \mathbb{E}[F(\eta)]$, and for $k \in \N$, the symbol $I_k$ indicates a {\bf multiple
Wiener-It\^o integral} of order $k$ with respect to the compensated measure $\hat{\eta} := \eta-\lambda$, the kernels
$u_k\in L^2(\lambda^k)$ (with $L^2(\lambda^0) = \R$) are $\lambda^k$-a.e.\ symmetric,
and the convergence of the series takes place in $L^2(\mathbb{P})$. The reader is referred
\cite[Chs.\ 12 and 18] {LastPenrose17} and \cite[Ch.\ 5]{PTbook} for more details. Observe that, for all $k,m\geq 0$ and all
pairs of a.e.\ symmetric kernels $u \in L^2(\lambda^k)$,
$v \in L^2(\lambda^m)$,
\begin{equation}\label{e:onreln}
\mathbb{E}[I_k(u) I_m(v)] = \mathds{1}_{\{k=m\}}\,k!\, \langle u, v \rangle_{L^2(\lambda^k)}.
\end{equation}
According to \cite[Eq.\ (18.20)]{LastPenrose17}, if  $F : \bN \to \R$ is as in \eqref{e:chaos},  then 
\begin{equation}
\label{e:lpformula}
u_k(x_1,\ldots,x_k) := \frac{1}{k!}\mathbb{E}[D^k_{x_1,\ldots,x_k}F(\eta)],
\end{equation}
{where $D^k$ is the $k$-th {\bf iterated difference operator}, defined as
\begin{equation}\label{e:addoneincexc}
D^k_{x_1,\ldots,x_k}F(\eta) := D_{x_1}(D^{k-1}_{x_2,\ldots,x_k}F(\eta)) = \sum_{A\subseteq [k]} (-1)^{k-|A|}F\left(\eta +\sum_{\ell\in A} \delta_{x_\ell} \right) ,
\end{equation}
where $D_x$ denotes the add-one cost operator defined at \eqref{e:add1}, and the last equality is easily deduced e.g.\ by recursion on $k$ -- see \cite[Section 18.1]{LastPenrose17}.
Combining \eqref{e:chaos}, \eqref{e:onreln} and \eqref{e:lpformula}, one obtains the relation
\begin{equation}
\label{e:FSident}
\mathbb{E}[F(\eta)^2] = \sum_{k=0}^{\infty} k! \|u_k\|^2_{L^2(\lambda^k)} = \mathbb{E}[F(\eta)]^2 + \sum_{k=1}^{\infty}\int_{\mathbb{X}^k}\frac{1}{k!}\mathbb{E}[D^k_{x_1,\ldots,x_k}F(\eta)]^2\lambda^k(\md x_1,\ldots,\md x_k),
\end{equation}
which will play a crucial role in the forthcoming definition of a spectral point processes.

\smallskip

To conclude the section, we record a useful result whose proof (left to the reader) can be deduced by a slight modification of the arguments in \cite[Proof of Lem.\ 3.1]{pecth}. For integers $1\leq\ell\leq k$, we write $k_{(\ell)} := k(k-1)\cdots (k-\ell+1)$.

\begin{lemma}\label{l:domd} Let $\eta$ be a Poisson process on $(\mathbb{X}, \mathcal{X})$ with a locally finite and non-atomic intensity $\lambda$. Let $F : \bN\to \R$ be a measurable mapping such that $\mathbb{E} [F(\eta)^2]<\infty$, and let \eqref{e:chaos} be the chaotic expansion of $F(\eta)$. Then, for every $\ell\geq 1$ and every measurable $C\subseteq \mathbb{X}$, 
\begin{eqnarray}\label{e:derchaos1}
&&\int_{\mathbb{C}^\ell} \mathbb{E}[( D^\ell_{x_1,\hdots,x_\ell} F(\eta) )^2] \, \lambda^\ell({\rm d}x_1,\ldots,{\rm d}x_\ell) \\
&&= \sum_{k=\ell}^\infty k_{(\ell)} k! \int_{C^\ell}  \|u_k (z_1,\ldots,z_\ell, \cdot )\|_{L^2(\lambda^{k-\ell})}^2{\rm d}\lambda^{\ell}(z_1,\ldots,z_\ell).\notag
\end{eqnarray}
Moreover, if either side of \eqref{e:derchaos1} is finite, one has that, $\lambda^\ell\otimes \mathbb{P}$--a.e. on $C^\ell\times \Omega$,
\begin{equation}\label{e:derchaos2}
D^\ell_{x_1,\ldots,x_\ell} F(\eta) = \sum_{k=\ell}^{\infty} k_{(\ell)} I_{k-\ell}\left(u_k(x_1,\ldots,x_\ell, \cdot)\right),
\end{equation}
where the infinite series converges in $L^2(C^\ell\times \Omega, \lambda^\ell \otimes \mathbb{P})$.
\end{lemma}

	\subsection{Spectral and pivotal point processes}\label{sec:Spec}
	
	\noindent In this section, we let $\eta$ be a Poisson point process on the Polish space $(\mathbb{X}, \cX)$, with non-atomic and locally finite intensity $\lambda$ (see Section \ref{ss:prem}). Consider a measurable mapping $F : {\bf N}\to \R$ such that $0<\mathbb{E}[F(\eta)^2]<\infty$, and let \eqref{e:chaos} be its chaotic decomposition. 
	In what follows, given such a mapping $F$, we will denote by $N = N_F$ the $\mathbb{N}_0$-valued random variable with distribution
	\begin{equation}\label{e:integer}
	\mathbb{P}[N = k] = k!\,\frac{\|u_k\|_{L^2(\lambda^k )}^2}{\E [F^2] }, \quad k \in \N_0,
	\end{equation}
where the fact that \eqref{e:integer} defines a probability measure on the integers is a consequence of \eqref{e:FSident}. The next definition introduces one of the central objects of our paper.	
	\begin{definition}\label{def:spproc}{\rm 
		The {\bf spectral point process} associated with $F$, denoted by $\gamma_F$, is the random point process on $(\mathbb{X}, \mathcal{X})$ given by the random measure
		$$
		\gamma_F (B) := \sum_{i=1}^{N} \delta_{X^{(N)}_i}(B),\quad B\in \mathcal{X},
		$$
		where: {\rm (i)} by convention, $\sum_{i=1}^0 \equiv 0$, {\rm (ii)} the random variable $N = N_F$ is defined at \eqref{e:integer}, {\rm (iii)} the random vectors $X^{(k)} =(X_1^{(k)},\ldots,X_k^{(k)})$, $k\geq 1$, have distribution
		$$
		\mathbb{P} [ X^{(k)} \subseteq \Gamma] = \frac{1}{\|u_k\|^2} \int_\Gamma u_k(x_1,\ldots,x_k)^2 \lambda(dx_1)\cdots \lambda(dx_k), \quad \Gamma\in \mathcal{X}^k,
		$$
		and {\rm (iv)} $N_F$ and $\{X^{(k)} : k\geq 1\}$ are stochastically independent.% (the joint distribution of the $X^{(k)}$'s is immaterial). 
  }
	\end{definition}
	\begin{remark}\label{r:quantumtsirelson}
 {\rm 
 \begin{enumerate}[(i)]
 \item Since $\lambda$ is non-atomic, one has that, $\mathbb{P}$--a.s., $\gamma_F$ is an element of ${\bf N}_s$ verifying ${\rm supp}\, \gamma_F := \{x\in \mathbb{X} : \gamma_F(\{x\}) =1 \} = \{X_1^{(N)}, \ldots, X_N^{(N)}\}$. 
 
 \item The probability measure on ${\bf N}_s$ induced by $\gamma_F$ implicitly appears in the quantum field particle interpretation of {\it Bosonic Fock spaces} -- see e.g.\ the classical references \cite[p.\ 105-108]{glimm} and \cite[Chs.\ IV and V]{meyer}.
 \item Let $\mathbb{X} = \R$, and $\lambda$ be the Lebesgue measure, and denote by $\mathcal{B}$ the Borel $\sigma$-field associated with the class $\mathbf{C}$ of all compact subsets of $\R$ endowed with the Hausdorff distance (see e.g.\ \cite[Section 3c]{Tsir2}). Since for every $k\in \N$, the mapping $(x_1,\ldots,x_k) \mapsto \{x_1,\ldots,x_k\}$ from $\R^k$ to $\mathbf{C}$ is continuous, one has that $ {\rm supp}\, \gamma_F = \{X_1^{(N)},\hdots, X_N^{(N)}\}$ is a well-defined ${\bf C}$-valued random element, and the map $\mathcal{B}\ni K \mapsto \mu_F(K):= \mathbb{P}[{\rm supp}\, \gamma_F \in K]$ defines a probability measure on $(\mathbf{C}, \mathcal{B})$ supported on the class of finite subsets of $\R$. A direct inspection (left to the reader) shows that when $\mathbb{E}[F^2]=1$, in the terminology of \cite[Section 3d]{Tsir2}, $\mu_F$ coincides with the {\bf spectral measure} uniquely associated with $F$ and with the {\it continuous factorization} $((\Omega, \mathcal{A}, \mathbb{P}), \{\mathcal{A}_{s,t}\})$, where $\mathcal{A}$ is the $\mathbb{P}$-completion of the $\sigma$-field generated by $\eta$ and, for finite $s\leq t$, the symbol $\mathcal{A}_{s,t}$ stands for the $\mathbb{P}$-completion of the $\sigma$-field generated by the class $\{\eta(B) : B\subseteq [s,t]\}$. See \cite[p.\ 274]{Tsir1}, as well as the recent reference \cite{las23}, for further details.
 \end{enumerate}
 
 }
	    
\end{remark}
    As discussed in the introduction, the spectral point processes $\gamma_F$ is a natural counterpart to the spectral sample $\mathscr{S}_f$ associated with a Boolean function $f$ on the discrete cube. Under some integrability assumption, the following result provides a characterization of the factorial moment measures $\E[\gamma_F^{(\ell)}]$ of $\gamma_F$, as defined in Section \ref{ss:prem}, for all $\ell\geq 1$.

	\begin{prop}\label{p:ell} Let $F : \bN\to \R$ be a measurable mapping such that $0<\mathbb{E} [F(\eta)^2]<\infty$, and assume that, for some $m\geq 1$ and a measurable set $C \subseteq \mathbb{X}$,
 \begin{equation}\label{e:integram}
\int_{C^m} \mathbb{E}[( D^m_{x_1,\ldots,x_m} F(\eta) )^2] \, {\rm d}\lambda^m(x_1,\ldots,x_m) <\infty.
\end{equation}
 Then, for all $\ell \leq m $, and all measurable and bounded $h : C^\ell \to \R_+$,
		\begin{eqnarray}\notag
			\E[F^2]\times \E\left[ \int_{C^\ell} h\, {\rm d}\gamma_F^{(\ell)} \right] \!\!&=&\!\! \E\left[ \int_{C^\ell} h(z_1,...,z_\ell) (D^{-} _{z_1}\cdots D^{ -}_{z_\ell}F)^2 {\rm d}\eta^{(\ell)}(z_1,...,z_\ell)\right]\\
			\!\!&=&\!\! \E\left[ \int_{C^\ell} h(z_1,...,z_\ell) (D^\ell_{z_1,...,z_\ell}F)^2 {\rm d}\lambda^{\ell}(z_1,...,z_\ell)\right] \label{e:rn}
			\\
			\notag \!\!&=&\!\! \E\left[ \int_{C^\ell} h(z_1,...,z_\ell) (D^\ell_{z_1,...,z_\ell}F)^2 {\rm d}\eta_0^{(\ell)}(z_1,...,z_\ell)\right],
		\end{eqnarray}
		where $F\equiv F(\eta)$ and $\eta_0$ is an independent copy of $\eta$. 	
	\end{prop} 
	\begin{proof} Fix $1\leq \ell \leq m$ and $h$ as in the statement. Let \eqref{e:chaos} be the chaotic expansion of $F(\eta)$. Since $\gamma_F$ is an element of ${\bf N}_s$ almost surely, one has that
		$$
		 \E[F^2]\times \E\left[ \int_{C^\ell} h\, {\rm d}\gamma_F^{(\ell)} \right] = \E[F^2] \times \E\left[ \int_{C^\ell} h\, d\gamma_F^{(\ell)} \mathds{1}\{ N_F\geq \ell\}\right],
		$$
		and the right-hand side of the previous relation equals 
		\begin{eqnarray*} 
			&& \sum_{k\geq \ell} k_{(\ell)} k! \int_{C^\ell} h(z_1,\ldots,z_\ell) \|u_k (z_1,\ldots,z_\ell, \cdot )\|_{L^2(\lambda^{k-\ell})}^2{\rm d}\lambda^{\ell}(z_1,\ldots,z_\ell)\\
			%&& = \sum_{k\geq \ell} k_{(\ell)}^2 (k-\ell)! \int_{C^\ell} h(z_1,\ldots,z_\ell) \|f_k (z_1,\ldots,z_\ell, \cdot )\|_{L^2(\lambda^{k-\ell})}^2 {\rm d}\lambda^{\ell}(z_1,\ldots,z_\ell)  \\
			&& = \E\left[ \int_{C^\ell} h(z_1,\ldots,z_\ell) (D^{\ell}_{z_1,\ldots,z_{\ell}} F)^2 {\rm d}\lambda^{\ell}(z_1,\ldots,z_\ell)\right],\notag
		\end{eqnarray*}
		where we have used relations \eqref{e:derchaos1} and \eqref{e:derchaos2} from Lemma \ref{l:domd}. The remaining identities in the statement follow from Mecke formula \eqref{e:mecke}.    
	\end{proof}

\begin{remark}\label{e:bagatelle}{\rm \begin{enumerate}
    \item[(i)] According to the terminology introduced in Section \ref{ss:prem}, relation \eqref{e:rn} implies that the $\ell$-th factorial moment measure associated with $\gamma_F$ is absolutely continuous with respect to $\lambda^\ell$ on the set $C^\ell$, with a Radon-Nikodym derivative given by the mapping $(z_1,...,z_\ell)\mapsto \mathbb{E}[(D_{z_1,...,z_\ell}^\ell F)^2]/\E[F^2]$. 
    \item[(ii)] If $| F(\mu) |\leq c$ for some finite constant $c>0$, then condition \eqref{e:integram} is trivially verified for all $m\geq 1$ and all measurable bounded $C\subset \mathbb{X}$ (recall that $\lambda$ is locally finite by assumption). Observe that, in this case, for all $\ell\geq 1$ and all measurable bounded $C$ one has that,
$$
\mathbb{E}[\gamma_F^{(\ell)}(C^\ell)] \leq (2\ell c)^2 \lambda(C)^\ell /\E [F^2]
$$
and the classical criterion stated in \cite[Proposition 4.12]{LastPenrose17} implies that the validity of relation \eqref{e:rn} for all $\ell\geq 1$ uniquely characterizes the distribution of $\gamma_F$ as a random element taking values in ${\bf N}$. 
\end{enumerate}

}
\end{remark}

In resonance with the discrete case (see e.g.\ \cite[Chs.\ 4 and 9]{GS} and \cite[Section 5]{Garban2011}), it will often be useful to compare $\gamma_F$ with the point process supported on {\it pivotal points} (these objects appear e.g.\ in {\it Margulis-Russo-type formulae} on the Poisson space -- see \cite[p.\ 209]{LastPenrose17}, as well as \cite{last14} and the references therein). 
	\begin{definition}\label{def:pivproc}{\rm
		Let $F: \bN \to \{-1,1\}$ be a measurable mapping. The set of {\bf pivotal points} associated with $F$ and $\eta$ is the random subset of $\mathbb{X}$ composed of all those $x \in \eta$ such that $D^-_xF(\eta)\neq 0$, where we adopted the notation \eqref{e:remove1}. In particular, the {\bf pivotal point process} associated with $F$ and $\eta$ is a random element in ${\bf N}_s$ given by
  $$
  \mathcal{P}_F(B) := \sum_{x\in \eta} \mathds{1}_{\{ D^-_x F(\eta) \neq 0\}} \, \delta_x(B), \quad B\in \mathcal{X}.
  $$}
	\end{definition}

The following statement is the counterpart of Proposition \ref{p:ell} for pivotal point processes associated with Boolean functions. The proof is an elementary consequence of \eqref{e:mecke} and is thus left to the reader. Note that the last sentence in the statement is again a consequence of \cite[Proposition 4.12]{LastPenrose17}.

\begin{prop}\label{p:ell2} Let $F: \bN \to \{-1,1\}$ be a measurable Boolean mapping. Then, for all $\ell\geq 1$, and all bounded and measurable $h: \mathbb{X}^\ell \to [0,\infty)$,  
\begin{eqnarray}\notag
			 \E\left[ \int_{\mathbb{X}^\ell} h \, {\rm d} {\mathcal P}_F^{(\ell)} \right] \!\!\!\!&= &\!\! \!\!\frac{1}{4^\ell} \E\left[ \int_{\mathbb{X}^\ell} h(z_1,...,z_\ell) (D^{-} _{z_1}F)^2 \cdots (D^{ -}_{z_\ell}F)^2 {\rm d}\eta^{(\ell)}(z_1,...,z_\ell)\right]\\
			\!\!\!\!&= & \!\! \!\! \frac{1}{4^\ell} \E\left[ \int_{\mathbb{X}^\ell} h(z_1,...,z_\ell) (D_{z_1}F)^2\cdots (D_{z_\ell}F)^2 {\rm d}\lambda^{(\ell)}(z_1,...,z_\ell)\right]\label{e:rn2}
			\\
			\!\!\!\!&=& \!\! \!\! \frac{1}{4^\ell} \E\left[ \int_{\mathbb{X}^\ell} h(z_1,...,z_\ell) (D_{z_1}F)^2 \cdots (D_{z_\ell}F)^2 {\rm d}\eta_0^{(\ell)}(z_1,...,z_\ell)\right],\notag
		\end{eqnarray}
where $F\equiv F(\eta)$ and $\eta_0$ is an independent copy of $\eta$. Moreover, \eqref{e:rn2} uniquely characterize the law of $\mathcal{P}_F$ as a random element taking values in ${\bf N}$.
\end{prop}

\begin{remark}{\rm Relations \eqref{e:rn} and \eqref{e:rn2} in the case $\ell=1$ imply that, when $F$ is a Boolean mapping and measurable $B \subseteq \X$ is such that $\int_{B} \E[(D_zF)^2] \, \lambda({\rm d}z) <\infty$, then
\begin{equation}\label{e:eqint}
\E[|\gamma_F(B)|] = 4\E[|\mathcal{P}_F(B)|] = \int_{B} \E[(D_zF)^2] \, \lambda({\rm d}z) . 
\end{equation}
In particular, the above identity holds for all bounded measurable sets $B$, which yields that the intensity function of the spectral point process is four times that of the pivotal point process.

By a direct inspection of formulae \eqref{e:rn} and \eqref{e:rn2}, one sees that, in general, it is not possible to establish analogous relations for factorial moment measures of order $\ell\geq 2$. This fact has to be contrasted with the statement of \cite[Cor.\ 9.4]{GS} (see also \cite[Lem.\ 5.4]{Garban2011}), proving that the marginals of order two of the discrete spectral sample and of the discrete pivotal set discussed in the introduction coincide {(though marginals of order three and above, they also do not coincide, see e.g.\ \cite[p.\ 37]{GPS10})}.
}
\end{remark}
 \subsection{Annealed and projected spectral samples}\label{sec:comparison}

 Our aim in this section is to establish an explicit relationship between the general notion of spectral point process introduced in Definition~\ref{def:spproc}, and the {\bf annealed spectral sample} introduced in \cite{Vann2021} -- see, in particular, Section 2.1 therein.  As a by-product, we will also study the connection between the {pivotal point process} in Definition~\ref{def:pivproc} and the \textbf{quenched pivotal point process}, introduced below in Definition~\ref{def:qpiv}.  The latter emerges naturally in connection to the annealed spectral sample, while the pivotal point process is its natural counterpart when considering the spectral point process.
 
	\subsubsection{Factorial moment measures}
 
	   In order to accomplish the above mentioned tasks, we specialize the general setting outlined in Sections \ref{ss:prem} and \ref{sec:Spec} to the following framework.
	
	\begin{itemize}
		
		\item[--] For a fixed $d\geq 1$, let $\mathbb{X} = \R^d \times \{-1,1\}$. We endow $\mathbb{X}$ with the Borel $\sigma$-field $\mathcal{X}$ associated with the metric $d((x,\eps), (x', \eps')) = \|x-x'\|_{\R^d} + \mathds{1}_{\{\eps\eps' = -1\}}$. As before, we write ${\bf N}(\mathbb{X})$ (resp.\ ${\bf N}_s(\mathbb{X})$) to indicate the class of all locally finite integer-valued positive measures (resp.\ locally finite simple positive measures) on $(\mathbb{X}, \mathcal{X})$. Depending on notational convenience, we will write equivalently $(x, \eps)$ or $(x, \eps_x)$ to denote a generic element of $\mathbb{X}$. In order to simplify the discussion, we will often identify a measure $\mu \in {\bf N}_s(\mathbb{X})$ with its support. Accordingly, generic elements of ${\bf N}_s(\mathbb{X})$ will be written in set form $\mu = \{(x, \eps_x)\}$, meaning that $\mu$ is identified with the simple measure whose support is given by the pairs $(x, \eps_x)$ featured in the set. Finally, if $S\subset \R^d$ is a finite set, we write
  \begin{equation}\label{e:simpleS}
  \mu_S := \sum_{x\in S}\delta_x\in {\bf N}_s(\R^d).
  \end{equation}
		
		\item[--] The symbol $\eta$ denotes a Poisson point process on $(\mathbb{X}, \mathcal{X})$ with non-atomic intensity ${\rm d} x\,  p({\rm d}\eps)$, where $p({\rm d}\eps) = \frac12 ( \delta_1({\rm d}\eps) + \delta_{-1} ({\rm d}\eps))$.\footnote{Our results trivially extend to the case of an intensity of the type $c \,  {\rm d} x\,  p({\rm d}\eps)$ for arbitrary $c>0$.} Given the discussion contained in Section \ref{ss:prem}, one has that $\eta$ is an ${\bf N}_s(\mathbb{X})$-valued random element. 
		
		\item[--] Given $\mu = \{(x, \eps_x)\}\in {\bf N}_s(\mathbb{X})$, we write 
		\begin{equation}\label{e:stella}
			\mu^{pr} := \{x\in \R^d : (x, \eps_x)\in \mu \mbox{  for some $\eps_x\in \{-1,1\}$} \},
		\end{equation}
		and we identify $\mu^{pr}$ with the unique element of ${\bf N}_s(\R^d)$ whose support is given by those $x\in \R^d$ verifying the property expressed on the right-hand side of \eqref{e:stella}. Plainly, the mapping ${\bf N}_s(\mathbb{X})\to {\bf N}_s(\R^d): \mu\mapsto \mu^{pr}$ is measurable and $\eta^{pr}$ is a well-defined random element with values in ${\bf N}_s(\R^d)$, having the law of a Poisson point process with intensity $\, {\rm d} x$. In resonance with \eqref{e:stella}, the notation $\eta^{pr}\cap A$ and $\eta^{pr}_A$ will be used interchangeably in order to indicate the random measure $\eta^{pr}(A\cap \bullet)$. 

      \item[--]  Note that $\eta$ is an independently marked Poisson process i.e., to obtain a point process distributed as $\eta$, we can take a Poisson process distributed as $\eta^{pr}$ and attach independent symmetric Rademacher random variables to each of its points. 
			
	\end{itemize}

	Fix a countable set $\mathcal{J}\subset \R^d$ and consider a mapping $G : \{-1,1\}^\mathcal{J} \to \R : \{\eps_x : x\in \mathcal{J}\}\mapsto G(\eps_x : x\in \mathcal{J})$ that only depends on a finite subset of coordinates $\{\eps_x : x\in \tilde{\mathcal{J}}\} \subseteq \{\eps_x : x\in \mathcal{J}\}$, where $\tilde{\mathcal{J}}$ is a finite subset of $\mathcal{J}$. Fourier analysis on the discrete cube (or, more generally, the theory of {\bf Hoeffding-ANOVA decompositions} -- see e.g.\ \cite[Ch.\ 4]{GS}, \cite{efronstein} and \cite[Section 5.1.5]{serfling}) implies that $G$ admits the orthogonal decomposition
	
 \begin{equation}\label{e:anova}
		G(\eps_x : x\in \mathcal{J}) = \sum_{T\subseteq \mathcal{J}, \; T\,\,  \mbox{\small finite}} \hat{G}(T) \prod_{x\in T}  \eps_x = \sum_{T\subseteq \tilde{\mathcal{J}}} \hat{G}(T) \prod_{x\in T}  \eps_x,
	\end{equation}
	where the first sum is over all finite subsets $T$ of $\mathcal{J}$ and, for each such $T$, the coefficient $\hat{G}(T)$ is given by
	\begin{eqnarray}\label{e:walsh0}
		\hat{G}(T) &=& \int_{\{-1,1\}^{\mathcal{J}}} \left\{ G(a_x : x\in \mathcal{J}) \prod_{y\in T} a_y\right\} \,\,\prod_{t\in \mathcal{J}} p({\rm d} a_t),
	\end{eqnarray}
 with $\prod_\emptyset \equiv 1$; note here that since $G$ only depends on the coordinates corresponding to $x \in \tilde{\mathcal{J}}$, one has $\hat{G}(T)=0$, whenever $T$ is not a subset of $\tilde{\mathcal{J}}$. The orthogonality claimed above resides in the fact that, for {finite} $S\neq T\subseteq \mathcal{J}$, one has that
 $$
 \int_{\{-1,1\}^{\mathcal{J}}} \left\{\prod_{x\in S} \eps_x \prod_{y\in T} \eps_y\right\} \, \prod_{t\in \mathcal{J}} p({\rm d} \eps_t) = 0.
 $$
	One can also check that for finite $T \subseteq \mathcal{J}$,
	\begin{equation}\label{e:quietplace}
		\hat{G}(T) \prod_{x\in T}  \eps_x=\left\{ \prod_{t\in T} (I-Q_t) \prod_{y \in \mathcal{J}\backslash T} Q_{y}\right\}\, G(\eps_x : x\in \mathcal{J}),
\end{equation}
where $I$ is the identity operator and $Q_t$ is the operator defined by
	$$
	Q_tG(\eps_x : x\in \mathcal{J}) = \frac12 G(\eps_x : x\in \mathcal{J}) + \frac12 G({\eps}^{t}_x : x\in \mathcal{J}),
	$$
	where ${\eps}^{t}$ is the element of $\{-1,1\}^\mathcal{J}$ such that $\eps_t^t = - \eps_t$, and $\eps_x^t = \eps_x$ for $x\neq t$ (in other words, the action of $Q_t$ consists in taking expectation with respect to the coordinate $\eps_t$). We record two consequences of \eqref{e:quietplace}: if $\mathcal{J}$ is finite, then (a) one has that
 \begin{equation}\label{e:lupita}
     \hat{G}(\mathcal{J})\prod_{x\in \mathcal{J}}  \eps_x = \sum_{T\subseteq \mathcal{J}} (-1)^{|\mathcal{J}| - |T|} \, \mathbb{E}\big[ G(\eps_x : x\in \mathcal{J})\, |\, \epsilon_x : x\in T\big] ,
 \end{equation}
and (b) if $G({\eps}_x : x\in \mathcal{J}) = G({\eps}_x : x\in \mathcal{J}_0)$ for some strict subset $\mathcal{J}_0\subset \mathcal{J}$, then 
\begin{equation}\label{e:quinn}
\hat{G}(\mathcal{J}_1)\prod_{x\in \mathcal{J}_1}  \eps_x = 0,  
\end{equation}
for all $\mathcal{J}_1\subseteq \mathcal{J}$ such that $\mathcal{J}_0$ is a strict subset of $\mathcal{J}_1$.

\begin{definition}\label{d:finsupp}{\rm
We write $\SA$ to indicate the class of those finitely supported measurable functions $F : {\bf N}(\mathbb{X}) \to \R$ such that there exists a measurable mapping $\varphi : {\bf N}_s(\mathbb{X}) \to {\bf N}_s(\mathbb{R}^d) : \mu \mapsto \varphi(\mu)$ with the following properties: (i) for all $\mu \in {\bf N}_s(\mathbb{X})$, $\varphi(\mu)\subseteq \mu^{pr}$ (see \eqref{e:stella}) has finite support, (ii) if $\mu, \nu \in {\bf N}_s(\mathbb{X}) $ are such that $\mu^{pr} = \nu^{pr}$, then $\varphi(\mu) = \varphi(\nu)$, and (iii) for all $\mu\in {\bf N}_s(\mathbb{X})$, $F(\mu) = F\big((x, \varepsilon_x)\in \mu : x\in \varphi(\mu)\big) $. To make the notation less implicit, we shall sometimes write $\varphi(\mu) =: \tilde{\mu}^{pr} $, that is, $\tilde{\mu}^{pr}$ indicates the finite subset of ${\mu}^{pr}$ determining the value of a finitely supported mapping $F(\mu)$. }
\end{definition}

\medskip 
 
 Relation \eqref{e:walsh0} implies in particular that the restriction of a mapping $F\in\SA$ to ${\bf N}_s(\mathbb{X})$ admits the following representation: for all $\mu= \{(x,\eps_x)\}\in {\bf N}_s(\bbX)$,
	\begin{equation*}\label{e:walsh1}
		F(\mu) = F( \{(x, \eps_x) : x\in \tilde{\mu}^{pr}\})  = \sum_{S\subseteq \tilde\mu^{pr}} \hat{F}(S; \mu^{pr}) \prod_{x\in S} \eps_x,
	\end{equation*}
	with $\mu^{pr}$ and $\tilde{\mu}^{pr}$ defined according to \eqref{e:stella} and Definition \ref{d:finsupp}, respectively, and
	\begin{equation*}\label{e:walsh2}
		\hat{F}(S; \mu^{pr}) := \int_{\{-1,1\}^{|\mu^{pr}|}} \left\{ F(\{x, \eps_x\}) \prod_{z\in S}\eps_z \right\}\prod_{x\in \mu^{pr}} p(\dd\eps_x).
	\end{equation*}
	
	\begin{remark}{\rm 
			When $\eta = \{ (x, \eps_x)\}$ is the Poisson measure introduced above, then
			$$
			\hat{F}(S; \eta^{pr}) = \E\left[F(\eta)\prod_{x\in S}\eps_x \, \Big|\, \eta^{pr}\right], \quad \text{finite } S\subseteq \eta^{pr}.
			$$
		}
	\end{remark}

	\begin{definition}{\rm (See \cite[Definition 2.2]{Vann2021}).
        \label{d:annealed}
 Let the above notation prevail, %fix a bounded set $A\subset \R^d$, 
 and consider a non-constant mapping $F\in \SA$. The {\bf annealed spectral sample} associated with $F$, written $\gamma_F^{an}$, is the point process on $\R^d$ characterized by the following relation: for all $\mathcal{C}\in \mathcal{N}(\R^d)$,
			\begin{equation}\label{e:gaman}
				\mathbb{P}(\gamma_F^{an} \in \mathcal{C}) = \frac{1}{\E[F^2]} \E\left\{ \sum_{S\subseteq \eta^{pr}, \; S\; \text{finite}} \mathds{1}_{\{\mu_S\in \mathcal{C}\}} \hat{F}(S; \eta^{pr})^2\right\},
			\end{equation}
   where we have adopted the notation \eqref{e:simpleS}.
		}
	\end{definition}
	
	For the rest of the section, we fix 
 a non-constant mapping $F \in \SA$, and we denote by $\gamma = \gamma_F$ and $\gamma^{an}= \gamma^{an}_F$, respectively, the spectral sample and the annealed spectral sample associated with $F$ (to simplify the notation, the dependence on $F$ is removed). Note that $\gamma$ is a random element with values in ${\bf N}_s(\bbX)$, whereas $\gamma^{an}$ has values in ${\bf N}_s(\R^d)$. 
	
	\begin{definition}\label{d:gammastella} {\rm For $\eta$ a Poisson process on $\mathbb{X} = \R^d \times\{-1,1\}$ as above, and $F = F(\eta)$ non-zero and square-integrable  (not necessarily belonging to {\bf S}), we define the {\bf projected spectral sample} associated with $F$ to be the ${\bf N}_s(\R^d)$-valued random element $ \gampr = \gampr_F $, where we have adopted the notation \eqref{e:stella}.}
	\end{definition}
		The next definition is a direct emanation of \cite[Definition 2.2]{Vann2021}.
	
 \begin{definition}\label{def:qpiv}{\rm For $\eta$ a Poisson process on $\mathbb{X}$ as above, consider a random variable $F = F(\eta)$ with values in $\{-1,1\}$ (not necessarily belonging to {\bf S}). The {\bf quenched Pivotal point process} associated with $F$ is defined as the random measure
		$$
		{\mathcal P}_F^{q} (B)={\mathcal P}^{q} (B):= \sum_{x \in \eta^{pr} } \mathbf{1}_{\{F(\eta) \neq F(\eta - \delta_{(x,\eps_x)} +\delta_{(x,-\eps_x)})\}} \, \mathds{1}_{\{x\in B \}}, \quad B\in \mathcal{B}(\R^d).
		$$
		In other words, ${\mathcal P}^{q}$ is the random element of ${\bf N}_s(\R^d)$ whose support is given by those $x\in \eta^{pr}$ such that changing $\varepsilon_x$ into $-\varepsilon_x$ yields a change in the sign of $F$.
  }
	\end{definition}
	
 \begin{remark}\label{r:kgammastella}{\rm  Consider a non-zero and bounded random variable $F = F(\eta)$. An application of Proposition \ref{p:ell} and of Remark \ref{e:bagatelle}-(ii) shows that the projected spectral sample $\gamma^{pr}$ is the point process on $\R^d$ whose $k$-th factorial moment measure (for all $k\geq 1$) is absolutely continuous with respect to the Lebesgue measure, with density 
			\begin{eqnarray}\label{e:kgammastella}
				\Phi_k(x_1,...,x_k) =  \frac{1}{\E[F^2]} \int_{\{-1,1\}^k} \!\!\mathbb{E}[ ( D_{x_1, \eps_1}\cdots D_{x_k, \eps_k} F)^2] \, p(\dd\eps_1)\cdots p(\dd\eps_k),
			\end{eqnarray}
It is easily seen that, if $F(\eta) = F(\eta^{pr})$ (that is, $F$ does not depend on the family of marks $\{\epsilon_x\}$), then \eqref{e:kgammastella} coincides with \eqref{e:rn}, as applied to the Poisson measure $\eta^{pr}$. 		  } 
	\end{remark}
	
	Consider $F(\eta)$ with $F \in \SA$. Given $k\geq 1$, $\{x_1,...,x_k\} = \mathcal{J}\subset \R^d$, and $T\subseteq \mathcal{J}$, we write
	$$
	\widehat{D_{x_1}\cdots D_{x_k} F} (T)
	$$
	to denote the random variable obtained by applying \eqref{e:walsh0} to the mapping $G$ on $\{-1,1\}^{\mathcal{J}}$ given by
	$$
	\{\eps_x: x\in \mathcal{J}\}\mapsto G(\eps_x: x\in \mathcal{J}) := D_{x_1, \eps_1 }\cdots D_{x_k, \eps_k} F(\eta),
	$$
	where $\eps_i:=\eps_{x_i}$, $i=1,...,k$. The following statement clarifies the relation between the factorial moment measures of $\gampr$ and $\gaman$ for bounded elements in $\SA$ -- see Definition \ref{d:finsupp}.
 
	\begin{theorem}\label{t:kgampran} Let $F\in \SA$ be bounded. Then, for $k\geq 1$, the $k$-th factorial moment measure of the annealed spectral sample $\gamma^{an}$ is absolutely continuous with respect to the Lebesgue measure, with a density $\Psi_k$ given by the following relation: for $x_1,...,x_k$ pairwise distinct,
		\begin{equation}\label{e:psik}
			\Psi_k(x_1,...,x_k) =\frac{1}{\E[F^2]} \E\left[ (\widehat{D_{x_1}\cdots D_{x_k} F} (\mathcal{J}))^2\right],
		\end{equation}
where $\mathcal{J} = \{x_1,...,x_k\}$. Moreover, for $k\geq 1$, the $k$-point correlation function $\Phi_k$ for projected spectral sample $\gamma^{pr}$ defined in \eqref{e:kgammastella} admits the following representation:
		\begin{equation}\label{e:phik}
			\Phi_k(x_1,...,x_k) =\frac{1}{\E[F^2]}\sum_{T\subseteq \mathcal{J}} \E\left[ (\widehat{D_{x_1}\cdots D_{x_k} F} (T))^2\right].
		\end{equation}
		In particular, $\Psi_k\leq \Phi_k$ a.e.\ $\dd x_1\cdots \dd x_k$, for all $k\geq 1$. Finally, the fact that the right-hand side of \eqref{e:psik} (resp.\ \eqref{e:phik}) is the density of the $k$-th factorial moment measure uniquely characterizes the distribution of $\gaman$ (resp.\ $\gamma^{pr}$) as a random element with values in ${\bf N}(\R^d)$.
	\end{theorem}

	\begin{proof} By homogeneity and without loss of generality, we can assume that $\E[F^2]=1$. Fix $k\geq 1$ and consider a symmetric and bounded mapping $h : (\R^d)^k\to \R_+$ which is supported on some compact hyperrectangle $C^k$. Given a set $Y$, we will follow \cite{Vann2021} and use the notation $S\subseteq_f Y$ to indicate that $S$ is a {\it finite} subset of $Y$. Then, standard combinatorial considerations based on the definition \eqref{e:gaman} yield
\begin{eqnarray*}
			&&\mathbb{E} \left[ \int_{C^k} h\, \dd (\gaman)^{(k)}\right] = k! \E\left[ \sum_{S\subseteq_f \, \eta^{pr} } \hat{F}(S; \eta^{pr})^2 \sum_{T\subseteq S : |T|=k} h(T)\right] \\
			&&\quad = k! \E\left[ \sum_{T\subseteq \eta^{pr} : |T|=k} h(T) \sum_{S\, \mbox{ \small finite } : \,  T\subseteq S} \hat{F}(S; \eta^{pr})^2\right]\\
			&&\quad = \mathbb{E}\left[ \int_{C^k} h(x_1,...,x_k) \left( \sum_{S\subseteq_f\, \eta^{pr} : \{x_1,...,x_k\}\subseteq S} \hat{F}(S; \eta^{pr})^2\right) \, (\eta^{pr})^{(k)}(\dd x_1,...,\dd x_k)\right].
		\end{eqnarray*}

		%\begin{eqnarray*}
			%\mathbb{E} \left[ \sum^{\neq}_{x_1,..,x_k\in \gaman} h(x_1,...,x_k)\right] &=& k! \E\left[ \sum_{S\subseteq \eta^{pr} } \hat{F}(S; \eta^{pr})^2 \sum_{T\subseteq S : |T|=k} h(T)\right] \\
			%&=& k! \E\left[ \sum_{T\subseteq \eta^{pr} : |T|=k} h(T) \sum_{S: T\subseteq S} \hat{F}(S; \eta^{pr})^2\right]\\
			%&=& \mathbb{E}\left[ \sum^{\neq}_{x_1,...,x_k\in \eta^{pr}} h(x_1,...,x_k) \sum_{S\subset \eta^{pr} : \{x_1,...,x_k\}\subset S} \hat{F}(S; \eta^{pr})^2\right].
		%\end{eqnarray*}
		\noindent We now apply the multivariate Mecke formula \eqref{e:mecke} to the last term in the previous chain of equations to infer that
		\begin{eqnarray*}
			 &&\mathbb{E} \left[ \int_{C^k} h\, \dd (\gaman)^{(k)}\right] \\ &&= \E\left[ \int_{C^k} h(x_1,...,x_k)\sum_{S\subseteq_f\, \eta^{pr}} \hat{F}\left(S\cup\{x_1,...,x_k\} ; \eta^{pr}\cup\{x_1,...,x_k\}\right)^2 \dd x_1\cdots \dd x_k \right],
		\end{eqnarray*}
		and relation \eqref{e:psik} is obtained once we show that 
		\begin{eqnarray}\label{e:combinatorial}
&& \mathbb{E}\left[\sum_{S\subseteq_f\, \eta^{pr}} \hat{F}\left(S\cup\{x_1,...,x_k\} ; \eta^{pr}\cup\{x_1,...,x_k\}\right)^2\right] \\ \notag
&&\quad\quad\quad\quad\quad\quad = \E\left[ \widehat{D_{x_1}\cdots D_{x_k} F} (\mathcal{J})^2\right], \,\mbox{when}\,\, \mathcal{J} = \{x_1,...,x_k\}. \end{eqnarray}
To prove such an identity, we start by observing that, by a conditioning argument, the left-hand side of \eqref{e:combinatorial} equals
$$
\mathbb{E}\left[\left( \sum_{S\subseteq_f\, \eta^{pr}} \hat{F}\left(S\cup\{x_1,...,x_k\} ; \eta^{pr}\cup\{x_1,...,x_k\}\right) \prod_{x\in S\cup\{x_1,...x_k\}}\varepsilon_x \right)^2\right],
$$
where $\{ \varepsilon_x : x\in \eta^{pr}\cup \{x_1,..., x_k\} \}$ is a collection of i.i.d.\ symmetric random signs such that, for $x\in \eta^{pr}$, $\varepsilon_x$ is the random mark associated with $x$ through the Poisson measure $\eta$. Now, as a consequence of \eqref{e:anova} and because $F\in {\bf S}$,
\begin{equation*}\label{e:consanova}
F \left(\eta + \sum_{\ell=1}^k \delta_{(x_\ell, \varepsilon_{x_\ell})} \right) =\!\!\!\! \sum_{S\subseteq_f \, \eta^{pr}\cup\{x_1,...,x_k\}} \hat{F}(S ; \eta^{pr}\cup\{x_1,...,x_k\}) \prod_{x\in S} \varepsilon_x =: \!\!\!\!\sum_{S\subseteq_f \, \eta^{pr}\cup\{x_1,...,x_k\}} U(S) ,
\end{equation*}
where $U(S) :=\hat{F}(S ; \eta^{pr}\cup\{x_1,...,x_k\}) \prod_{x\in S} \varepsilon_x $. The last identity yields that, for all $T\subseteq \{x_1,...,x_k\}$,
\begin{equation*}\label{e:ie}
H(T) := \mathbb{E}\left[ F \left(\eta + \sum_{\ell=1}^k \delta_{(x_\ell, \varepsilon_{x_\ell})} \right)\, \Big| \, \eta+\sum_{x\in T}\delta_{(x, \varepsilon_x)}\right]=\sum_{S\subseteq_f \, \eta^{pr}\cup T} U(S),
\end{equation*}
(we stress that $H(T)$ is a function of $\eta + \sum_{x\in T} \delta_{(x, \varepsilon_{x})}$), which in turn implies that
\begin{eqnarray*}
\sum_{T\subseteq \{x_1,...x_k\}} (-1)^{k - |T|} H(T) &= &\sum_{T\subseteq \{x_1,...x_k\}} (-1)^{k - |T|} \sum_{S\subseteq_f \, \eta^{pr}\cup T} U(S) \\ &=&
\sum_{S\subseteq_f \eta^{pr}\cup \{x_1,...,x_k\}} U(S) \sum_{S\backslash \eta^{pr}\subseteq T\subseteq \{x_1,...,x_k\}} (-1)^{k - |T|} \\
&=& \sum_{\{x_1,...,x_k\}\subseteq S\subseteq_f \eta^{pr}\cup \{x_1,...,x_k\}} U(S).
\end{eqnarray*}
We have therefore proved that the quantity $\sum_{T\subseteq \{x_1,...x_k\}} (-1)^{k - |T|} H(T)$ equals
$$
\sum_{S\subseteq_f\, \eta^{pr}} \hat{F}\left(S\cup\{x_1,...,x_k\} ; \eta^{pr}\cup\{x_1,...,x_k\}\right) \prod_{x\in S\cup\{x_1,...x_k\}}\varepsilon_x,
$$
and we eventually deduce the desired formula \eqref{e:combinatorial} from the chain of identities 
\begin{equation}\label{e:rr}
\sum_{T\subseteq \{x_1,...x_k\}} (-1)^{k - |T|} H(T) = \sum_{T\subseteq \{x_1,...x_k\}} (-1)^{k - |T|} H'(T) = \widehat{D_{x_1}\cdots D_{x_k} F} (\mathcal{J}) \prod_{j \in \mathcal{J}}\varepsilon_{x_j},  
\end{equation}
with $\mathcal{J} = \{x_1,...,x_k\}$ and
$$
H'(T) := \mathbb{E}\left[ D_{x_1,\varepsilon_{x_1}} \cdots D_{x_k,\varepsilon_{x_k}} F \left(\eta \right)\, \Big| \, \eta+\sum_{x\in T}\delta_{(x, \varepsilon_x)}\right].
$$
For clarity, we observe that the first equality in \eqref{e:rr} follows from \eqref{e:addoneincexc} and \eqref{e:quinn}, while the second one can be inferred from \eqref{e:lupita}. Formula \eqref{e:phik} is a direct consequence of \eqref{e:anova} and \eqref{e:kgammastella}, as well as of the orthogonality of distinct components of the Hoeffding decomposition. The final assertion is obtained by applying arguments analogous to those developed in Remark \ref{e:bagatelle}-(ii).
	\end{proof}

	\subsubsection{Comparison of correlation functions}
	\label{ss:comparison_samples}
	
    In this section, we study in more detail the comparison of the first two correlation functions $\Phi_i,\Psi_i, i=1,2$ of the projected and annealed samples, respectively, associated with a function $F\in \SA$ as in the statement of Theorem \ref{t:kgampran} taking values in $\{-1,1\}$ (that is, $F$ is Boolean); observe, in particular, that the content of Remark \ref{r:kgammastella} applies to $F$.
 \begin{enumerate}
 \item[(i)] {\it Factorial moment measures of order 1.}  Re-writing the formulas for $\Psi_1$ and $\Phi_1$ from \eqref{e:psik} and \eqref{e:phik} respectively, we have that
	\bea
	\Psi_1(x) & := &  \E\Big[ \E[ \eps D_{x,\eps}F(\eta) \mid \eta ]^2 \Big],  \nonumber \\
	\label{e:onepointcorrelation} \Phi_1(x) & := &  \E\Big[ \E[ \eps D_{x,\eps}F(\eta) \mid \eta ]^2 \Big] +  \E\Big[ \E[D_{x,\eps}F(\eta) \mid \eta ]^2 \Big],
	\eea
	where $\eps$ is a Rademacher random variable (i.e.,  $\{-1,+1\}$-valued) independent of $\eta$.  This yields that $\Psi_1(x) \leq \Phi_1(x)$ as also indicated by Theorem \ref{t:kgampran}. More can be said if $F$ is {\bf increasing}, that is, if for all $x \in \R^d$,  
	\begin{equation}\label{e:fmon}
		F(\eta + \delta_{(x,0)}) \le F(\eta)\le F(\eta + \delta_{(x,1)}).
	\end{equation}
	In this case, one indeed has that
	\begin{equation}
		\label{e:onepointcomparison}
		\Phi_1(x)= 2 \Psi_1(x).
	\end{equation}
	To see this, first note that by \eqref{e:fmon}, we have that given $\eta$, one of the two quantities $D_{x,1}F(\eta)$ and $D_{x,-1}F(\eta)$ has to be zero. Hence,
	\begin{eqnarray*}
		\E[ \eps D_{x,\eps}F(\eta) \mid \eta ]^2&=& \frac{1}{4} \left[D_{x,1}F(\eta) - D_{x,-1}F(\eta)\right]^2 \\
		&=& \frac{1}{4} \left[D_{x,1}F(\eta)^2 + D_{x,-1}F(\eta)^2\right]
		= \E[D_{x,\eps}F(\eta) \mid \eta ]^2,
	\end{eqnarray*}
	and using \eqref{e:onepointcorrelation} we infer \eqref{e:onepointcomparison}.

	\item[(ii)] {\it Factorial moment measures of order 2.} There does not seem to be any obvious equality between second-order correlation functions, even in the case of an increasing $F$. Some partial result can be obtained as follows: write $x^{1}$ and $x^{-1}$ for $\delta_{(x,1)}$ and $\delta_{(x,-1)}$, respectively and, for $x \neq y \in \R^d$, define the event
	\begin{eqnarray*} 
			E{(x,y)} &=&\{ F(\eta+x^1+y^1)=1=-F(\eta+x^{-1}+y^{-1}),\\
			&& F(\eta+x^1+y^{-1})=1=-F(\eta+x^{-1}+y^{1}),\\
			&& F(\eta+y^1)=1=-F(\eta+y^{-1}) \, \}.
	\end{eqnarray*}
Then, one has the following statement, whose proof can be found in Appendix \ref{s:pfz1max}.
	\begin{lemma}\label{lem:z1max}
		Let $F$ be as above and increasing. For $x \neq y \in \R^d$, we have
		\begin{equation*}
			\label{e:twopointcomparison}
			\Phi_2(x,y)\le 9 \Psi_2(x,y) + 48 [\P(E(x,y) \cup E(y,x))].
		\end{equation*}
	\end{lemma}
\end{enumerate}
 
%Typically, (e.g., in the case of critical Voronoi percolation), the integral of $\P(B(x,y) \cup B(y,x))]$ on $W_{L/2}^2$ is small. 

%
	
	We conclude the section by observing that, unlike the case of spectral point processes,  the factorial measures of pivotal processes are related by a simple identity.
	\begin{prop}
		\label{p:comp_Pivotal}
		Let $F\in \SA$ be a Boolean function of a Poisson process $\eta$ on $\X$ and let ${\mathcal P}_F$ and ${\mathcal P}^q_F$ be the pivotal and quenched pivotal processes as in Definitions~\ref{def:pivproc} and \ref{def:qpiv} respectively. Then, for all $k\geq 1$ and measurable $C_1,...,C_k \subset \R^d$, one has that
  $$
   \E \left[(\mathcal{P})^{k}((C_1\times \{-1,1\} )\times\cdots\times (C_k  \times \{-1,1\})) \right] = \frac{1}{2^k} \E \left[(\mathcal{P}^q)^{k}(C_1 \times\cdots C_k)\right].
  $$

  %For $B \subset D$ and $k \geq 1$, 
		
%		\begin{multline*}
%			\E\left[ \sum_{x_{i_1}, ..., x_{i_k}\in {\mathcal P}_F}^{\neq} \mathds{1}(x_{i_1},..., x_{i_k} \in   B \times \{-1,+1\})\right] = \E\left[{\mathcal P}_F \left((B  \times \{-1,+1\})^{(k)}\right)\right]\\  
%			= \frac{1}{2^k} \E[{\mathcal P}^q_F(B^{(k)})]=\frac{1}{2^k}\E\left[ \sum_{y_{i_1}, ..., y_{i_k}\in {\mathcal P}_F^q}^{\neq} \mathds{1}(y_{i_1},..., y_{i_k} \in   B)\right].
%		\end{multline*}
		%
	\end{prop}

 \section{Applications to continuum percolation}
\label{s:sharp_Nst_cont_perc}
In this section, we prove sharp noise instability for crossings in critical Poisson Boolean model (Theorem~\ref{t:sharpNS_Boolean}) and sharp noise sensitivity (as well as sharp noise instability) for crossings in the critical Poisson Voronoi percolation model (Theorem \ref{t:sharpNS_voronoi}). 
 
First, in Section~\ref{ss.gen_sns_cont}, we will provide four geometric conditions \ref{A1} -- \ref{A3} on a general static continuum percolation model, which imply sharp noise instability as defined at \eqref{eq:sNStab} for the model, when considered under the OU dynamics. We will then apply this result to the Boolean percolation model in Section~\ref{ss:proofBoolean}. We will check condition \ref{A2} (quasi-multiplicativity as in Theorem \ref{t:quasimPBM}) in Section \ref{sec:quasim_boolean} and conditions \ref{A1}, \ref{A3} and \ref{A4} in Section \ref{ss:proofBoolean} to conclude the proof of Theorem~\ref{t:sharpNS_Boolean}. Verifying these conditions uses the explicit representation of first and second order correlations of the spectral and pivotal processes from Section \ref{s:specpp}.
 
In Section \ref{s:sharpNSvor} we prove sharp noise sensitivity and sharp noise instability for the critical Poisson-Voronoi percolation model. We achieve this by relating our estimates for spectral point processes with those for annealed spectral samples obtained in \cite{Vann2021}, and also by using the sharp noise sensitivity under the frozen dynamics established therein. This will utilize the results contained in Section \ref{ss:comparison_samples}, as well as the covariance inequality in Proposition \ref{p:covcomp}, which was suggested to us by H.\ Vanneuville.

We will adopt the usual asymptotic notation for comparing the behaviour of functions at infinity. For two functions $f,g:[0,\infty) \to \R$, we write $f(x) \asymp g(x)$ if there exists constants $C_1,C_2>0$ such that $C_1 \le \liminf_{x \to \infty} f(x)/g(x) \le \limsup_{x \to \infty} f(x)/g(x) \le C_2$. Also, we write $f(x) \gtrsim g(x)$ or $f(x)\lesssim g(x)$ when the corresponding one-sided inequalities hold respectively.

\subsection{Criteria for sharp noise instability in continuum percolation}
\label{ss.gen_sns_cont}
Let $\eta$ be a Poisson process on the space $\X = \R^2 \times \M$ with intensity measure $\lambda (\md \tx) := c \, \md x \times \nu(\md a)$ where $\nu$ is a probability on a measurable mark space $(\M, \mathcal{M})$, and $\tx=(x,a)$ denotes a generic element in $\X$. Without loss of generality, from now on we may assume that $(\Omega, \mathcal{F}) = ({\bf N}_s(\mathbb{X}), \mathcal{N})$ and that the probability $\mathbb{P}$ coincides with the law of $\eta$; see Section \ref{ss:prem}.

For $\eta$ as above and $t>0$, the perturbed version $\eta^t$ of $\eta$ after time $t$ under the OU dynamics can be defined as follows: delete each point of $\eta$ independently with probability $(1 - e^{-t})$ to obtain a thinned process $\eta_1$, and then add an independent Poisson process $\eta_2$ with intensity $(1-e^{-t}) \, \lambda $, i.e.,
$\eta^t =_d \eta_1 + \eta_2$ with $\eta_1$ and $\eta_2$ independent. See e.g.\ \cite[Section 6.1]{LPY2021} for a formal definition of this Markov process.

Consider a continuum percolation model defined on the Poisson process $\eta$, i.e., a partition of the space $\R^2$ into occupied and vacant regions, determined by $\eta$. %by some mechanism, given the configuration $\eta$, we
More formally, we assume here that there is a measurable mapping that associates with $\eta$ a random closed set $\Occ(\eta) \subseteq \R^2$, called the {\bf occupied region}. We call its complement $\mathcal{V}(\eta) = \R^2 \setminus \Occ(\eta)$ the {\bf vacant region}. For more details on random sets including measurability issues, we refer the reader to \cite{LastPenrose17, Molchanov2017}.  We use standard notations for the arm events as in \cite{Vann2019,Vann2021}. For example, $A^x(r,R)$, $0<r<R<\infty$, denotes the annulus $x + [-R,R]^2 \setminus (-r,r)^2$ and $W_L: =  [-L,L]^2$ is the square window of side $2L$ centered at the origin. For $j \in \N$, we write $A_j^x(r,R)$ for the $j$-{\bf arm event} (corresponding to the presence of $j$ paths of alternating types -- 
occupied or vacant -- from $\partial W_r$ to $\partial W_R$) in the annulus $A^x(r,R)$; also, $\alpha_j(r,R) := \P(A_j^0(r,R))$, where $0$ denotes the origin in $\R^2$.  Furthermore, for $R\ge 1$, set $A_j^x(R) = A_j^x(1,R)$ and $\alpha_j(R) = \alpha_j(1,R)$. 
	
We are interested in the event of a left-right (L-R) crossing of the box $W_L =  [-L,L]^2$, $L \geq 1$ through the occupied region. We write $\widehat W_L:=W_L \times \M$. We set $f_L$ to be the $\pm 1$-indicator that there is such an L-R occupied crossing of $W_L$ in $\Occ(\eta) \cap W_L$.  Let $\gamma_L := \gamma_{f_L}$ and $\cP_L := \cP_{f_L}$ be respectively the spectral and pivotal point processes of $f_L$ as in Definitions~\ref{def:spproc} and \ref{def:pivproc}.  We now state four geometric conditions on the percolation model that ensure sharp noise instability of the crossing functional $f_L$, as defined at \eqref{eq:sNStab}.

\begin{itemize}
\item[\namedlabel{A1}{(\textbf{H1})}] \textbf{Lower-bound for $4$-arm probabilities}: There exists an absolute constant $\eps \in (0,2)$ such that for every $1 \le r \le R<\infty$,
		\begin{equation*}
			\alpha_4(r,R) \ge  \eps \left(\frac{r}{R}\right)^{2-\eps}.
		\end{equation*} 
  
\item[\namedlabel{A2}{(\textbf{H2})}] \textbf{Quasi-multiplicativity for $4$-arm probabilities}: there exists a constant $C \in [1,\infty)$ such that for all $1 \le r_1 \le r_2 \le r_3<\infty$,
		\begin{equation*}
			\frac{1}{C} \alpha_{4}\left(r_{1}, r_{3}\right) \leq \alpha_{4}\left(r_{1}, r_{2}\right) \alpha_{4} \left(r_{2}, r_{3}\right) \leq C \alpha_{4}\left(r_{1}, r_{3}\right).
		\end{equation*}

\item[\namedlabel{A4}{(\textbf{H3})}] \textbf{Upper-bound for second-difference operator}: Let $\eps \in (0,2)$ be as in \ref{A1}. There exists $\rho \in (0,\infty)$ such that for $x,y \in W_{\rho L}$ and $|x-y| \geq L^{\eps/3}$,
		\begin{equation*}
			\E[|D_{\tx}D_{\ty}f_L( \eta)|^2] \lesssim \alpha_4^2\left(\frac{|x-y|}{4}\right)
   + \text{O}(L^2 \alpha_4(L)^2).
		\end{equation*}

\item[\namedlabel{A3}{(\textbf{H4})}]\textbf{Estimates for pivotal sample}: Assume that $\E |\cP_L| \lesssim L^2\alpha_4(L)$ and $\E |\cP_L \cap \widehat W_{\rho L}| \gtrsim L^2 \alpha_4(L)$ for $\rho \in (0,\infty)$ as in \ref{A4}. 

Note that by \eqref{e:eqint} and Proposition~\ref{p:ell}, we can interchangeably use $\gamma_L$ (instead of $\cP_L$) in assumption \ref{A3} and rewrite it in terms of add-one cost operators as follows:
		\begin{equation*}
	    L^2\alpha_4(L) \lesssim  \int_{\widehat W_{\rho L}}\E |D_{\tx} f_L(\eta)|^2 \lambda(\md \tx)  \leq   \int_{\X} \E |D_{\tx} f_L(\eta)|^2 \lambda(\md \tx) \lesssim L^2\alpha_4(L).
		\end{equation*}
	\end{itemize}
The four conditions above, especially \ref{A1}, \ref{A2} and \ref{A3}, are important properties in their own rights of the percolation model. The quantity $\rho$ in \ref{A4} can be taken to be in the interval $(0,1)$ when the percolation model in consideration satisfies certain asymptotic independence properties. Typically, the above properties are expected to hold for many critical percolation models (where $f_L$ is non-degenerate asymptotically) and our upcoming proposition does imply asymptotic non-degeneracy of $f_L$; see the discussion below. Indeed, the following proposition shows that establishing \ref{A1} -- \ref{A3} yields sharp noise instability. Whether similar explicit criteria for sharp noise sensitivity exists is unclear.

\begin{prop}\label{p:sharpNStab}
Let $\Occ(\eta)$ be the occupied region in a percolation model defined on a Poisson process $\eta$ as above and assume that assumptions \ref{A1} -- \ref{A3} are in force. Then, $f_L$ exhibits sharp noise instability at time-scale $t_L = \frac{1}{L^2\alpha_4(L)}$, as $L \to \infty$.
 \end{prop}
The crucial ingredients for the proof of Proposition \ref{p:sharpNStab} are Lemma \ref{l:2ndmomspectravoronoi}, which relates the first and second moments of the size of the spectral process, and the Paley-Zygmund inequality. While an estimate for the first moment can be derived easily from Proposition \ref{p:ell} and the bounds in \ref{A3}, the more involved estimate for the second moment necessitates the four assumptions \ref{A1} -- \ref{A3}. The proof can be found in Appendix \ref{s:pfSharpNStab}. Note that, in particular, our proof yields the following concentration estimates for the size of the spectral sample of $f_L$: for any $\epsilon \in (0,1)$, Markov's inequality and assumption \ref{A3} yields
$$ \P\big( |\gamma_{L}| \geq \frac{1}{\epsilon}L^2\alpha_4(L) \big) \leq C \epsilon,$$
 for some constant $C>0$, while from the proof of Proposition \ref{p:sharpNStab}, it follows that there exist constants $c,c_0>0$ such that
 $$  \P\big( |\gamma_{L}| \geq c L^2\alpha_4(L) \big) \ge c_0.$$
 By \eqref{e:quantnoNStab} in Proposition \ref{p:noNStabgen}, these estimates imply that the crossing functional $f_L$ is asymptotically non-degenerate (see remark thereafter).

	\begin{remark}{\rm
		One often has that the crossing functional $f_L$ is \textbf{increasing} (see also Eq.\ \eqref{e:fmon}) as a function of $\eta$, i.e., adding (deleting) points in $\eta$ results in an increase (decrease) of the functional. If we assume that $f_L$ is increasing, then we can use the following observation to verify \ref{A4}: It is straightforward to check, using \eqref{e:fmon}, that $D_{\tx}D_{\ty}f_L(\eta) \neq 0$ implies that either $D_{\tx} f_L(\eta)$ or $D_{\tx} f_L(\eta + \delta_{\ty})$ is non-zero, and either $D_{\ty} f_L( \eta)$ or $D_{\ty} f_L(\eta + \delta_{\tx})$ is non-zero. Thus, we have
		\begin{multline}\label{e:4sum}
			\E[|D_{\tx}D_{\ty}f_L(\eta)|] \lesssim \P(D_{\tx} f_L(\eta) \neq 0, D_{\ty} f_L(\eta) \neq 0) \\+ \P(D_{\tx} f_L(\eta) \neq 0, D_{\ty} f_L( \eta + \delta_{\tx}) \neq 0) 
			+ \P(D_{\tx} f_L(\eta + \delta_{\ty}) \neq 0, D_{\ty} f_L(\eta) \neq 0)\\ +\P(D_{\tx} f_L(\eta + \delta_{\ty}) \neq 0, D_{\ty} f_L(\eta + \delta_{\tx}) \neq 0).
		\end{multline}
	This relation parallels in some sense the notion of {\bf joint pivotality} in the discrete setting -- even though we do not have a similar equality as in the discrete setting; see \cite{GS}.}
 \end{remark}

\subsection{Quasi-multiplicativity in Boolean percolation --- proof of Theorem \ref{t:quasimPBM}}
\label{sec:quasim_boolean}
In this section, we prepare for the proof of Theorem \ref{t:sharpNS_Boolean} by proving the quasi-multiplicativity stated in Theorem \ref{t:quasimPBM} for 4-arm probabilities, as well as for 3-arm probabilities (in the half-plane) for the Boolean percolation model, which in particular confirms the validity of condition \ref{A2} in such a framework. We will deduce this from some lemmas and theorems, stated below, whose proofs are postponed to the subsequent sections. Using these, we complete the proof of Theorem \ref{t:quasimPBM} at the end of this section. Our strategy of proof is significantly inspired by \cite{Kes87}; however, we have tried as much as possible to present our arguments in a self-contained manner, with many illustrative figures. For the sake of readability, some technicalities have been skipped nonetheless, and we refer the reader to \cite{Kes87} for more details. The only places where a deeper knowledge of \cite{Kes87} is assumed are the last part of the proof of Lemma \ref{lem:del}, and the proof of \eqref{e:ineqlowcross} in the same Lemma, where precise references are provided.

We start with the case of 4-arm probabilities. As mentioned above, the proof follows the strategy of analogous results for Bernoulli percolation in \cite{Kes87} (see also \cite{Nol08}), but additionally accounting for local dependencies in Boolean percolation. The idea is to compare $4$-arm events with certain restricted $4$-arm events and show quasi-multiplicativity in the dyadic scale with respect to the latter. Showing quasi-multiplicativity with respect to the restricted events is slightly easier (see Theorem \ref{thm3}), but the challenge then lies in showing that the $4$-arm probabilities are comparable to the probabilities of the corresponding restricted $4$-arm events (see Theorems \ref{thm1} and \ref{thm2}). Finally, we extend quasi-multiplicativity from dyadic scales to all scales. %In order to make the main ideas of the proof more transparent and to keep the length of the present work within reasonable bounds, we do not give full proofs of all lemmas, but rather write detailed descriptions (or proof sketches) in some cases. We also provide more precise references as well as comparison to the analogous results in \cite{Kes87} to understand the differences better.

We recall the required notation. Let $\eta$ be a Poisson process on the unmarked space $\X = \R^2$ with intensity measure $\lambda \; \md x$ for some $\lambda>0$. Consider the Boolean percolation model, whose occupied region is defined as union of unit radius balls centered at points in the Poisson process.  Formally, $\Occ(\eta) := \cup_{x \in \eta}B_1(x)$ is the occupied set and its complement is the vacant set $\mathcal{V}(\eta) := \R^d \setminus \Occ(\eta)$. With the term ``percolation'', we indicate the existence of an unbounded connected component in $\Occ(\eta)$. It is known (see \cite{MR96}) that there exists a critical intensity $\lambda_c \in (0,\infty)$ such that
\[
\P(\mbox{$\Occ(\eta)$ percolates})= 
\begin{cases}
    0 & \text{if } \lambda \leq \lambda_c \\
   1 & \text{if } \lambda > \lambda_c.
\end{cases}
\]
Fix $\lambda = \lambda_c$, i.e., consider the model at criticality. We are interested in the event of a L-R crossing of the box $W_L$, $L>0$, through the occupied region $\Occ(\eta)$. Recall that $f_L$ denotes the $\pm1$-indicator of the above crossing event. The non-triviality of $f_L$ follows from a RSW type result (see \cite[Theorem~1.1(ii)]{Ahlbergsharp18} for example): for all $\kappa>0$, there exists a constant $c_0 = c_0(\kappa) \in (0,1/2)$ such that for every $L \ge 1$,
	\begin{equation}\label{eq:RSW}
		c_0<\P_{\lambda_c}(\text{there is a occupied LR crossing of the box $[0,\kappa L] \times [0,L]$}) <1-c_0.
	\end{equation}

With a slight abuse of notation (adopted uniquely in the present section), for $s \in \N_0$, we now write $W_s = [-2^s, 2^s]^2$ (instead of writing $W_{2^s}$) and let $A_4(s,r)\equiv A_4(2^s,2^r)$ for $r,s \in \N_0$, and so on. For a set $W$, we use the notation $\overline W:= W + B_1(0)$.

 For $\N \ni r \ge 4$, define $\cA(1,r) := [-2^{r-1},-2^r] \times [-2^{r-1},2^{r-1}]$, $\cA(3,r) := - \cA(1,r)$ to be the rectangular regions near the left and right boundaries of $W_r \setminus W_{r-1}$ respectively. Set $\cB(2,r)$ and $\cB(4,r)$ to be $\pi/2$-rotations in the clockwise direction of $\cA(1,r)$ and $\cA(3,r)$ respectively. Note that $\cB(2,r)$ and $\cB(4,r)$ are the rectangular regions near the top and bottom boundaries of $W_r \setminus W_{r-1}$ respectively. Further, define contractions of the above sets as $\cA(1,r)^- := [-2^{r-1}-2,-2^r+2] \times [-2^{r-1},2^{r-1}]$, $\cA(3,r)^- := - \cA(1,r)^-$, while $\cB(2,r)^-$ and $\cB(4,r)^-$ are again the $\pi/2$-rotations in  clockwise direction of $\cA(1,r)^-$ and $\cA(3,r)^-$ respectively. All of these sets are well defined since $2^{r-1}>4$; see Figure \ref{fig:A4T}.
 \medskip
 
 \noindent For $r \ge 4$ and an integer $0 \le s \le r-1$, define the event $\widetilde A_4(s,r)$ as (see Figure \ref{fig:A4T})
	\beaa
		\widetilde A_4(s,r) &:=&\text{there are four paths of alternating types from $\partial W_s$ to $\partial W_r$ such that all } \\
			&&\text{parts of the two occupied paths outside of $W_{r-1}$ lie in $\mathcal{A}(1,r)$ and $\mathcal{A}(3,r)$ }\\
			&&\text{while all parts of the vacant paths outside of $W_{r-1}$ lie in $\mathcal{B}(2,r)$ and }\\
			&&\text{$\mathcal{B}(4,r)$. In addition, there are occupied vertical crossings in $\mathcal{A}(1,r)^{-}$}\\
			&&\text{and $\mathcal{A}(3,r)^{-}$, and vacant horizontal crossings of $\mathcal{B}(2,r)^{-}$ and $\mathcal{B}(4,r)^{-}$.}
	\eeaa
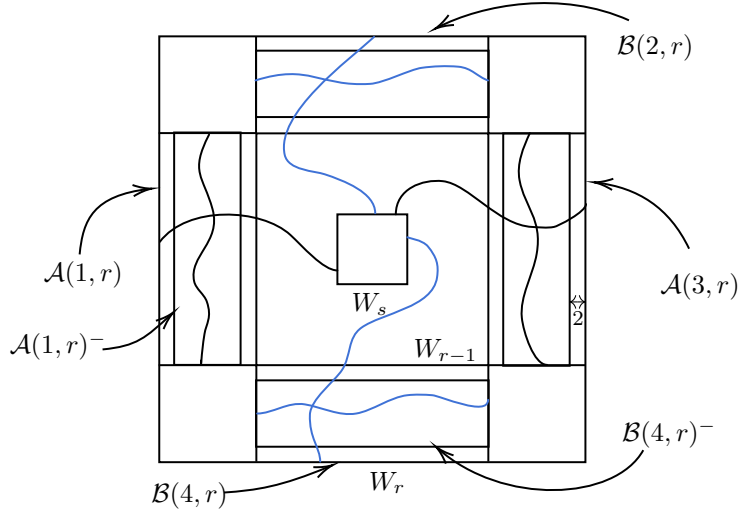
\begin{figure}
% \centering 
\tikzset{every picture/.style={line width=0.75pt}} %set default line width to 0.75pt        

\begin{tikzpicture}[x=0.65pt,y=0.65pt,yscale=-1,xscale=1]
\path (-20,300); %set diagram left start at 0, and has height of 300

%Shape: Square [id:dp9673608098145512] 
\draw   (191.67,25.33) -- (439.67,25.33) -- (439.67,273.33) -- (191.67,273.33) -- cycle ;
%Shape: Square [id:dp8225278251761081] 
\draw   (248.17,81.83) -- (383.17,81.83) -- (383.17,216.83) -- (248.17,216.83) -- cycle ;
%Shape: Right Angle [id:dp3371386677055819] 
\draw   (248,25.67) -- (248,81.67) -- (191.33,81.67) ;
%Shape: Right Angle [id:dp08290670214266038] 
\draw   (191.33,216.83) -- (248.17,216.83) -- (248.17,273) ;
%Shape: Right Angle [id:dp9687547548284934] 
\draw   (440,81.83) -- (383.17,81.83) -- (383.17,25) ;
%Shape: Right Angle [id:dp23813042121965378] 
\draw   (383.17,273) -- (383.17,216.83) -- (440,216.83) ;
%Shape: Rectangle [id:dp7415742955470384] 
\draw   (248,33.67) -- (383.33,33.67) -- (383.33,72.33) -- (248,72.33) -- cycle ;
%Shape: Rectangle [id:dp21543063151189212] 
\draw   (248,225.67) -- (383.33,225.67) -- (383.33,264.33) -- (248,264.33) -- cycle ;
%Shape: Rectangle [id:dp42723377874965074] 
\draw   (239,81.33) -- (239,216.67) -- (200.33,216.67) -- (200.33,81.33) -- cycle ;
%Shape: Rectangle [id:dp4162507915742486] 
\draw   (430.5,81.83) -- (430.5,217.17) -- (391.83,217.17) -- (391.83,81.83) -- cycle ;
%Shape: Square [id:dp7332813327113779] 
\draw   (295.33,129) -- (336,129) -- (336,169.67) -- (295.33,169.67) -- cycle ;
%Shape: Free Drawing [id:dp7299546187467527] 
\draw  [line width=0.75] [line join = round][line cap = round] (329.33,129) .. controls (329.33,114.12) and (344.79,108.39) .. (356.67,109.67) .. controls (371.77,111.28) and (388.56,135.33) .. (408.67,136.33) .. controls (412.88,136.54) and (417.18,137.09) .. (421.33,136.33) .. controls (423.79,135.89) and (440,127.56) .. (440,123) ;
%Shape: Free Drawing [id:dp4517561167186126] 
\draw  [line width=0.75] [line join = round][line cap = round] (192,145) .. controls (207.93,123.76) and (250.86,125.29) .. (268,141) .. controls (274.84,147.27) and (287.06,161.67) .. (295.33,161.67) ;
%Shape: Free Drawing [id:dp8869705452337742] 
\draw  [line width=0.75] [line join = round][line cap = round] (220.67,82.53) .. controls (209.17,104.88) and (217.33,106.5) .. (222.67,125.97) .. controls (227.98,145.36) and (208.61,153.05) .. (212,168.12) .. controls (213,172.57) and (217.85,175.46) .. (219.33,179.79) .. controls (222.54,189.14) and (219.01,192.45) .. (218,198.59) .. controls (215.97,210.91) and (216,220.41) .. (216,214.16) ;
%Shape: Free Drawing [id:dp13586245949200748] 
\draw  [line width=0.75] [line join = round][line cap = round] (407.33,81.68) .. controls (403.5,89.99) and (399.52,99.48) .. (401.33,109.07) .. controls (404.37,125.18) and (414.72,131.49) .. (411.33,151.84) .. controls (409.23,164.47) and (401.35,176.52) .. (402.67,189.26) .. controls (403.53,197.6) and (409.66,216.66) .. (417.33,216.66) ;
%Shape: Free Drawing [id:dp5505336479180345] 
\draw  [color={rgb, 255:red, 63; green, 114; blue, 214 }  ,draw opacity=1 ][line width=0.75] [line join = round][line cap = round] (247.67,51) .. controls (279.38,40.99) and (284.46,54.79) .. (314.33,53) .. controls (322.4,52.52) and (337.33,44.65) .. (344.33,43.67) .. controls (351.38,42.68) and (358.58,42.4) .. (365.67,43) .. controls (373.68,43.68) and (376.75,48.92) .. (383,51) ;
%Shape: Free Drawing [id:dp5561942116517731] 
\draw  [color={rgb, 255:red, 63; green, 114; blue, 214 }  ,draw opacity=1 ][line width=0.75] [line join = round][line cap = round] (316.67,25.67) .. controls (301.79,39.35) and (259.28,70.07) .. (264.67,94.33) .. controls (266.71,103.51) and (281.9,103.63) .. (288,105.67) .. controls (298.89,109.3) and (317.33,116.2) .. (317.33,128.33) ;
%Shape: Free Drawing [id:dp7199946786528157] 
\draw  [color={rgb, 255:red, 63; green, 114; blue, 214 }  ,draw opacity=1 ][line width=0.75] [line join = round][line cap = round] (285.33,273) .. controls (285.33,264.8) and (278.84,259.53) .. (279.33,245.67) .. controls (279.66,236.62) and (293.5,226.54) .. (298,219.67) .. controls (303.53,211.2) and (301.01,200.95) .. (309.33,195) .. controls (321.41,186.38) and (350.79,183.45) .. (353.33,165.67) .. controls (354.4,158.18) and (353.07,153.93) .. (348,147.67) .. controls (344.83,143.75) and (336,142.65) .. (336,142.33) ;
%Curve Lines [id:da9110033897895768] 
\draw    (145.33,152.33) .. controls (147.3,131.98) and (153.15,118.09) .. (185.82,114.17) ;
\draw [shift={(187.33,114)}, rotate = 173.84] [color={rgb, 255:red, 0; green, 0; blue, 0 }  ][line width=0.75]    (10.93,-3.29) .. controls (6.95,-1.4) and (3.31,-0.3) .. (0,0) .. controls (3.31,0.3) and (6.95,1.4) .. (10.93,3.29)   ;
%Curve Lines [id:da012078996963903021] 
\draw    (500,149) .. controls (488.3,124.62) and (469.95,113.56) .. (446.48,110.55) ;
\draw [shift={(444.67,110.33)}, rotate = 6.25] [color={rgb, 255:red, 0; green, 0; blue, 0 }  ][line width=0.75]    (10.93,-3.29) .. controls (6.95,-1.4) and (3.31,-0.3) .. (0,0) .. controls (3.31,0.3) and (6.95,1.4) .. (10.93,3.29)   ;
%Curve Lines [id:da6890006926740964] 
\draw    (465,16) .. controls (425.8,-1.64) and (379.88,6.65) .. (353.58,20.17) ;
\draw [shift={(352,21)}, rotate = 331.7] [color={rgb, 255:red, 0; green, 0; blue, 0 }  ][line width=0.75]    (10.93,-3.29) .. controls (6.95,-1.4) and (3.31,-0.3) .. (0,0) .. controls (3.31,0.3) and (6.95,1.4) .. (10.93,3.29)   ;
%Curve Lines [id:da19968103561336226] 
\draw    (229.33,291) .. controls (248.28,289.69) and (262.1,287.75) .. (290.89,275.75) ;
\draw [shift={(292.67,275)}, rotate = 157.11] [color={rgb, 255:red, 0; green, 0; blue, 0 }  ][line width=0.75]    (10.93,-3.29) .. controls (6.95,-1.4) and (3.31,-0.3) .. (0,0) .. controls (3.31,0.3) and (6.95,1.4) .. (10.93,3.29)   ;
%Curve Lines [id:da6523212551093673] 
\draw    (158.67,207.67) .. controls (178.71,205.73) and (183.71,200.96) .. (198.59,187.29) ;
\draw [shift={(200,186)}, rotate = 137.49] [color={rgb, 255:red, 0; green, 0; blue, 0 }  ][line width=0.75]    (10.93,-3.29) .. controls (6.95,-1.4) and (3.31,-0.3) .. (0,0) .. controls (3.31,0.3) and (6.95,1.4) .. (10.93,3.29)   ;
%Curve Lines [id:da327370859016332] 
\draw    (473.33,266.33) .. controls (456.83,296.03) and (391.33,298.94) .. (353.8,265.36) ;
\draw [shift={(352.67,264.33)}, rotate = 42.88] [color={rgb, 255:red, 0; green, 0; blue, 0 }  ][line width=0.75]    (10.93,-3.29) .. controls (6.95,-1.4) and (3.31,-0.3) .. (0,0) .. controls (3.31,0.3) and (6.95,1.4) .. (10.93,3.29)   ;
%Shape: Free Drawing [id:dp9686857248761971] 
\draw  [color={rgb, 255:red, 63; green, 114; blue, 214 }  ,draw opacity=1 ][line width=0.75] [line join = round][line cap = round] (248.63,245) .. controls (257.7,245) and (266.41,237.98) .. (276.44,239.67) .. controls (286.39,241.34) and (296.28,246.67) .. (306.23,245) .. controls (315.97,243.37) and (325.93,235.46) .. (336.03,234.33) .. controls (340.85,233.79) and (345.75,233.95) .. (350.59,234.33) .. controls (355.21,234.7) and (383.03,247.26) .. (383.03,237) ;

% Text Node
\draw (312,277) node [anchor=north west][inner sep=0.75pt]  [font=\footnotesize]  {$W_r$};
% Text Node
\draw (302,173) node [anchor=north west][inner sep=0.75pt]  [font=\footnotesize]  {$W_s$};
% Text Node
\draw (338,200) node [anchor=north west][inner sep=0.75pt]  [font=\footnotesize]  {$W_{r-1}$};
% Text Node
\draw (122.67,154.73) node [anchor=north west][inner sep=0.75pt]  [font=\footnotesize]  {$\mathcal{A}( 1,r)$};
% Text Node
\draw (481.67,161.4) node [anchor=north west][inner sep=0.75pt]  [font=\footnotesize]  {$\mathcal{A}( 3,r)$};
% Text Node
\draw (456.33,22.07) node [anchor=north west][inner sep=0.75pt]  [font=\footnotesize]  {$\mathcal{B}( 2,r)$};
% Text Node
\draw (185.33,283.4) node [anchor=north west][inner sep=0.75pt]  [font=\footnotesize]  {$\mathcal{B}( 4,r)$};
% Text Node
\draw (103,196) node [anchor=north west][inner sep=0.75pt]  [font=\footnotesize]  {$\mathcal{A}( 1,r)^{-}$};
% Text Node
\draw (459.33,246.73) node [anchor=north west][inner sep=0.75pt]  [font=\footnotesize]  {$\mathcal{B}( 4,r)^{-}$};

%added two width
\draw[ ->, thin] (435,180) -- (439.67,180);
\draw[<-, thin] (430.5,180) -- (435,180);

\draw (430.5,184) node [anchor=north west][inner sep=0.75pt]  [font=\footnotesize]  {\scriptsize{$2$}};

\end{tikzpicture}

 \caption{Illustration of the event $\widetilde A_4(s,r)$. The blue curves denote vacant paths while black ones denote occupied paths. Note, the crossings inside any contracted rectangle, e.g., $\cA^-(1,r)$, is independent of all other crossings of opposite type (blue in this case), in the figure.}
\label{fig:A4T}
\vspace{-.5cm}
\end{figure}
 
 \noindent Here, by a vertical crossing we mean a top to bottom crossing, while a horizontal crossing refers to a left to right crossing. Let $\widetilde \alpha_4(s,r):=\P_{\lambda_c}(\widetilde A_4(s,r))$ for $s \le r-1$ and $\widetilde \alpha_4(r,r)=1$.

 \begin{lemma}\label{lem:extend} There exists $C \in (1,\infty)$ such that, for all $r \ge 4$ and all $s \le r$, it holds that 
	$$
	\widetilde \alpha_4(s,r) \le C \widetilde \alpha_4(s,r+1).
	$$
\end{lemma}

%{not used} Note that Lemma~\ref{lem:extend} in particular implies that for $j \in \N_0$ with $j+4 \ge s$,
%\begin{equation}\label{eq:rem}
%	\widetilde \alpha_4(s,j+4) \ge \widetilde \alpha_4(s,4 \vee s)\;C^{-j \wedge (j+4-s)}.
%\end{equation}

\begin{theorem}\label{thm1} There exists $r_0 \in \N$ with $r_0 \geq 4$ such that for all $r \ge r_0$ and all $s \le r$,
		$$
		\alpha_4(s,r) \asymp \widetilde \alpha_4(s,r).
		$$
	\end{theorem} 

Next we define a similar event inward. For $3 \le s\le r-1$ (so that $2^s>4$), define the event $\dbtilde A_4(s,r)$ as (see Figure \ref{fig:dbtildeA4})
\beaa
\dbtilde{A}_4(s,r) & = & \text{there are four paths of alternating type from $\partial W_r$ to $\partial W_s$ such that all parts} \\
&&\text{of the two occupied paths outside of $W_r \setminus W_{s+1}$ lie in $\mathcal{A}(1,s+1)$ and }\\
&&\text{$\mathcal{A}(3,s+1)$, while all parts of the two vacant paths outside of $W_r \setminus W_{s+1}$ lie} \\
&&\text{in $\mathcal{B}(2,s+1)$ and $\mathcal{B}(4,s+1)$. In addition there are occupied vertical}\\
&&\text{crossings in $\mathcal{A}(1,s+1)^{-}$ and $\mathcal{A}(3,s+1)^{-}$, and vacant horizontal crossings of}\\
&&\text{$\mathcal{B}(2,s+1)^{-}$ and $\mathcal{B}(4,s+1)^{-}$.}
\eeaa

\begin{figure}
\tikzset{every picture/.style={line width=0.75pt}} %set default line width to 0.75pt        

\begin{tikzpicture}[x=0.65pt,y=0.65pt,yscale=-1,xscale=1]
\path (-40,300); %set diagram left start at 0, and has height of 300

%Shape: Square [id:dp9673608098145512] 
\draw   (191.67,25.33) -- (439.67,25.33) -- (439.67,273.33) -- (191.67,273.33) -- cycle ;
%Shape: Square [id:dp8225278251761081] 
\draw   (259.67,93.33) -- (371.67,93.33) -- (371.67,205.33) -- (259.67,205.33) -- cycle ;
%Shape: Right Angle [id:dp3371386677055819] 
\draw   (295.33,93.33) -- (295.33,129) -- (259.67,129) ;
%Shape: Right Angle [id:dp08290670214266038] 
\draw   (259.67,169.67) -- (295.33,169.67) -- (295.33,205.33) ;
%Shape: Right Angle [id:dp9687547548284934] 
\draw   (371.67,129) -- (336,129) -- (336,93.33) ;
%Shape: Right Angle [id:dp23813042121965378] 
\draw   (336,205.33) -- (336,169.67) -- (371.67,169.67) ;
%Shape: Rectangle [id:dp42723377874965074] 
\draw   (364.5,129.25) -- (364.5,169.75) -- (342.5,169.75) -- (342.5,129.25) -- cycle ;
%Shape: Square [id:dp7332813327113779] 
\draw   (295.33,129) -- (336,129) -- (336,169.67) -- (295.33,169.67) -- cycle ;
%Shape: Rectangle [id:dp6757364327578599] 
\draw   (288.5,129.25) -- (288.5,169.75) -- (266.5,169.75) -- (266.5,129.25) -- cycle ;
%Shape: Rectangle [id:dp9535933955098583] 
\draw   (295.32,99.63) -- (335.82,99.88) -- (335.68,122.62) -- (295.18,122.38) -- cycle ;
%Shape: Rectangle [id:dp21358423246966973] 
\draw   (295.33,175.63) -- (335.82,175.87) -- (335.68,199.12) -- (295.18,198.88) -- cycle ;
%Shape: Free Drawing [id:dp2690926321875111] 
\draw  [color={rgb, 255:red, 63; green, 114; blue, 214 }  ,draw opacity=1 ][line width=0.75] [line join = round][line cap = round] (285,25.25) .. controls (282.47,25.25) and (277.73,37.13) .. (279.5,41.25) .. controls (282.6,48.47) and (300.65,57.52) .. (302,59.25) .. controls (308.72,67.86) and (296.25,80.39) .. (304.5,87.75) .. controls (305.2,88.37) and (331.63,103.36) .. (324,112.75) .. controls (318.41,119.63) and (305.6,120.34) .. (303.5,128.75) ;
%Shape: Free Drawing [id:dp4025877111227538] 
\draw  [color={rgb, 255:red, 63; green, 114; blue, 214 }  ,draw opacity=1 ][line width=0.75] [line join = round][line cap = round] (439.52,222.91) .. controls (411.67,222.91) and (411.06,238.85) .. (385.19,243.29) .. controls (364.94,246.76) and (348.93,228.42) .. (329.37,233.84) .. controls (321.2,236.11) and (314.03,241.22) .. (305.95,243.78) .. controls (302.46,244.89) and (295.58,247.38) .. (294.98,243.78) .. controls (292.31,227.82) and (301.57,224.07) .. (309.93,211.97) .. controls (318,200.31) and (305.96,191.77) .. (306.45,179.16) .. controls (306.63,174.42) and (316.41,172.06) .. (316.41,169.72) ;
%Shape: Free Drawing [id:dp30569342935808863] 
\draw  [color={rgb, 255:red, 63; green, 114; blue, 214 }  ,draw opacity=1 ][line width=0.75] [line join = round][line cap = round] (295.5,187.25) .. controls (311.57,197.96) and (316.22,174.71) .. (336,186.25) ;
%Shape: Free Drawing [id:dp2000044146272968] 
\draw  [color={rgb, 255:red, 63; green, 114; blue, 214 }  ,draw opacity=1 ][line width=0.75] [line join = round][line cap = round] (295.5,109.75) .. controls (301.61,118.92) and (312,108.51) .. (317.5,106.25) .. controls (322.6,104.15) and (326.34,111.23) .. (330,112.75) .. controls (331.9,113.54) and (333.94,111.25) .. (336,111.25) ;
%Shape: Free Drawing [id:dp501726867339563] 
\draw  [color={rgb, 255:red, 0; green, 0; blue, 0 }  ,draw opacity=1 ][line width=0.75] [line join = round][line cap = round] (278,129.15) .. controls (275.01,131.72) and (269.31,136.68) .. (271.5,141.41) .. controls (272.87,144.35) and (277.46,145.2) .. (278.5,148.27) .. controls (282.64,160.45) and (280.8,160.53) .. (276,169.35) ;
%Shape: Free Drawing [id:dp5075262012803015] 
\draw  [color={rgb, 255:red, 0; green, 0; blue, 0 }  ,draw opacity=1 ][line width=0.75] [line join = round][line cap = round] (350.5,129.25) .. controls (350,132.42) and (348.32,135.62) .. (349,138.75) .. controls (349.32,140.21) and (352.25,139.46) .. (353,140.75) .. controls (354.28,142.96) and (354.73,145.71) .. (354.5,148.25) .. controls (354.21,151.4) and (352.35,154.2) .. (351.5,157.25) .. controls (351.02,158.99) and (356.5,161.49) .. (357,162.75) .. controls (357.87,164.92) and (357.5,167.41) .. (357.5,169.75) ;
%Shape: Free Drawing [id:dp3922024759271092] 
\draw  [color={rgb, 255:red, 0; green, 0; blue, 0 }  ,draw opacity=1 ][line width=0.75] [line join = round][line cap = round] (376.26,25.27) .. controls (363.84,25.27) and (367.13,45.98) .. (371.23,49.86) .. controls (377.79,56.05) and (401.45,67.08) .. (403.92,76.96) .. controls (406.73,88.17) and (397.74,95.68) .. (394.87,105.56) .. controls (393.04,111.88) and (392.5,118.56) .. (392.36,125.13) .. controls (392.13,135.83) and (415.48,131.06) .. (411.97,147.71) .. controls (410.76,153.45) and (399.02,153.17) .. (396.38,152.73) .. controls (384.07,150.68) and (363.3,136.14) .. (352.63,140.19) .. controls (346.58,142.48) and (336.03,143.77) .. (336.03,150.22) ;
%Shape: Free Drawing [id:dp85731692075048] 
\draw  [color={rgb, 255:red, 0; green, 0; blue, 0 }  ,draw opacity=1 ][line width=0.75] [line join = round][line cap = round] (192,130.25) .. controls (206.7,108.2) and (223.9,149.24) .. (237.5,150.75) .. controls (258.55,153.09) and (276.08,131.93) .. (295,142.25) ;

% Text Node
\draw (442,150) node [anchor=north west][inner sep=0.75pt]  [font=\small]  {$W_r$};
% Text Node
\draw (310.5,150) node [anchor=north west][inner sep=0.75pt]  [font=\small]  {$W_s$};
% Text Node
\draw (335.5,210) node [anchor=north west][inner sep=0.75pt]  [font=\small]  {$W_{s+1}$};

\end{tikzpicture}
\vspace{-.5cm}
\caption{Illustration of the event ${\overset{\; \approx}{A}}_4(s,r)$ with blue denoting vacant paths and black denoting occupied paths.}
\label{fig:dbtildeA4}
\end{figure}
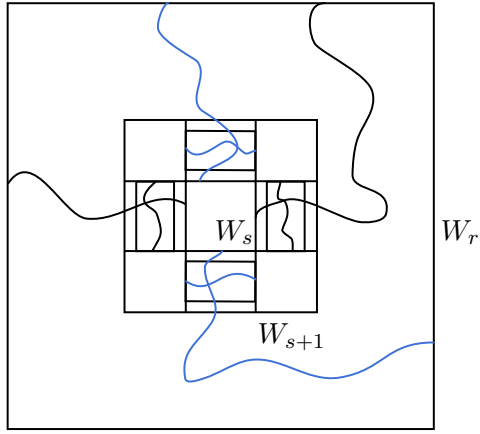

\noindent Let $\dbtilde \alpha_4(s,r)=\P_{\lambda_c}(\dbtilde A_4(s,r))$ for $s \le r - 1$ and $\dbtilde \alpha_4(r,r)=1$.

\begin{lemma}\label{lem:extendin} There exists $C \in (1,\infty)$ such that for all $s \ge 4$ and all $r \geq s$ it holds that
	$$
	\dbtilde \alpha_4(s,r) \le C \dbtilde \alpha_4(s-1,r).
	$$
\end{lemma}

%Note that this implies that for $j \ge 0$ with $s_1(R)-j \le r$,
%\begin{equation}\label{eq:remin}
%	\widetilde \alpha_4(s_1(R)-j,r) \ge \widetilde \alpha_4(s_1(R),r)C_0'(R)^{-j}.
%\end{equation}

\begin{theorem}\label{thm2} There exists $s_0 \in \N$ with $s_0 \geq 3$ such that for all $s \ge s_0$ and all $r \ge s$, it holds that
	$$
	\alpha_4(s,r) \asymp \dbtilde \alpha_4(s,r).
	$$
\end{theorem}

As a final ingredient to prove quasi-multiplicativity, we need the following result.

\begin{theorem}\label{thm3} Let $r_0,s_0$ be as in Theorems \ref{thm1} and \ref{thm2} and set  $r_1=(s_0+2) \vee r_0$. There exists $C \in (0,\infty)$ depending only on $r_1$ such that for all $t\ge r$ and all $r\ge s$ with $r \ge r_1$, it holds that 
	$$
	\alpha_4(s,t) \ge C \widetilde \alpha_4(s,r) \dbtilde \alpha_4(r,t).
	$$
\end{theorem}
The proofs of Lemmas \ref{lem:extend} and \ref{lem:extendin} are similar and we only prove the former in Section \ref{ss:proofs_lem1}. The proofs involve tackling the local dependence in the model, and rely upon the FKG inequality for Poisson process as well as the RSW bound \eqref{eq:RSW}. The proofs of Theorems \ref{thm1} and \ref{thm2} in Section \ref{ss:proof_thms12} are more involved and form the most technical part in the proof of quasi-multiplicativity. The proof of Theorem \ref{thm3}, which again relies upon the FKG inequality and the RSW bound along with standard gluing arguments, is presented in Section \ref{ss:proof_thm3}. 

Now we complete the proof of Theorem \ref{t:quasimPBM} using the above results.

\begin{proof}[Proof of Theorem \ref{t:quasimPBM}] Recall that in this section, $\alpha_4(2^s,2^r)$ is denoted by $\alpha_4(s,r)$.
\medskip

\noindent \underline{\textit{Proof in the dyadic case for $\alpha_4$}:} We first prove quasi-multiplicativity in the dyadic case, i.e., we show that for non-negative integers $t\ge r\ge s$, we have
\begin{equation}\label{eq:qmd}
\alpha_4(s,t) \asymp C_1 \alpha_4(s,r) \alpha_4(r,t).
\end{equation}
To this end, let $r_1$ be as in Theorem \ref{thm3}. For any non-negative integers $t\ge r\ge s$ when $t \le r_1$, it trivially holds that
\begin{equation}\label{eq:tler}
    \alpha_4(s,t) \asymp C_1 \alpha_4(s,r) \alpha_4(r,t)
\end{equation}
for some constant $C_1 \in (0,\infty)$. Now consider $t\ge r\ge s$ with $r \le r_1$ and $t >r_1$. By the above three theorems (Theorems \ref{thm1}, \ref{thm2} and \ref{thm3}), we obtain that there exists $C_1 \in (0,\infty)$ such that
	$$
	\alpha_4(s,t) \ge C_1 \alpha_4(s,r_1) \alpha_4(r_1,t) \ge C_1 C_2 \alpha_4(s,r) \alpha_4(r,t),
	$$
	where in the final inequality, we have used that there exists a constant $C_2 \in (0,\infty)$ such that $\alpha_4(s,r_1) \ge C_2 \alpha_4(s,r)$ for any $s \le r \le r_1$. Finally, when $r>r_1$, Theorems \ref{thm1}, \ref{thm2} and \ref{thm3} again imply
 $$
 \alpha_4(s,t) \ge C_1 \alpha_4(s,r) \alpha_4(r,t)
 $$
 for some $C_1 \in (0,\infty)$. This proves
 that for $t\ge r\ge s$, we have
\begin{equation}\label{eq:qmd1}
\alpha_4(s,t) \gtrsim \alpha_4(s,r) \alpha_4(r,t),
\end{equation}
showing one side of \eqref{eq:qmd}.
 
The other direction in \eqref{eq:qmd} is much easier to argue using conditional probability. Fix $r \ge r_0+1$. For any $t \ge r\ge s +1$, notice by independence (i.e., $A_4(s,t) \subseteq A_4(s,r-1) \cap A_4(r,t)$, and these are independent), Theorem~\ref{thm1} and Lemma~\ref{lem:extend} that 
	\beaa
	\alpha_4(s,t) &\le& \alpha_4(s,r-1) \alpha_4(r,t) \asymp \widetilde \alpha_4(s,r-1) \alpha_4(r,t) \\
	&\le& C \widetilde \alpha_4(s,r) \alpha_4(r,t) \asymp \alpha_4(s,r) \alpha_4(r,t),
	\eeaa
while the inequality is trivial if $s=r$. Finally, if $r \le r_0$, we argue similarly as above. The case where $t \le r_0$ is already considered in \eqref{eq:tler}, while if $s \le r \le r_0 < t$, we have
	$$
	\alpha_4(s,t) \le C_2^{-1}\alpha_4(s,r_0) \alpha_4(r_0,t) \le C_2^{-2} \alpha_4(s,r) \alpha_4(r,t),
	$$
	where we use that there exists $C_2 \in (0,\infty)$ with $\inf_{s' \le r' \le r_0} \alpha_4(s',r') \ge C_2$, and the final step uses \eqref{eq:qmd1}, so that
	$$
	\alpha_4(r,t) \gtrsim \alpha_4(r,r_0)\alpha_4(r_0,t) \ge C_2 \alpha_4(r_0,t).
	$$ 
This proves the other inequality in \eqref{eq:qmd}, completing the proof in the dyadic case.
\medskip 

\noindent \underline{\textit{Proof in the general case for $\alpha_4$}:}
First notice that for any positive integers $j \le k$, we have by \eqref{eq:qmd} that, 
\begin{equation}\label{eq:dyadic}
\alpha_4(j-1,k+1) \asymp \alpha_4(j-1,j) \alpha_4(j,k) \alpha_4(k,k+1) \asymp \alpha_4(j,k),
\end{equation}
where the final step is by noting from \ref{A1}, which is satisfied by our Boolean model (see \cite[(2.16)]{MS22}), that $1 \ge \alpha_4(i-1,i) \ge \epsilon (1/2)^{2-\epsilon}$ for any $i \in \N_0$. The result for arbitrary real numbers $1 \le r_1 \le r_2 \le r_3$ now follows by a dyadic approximation of $r_1,r_2$ and $r_3$ and then using \eqref{eq:dyadic} and \eqref{eq:qmd}. This completes the proof of Theorem \ref{t:quasimPBM} for $\alpha_4$.

\medskip 

\noindent \underline{\textit{Proof for $\alpha_3^+$}:}
Note that, in the half-plane, the arms in the 3-arm event are always of alternating type (which is not the case when looking for 3-arms in the whole annulus). Thus, the proof of quasi-multiplicativity for the 3-arm probabilities in the half-plane follows exactly that of $\alpha_4$, by proving suitably adapted versions of Lemmas \ref{lem:extend}, \ref{lem:extendin}, and Theorems \ref{thm1}, \ref{thm2} and \ref{thm3} for $\alpha_3^+$.
\end{proof}

\subsubsection{Proofs of Lemmas \ref{lem:extend} and \ref{lem:extendin}}
\label{ss:proofs_lem1}

The proofs of these lemmas are very similar, so we just provide a proof for Lemma \ref{lem:extend} here.

\begin{proof}[Proof of Lem.\ \ref{lem:extend}] Define the event (see Figure \ref{A4TExt})
	\beaa
	\widetilde B(r,r+1) &=&\text{there are two vertical occupied crossings of $\mathcal{A}(1,r+1)^-$, $\mathcal{A}(3,r+1)^-$ and} \\ 
     && \text{two horizontal vacant crossings of $\mathcal{B}(2,r+1)^-$, $\mathcal{B}(4,r+1)^-$, respectively.}\\
	&&\text{In addition there are two horizontal occupied crossings of the rectangles }\\
	&&\text{$[2^{r-1}+2, 2^{r+1}]\times[ -2^{r-1},2^{r-1}]$, and its reflection along the $y$-axis, and two }\\
	&&\text{vertical vacant crossings of rectangles obtained by $\pi/2$ clockwise rotations}\\
        &&\text{of them}.
	\eeaa

\begin{figure} 
\tikzset{every picture/.style={line width=0.75pt}} %set default line width to 0.75pt        

\begin{tikzpicture}[x=0.5pt,y=0.5pt,yscale=-1,xscale=1]
\path (-45,490); %set diagram left start at 0, and has height of 494

%Shape: Square [id:dp7509078966622402] 
\draw   (265,86) -- (588,86) -- (588,409) -- (265,409) -- cycle ;
%Shape: Square [id:dp8393076050459758] 
\draw   (320.25,141.25) -- (532.75,141.25) -- (532.75,353.75) -- (320.25,353.75) -- cycle ;
%Shape: Square [id:dp9959434788950945] 
\draw   (371.38,192.38) -- (481.63,192.38) -- (481.63,302.63) -- (371.38,302.63) -- cycle ;
%Shape: Right Angle [id:dp08882595356010947] 
\draw   (320.25,86) -- (320.25,141.25) -- (265,141.25) ;
%Shape: Right Angle [id:dp0006252369468797703] 
\draw   (532.75,409) -- (532.75,353.75) -- (588.75,353.75) ;
%Shape: Right Angle [id:dp534135583430547] 
\draw   (588,141.25) -- (532.75,141.25) -- (532.75,86) ;
%Shape: Right Angle [id:dp6216707142438411] 
\draw   (265,353.77) -- (320.25,353.75) -- (320.27,409) ;
%Shape: Rectangle [id:dp033708233070655735] 
\draw   (320,98) -- (533,98) -- (533,130) -- (320,130) -- cycle ;
%Shape: Rectangle [id:dp08263270088980867] 
\draw   (320,366) -- (533,366) -- (533,398) -- (320,398) -- cycle ;
%Shape: Rectangle [id:dp31437086252578994] 
\draw   (276.5,353.67) -- (276.5,141) -- (308.5,141) -- (308.5,353.67) -- cycle ;
%Shape: Rectangle [id:dp3117650181967928] 
\draw   (575.5,141) -- (575.5,353.5) -- (543.5,353.5) -- (543.5,141) -- cycle ;
%Shape: Right Angle [id:dp5441548478863614] 
\draw   (321,302.63) -- (371.38,302.63) -- (371.38,354) ;
%Shape: Right Angle [id:dp682693099389998] 
\draw   (481.63,353) -- (481.63,302.63) -- (533,302.63) ;
%Shape: Right Angle [id:dp8293607905013762] 
\draw   (532,192.38) -- (481,192.38) -- (481,141) ;
%Shape: Right Angle [id:dp23094215068835533] 
\draw   (371.38,142) -- (371.38,192.38) -- (320,192.38) ;
%Shape: Rectangle [id:dp6540934521046624] 
\draw  [line width=0.75]  (371.67,151) -- (481,151) -- (481,182.33) -- (371.67,182.33) -- cycle ;
%Shape: Rectangle [id:dp4674566138380629] 
\draw   (372,313) -- (481,313) -- (481,344) -- (372,344) -- cycle ;
%Shape: Rectangle [id:dp4144226759006515] 
\draw   (491,302) -- (491,192.33) -- (522,192.33) -- (522,302) -- cycle ;
%Shape: Rectangle [id:dp728599407941666] 
\draw   (330,302) -- (330,193) -- (361,193) -- (361,302) -- cycle ;
%Shape: Rectangle [id:dp0677286956124672] 
\draw  [color={rgb, 255:red, 0; green, 0; blue, 0 }  ,draw opacity=1 ][line width=0.75]  (491,192.33) -- (588,192.33) -- (588,302.33) -- (491,302.33) -- cycle ;
%Shape: Rectangle [id:dp5642990064279514] 
\draw  [color={rgb, 255:red, 0; green, 0; blue, 0 }  ,draw opacity=1 ][line width=0.75]  (265,193) -- (361,193) -- (361,302) -- (265,302) -- cycle ;
%Shape: Rectangle [id:dp7682496234567895] 
\draw  [color={rgb, 255:red, 0; green, 0; blue, 0 }  ,draw opacity=1 ][line width=0.75]  (372,182) -- (372,86) -- (481,86) -- (481,182) -- cycle ;
%Shape: Rectangle [id:dp24752776074320115] 
\draw  [color={rgb, 255:red, 0; green, 0; blue, 0 }  ,draw opacity=1 ][line width=0.75]  (372,409) -- (372,313) -- (481,313) -- (481,409) -- cycle ;
%Shape: Free Drawing [id:dp6443648274662999] 
\draw  [color={rgb, 255:red, 0; green, 0; blue, 0 }  ,draw opacity=1 ][line width=0.75] [line join = round][line cap = round] (561,141) .. controls (557.4,141) and (550.04,164.18) .. (551,169) .. controls (554.61,187.05) and (564.73,206.47) .. (563,229) .. controls (561.4,249.78) and (551.62,272.47) .. (557,294) .. controls (560.72,308.86) and (566.3,319.51) .. (563,336) .. controls (561.86,341.69) and (555,353.13) .. (555,353) ;
%Shape: Free Drawing [id:dp7178138335203701] 
\draw  [color={rgb, 255:red, 0; green, 0; blue, 0 }  ,draw opacity=1 ][line width=0.75] [line join = round][line cap = round] (291,142) .. controls (288,142) and (286.95,148.14) .. (286,151) .. controls (281.6,164.21) and (290.03,177.47) .. (291,191) .. controls (291.53,198.35) and (291.23,211.38) .. (290,220) .. controls (288.95,227.32) and (284.31,232.45) .. (284,240) .. controls (283.69,247.33) and (283.72,254.67) .. (284,262) .. controls (284.61,277.76) and (292.04,292.72) .. (290,309) .. controls (289.13,315.98) and (284.57,322.57) .. (284,330) .. controls (283.03,342.55) and (290.32,354) .. (301,354) ;
%Shape: Free Drawing [id:dp28562980317954834] 
\draw  [color={rgb, 255:red, 0; green, 0; blue, 0 }  ,draw opacity=1 ][line width=0.75] [line join = round][line cap = round] (490.86,249) .. controls (495.92,249) and (500.68,241.08) .. (511.63,243) .. controls (517.87,244.09) and (529.49,253.4) .. (537.59,256) .. controls (545.24,258.46) and (555.38,260.79) .. (563.55,260) .. controls (570.75,259.31) and (588.47,241.85) .. (588.47,237) ;
%Shape: Free Drawing [id:dp8377642220921213] 
\draw  [color={rgb, 255:red, 0; green, 0; blue, 0 }  ,draw opacity=1 ][line width=0.75] [line join = round][line cap = round] (265.05,266) .. controls (266.44,263.71) and (270.67,256.19) .. (272.11,255) .. controls (295.33,235.84) and (310.07,255.68) .. (333.69,251) .. controls (342.57,249.24) and (349.64,239.2) .. (353.88,235) .. controls (355.95,232.95) and (362.8,231) .. (360.94,231) ;
%Shape: Free Drawing [id:dp18438763703908845] 
\draw  [color={rgb, 255:red, 63; green, 114; blue, 214 }  ,draw opacity=1 ][line width=0.75] [line join = round][line cap = round] (320.99,113) .. controls (346.08,96.22) and (377.11,115.82) .. (401.74,117) .. controls (422.91,118.01) and (439.94,112.46) .. (459.57,110) .. controls (476.15,107.92) and (503.47,114.74) .. (514.41,112) .. controls (519.55,110.71) and (522.62,110.51) .. (526.37,109) .. controls (528.28,108.23) and (534.16,105) .. (532.35,105) ;
%Shape: Free Drawing [id:dp7720567475983307] 
\draw  [color={rgb, 255:red, 63; green, 114; blue, 214 }  ,draw opacity=1 ][line width=0.75] [line join = round][line cap = round] (321.02,383) .. controls (328,376.05) and (338.74,375.69) .. (349.15,375) .. controls (374.06,373.35) and (375.54,375.58) .. (401.38,381) .. controls (414.56,383.76) and (427.92,385.8) .. (441.57,385) .. controls (450.43,384.48) and (460.25,380.38) .. (468.69,380) .. controls (485.94,379.22) and (498.67,384.15) .. (514.9,383) .. controls (519.35,382.68) and (522.23,378.23) .. (525.95,377) .. controls (528.33,376.21) and (530.88,376.09) .. (532.98,374) ;
%Shape: Free Drawing [id:dp6226282857922163] 
\draw  [color={rgb, 255:red, 63; green, 114; blue, 214 }  ,draw opacity=1 ][line width=0.75] [line join = round][line cap = round] (399,85.95) .. controls (399,90.59) and (397.64,95.2) .. (398,99.82) .. controls (398.51,106.34) and (410.76,117.41) .. (414,120.63) .. controls (417.73,124.32) and (434.39,132.69) .. (436,137.47) .. controls (439.65,148.32) and (437.12,160.63) .. (435,171.15) .. controls (434.34,174.43) and (438,182.57) .. (438,181.06) ;
%Shape: Free Drawing [id:dp35585733800779384] 
\draw  [color={rgb, 255:red, 63; green, 114; blue, 214 }  ,draw opacity=1 ][line width=0.75] [line join = round][line cap = round] (449,313.53) .. controls (449,316.7) and (438.01,331.13) .. (437,340.55) .. controls (435.5,354.54) and (448.45,376.45) .. (437,388.35) .. controls (431.84,393.71) and (422.36,398.38) .. (418,402.9) .. controls (417.08,403.86) and (407.41,409.14) .. (410,409.14) ;

% Text Node
\draw (592,215
) node [anchor=north west][inner sep=0.75pt]    {\small{$W_{r+1}$}};
% Text Node
\draw (499,330) node [anchor=north west][inner sep=0.75pt]    {\small{$W_r$}};
% Text Node
\draw (405,278) node [anchor=north west][inner sep=0.75pt]    {\small{$W_{r-1}$}};
\end{tikzpicture}

\vspace{-1.5cm}
\caption{Illustration of the event $\widetilde B(r,r+1)$ with blue denoting vacant paths and black denoting occupied paths.}
\label{A4TExt}
\end{figure}
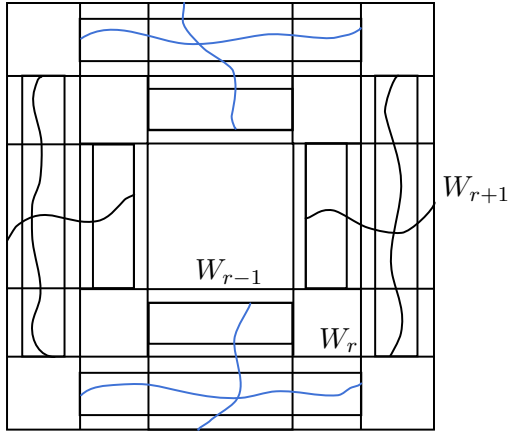

\noindent We then have $\widetilde A_4(s,r+1) \supseteq \widetilde A_4(s,r) \cap \widetilde B(r,r+1)$. In addition, one can see that $\widetilde B(r,r+1)$ is positively associated with $\widetilde A_4(s,r)$. Write $ \widetilde A_4(s,r)=A_+ \cap A_-$ and $\widetilde B(r,r+1) = B_+ \cap B_-$, where $A_+$ and $B_+$ are increasing events (ones corresponding to occupied events), and $A_-$ and $B_-$ are decreasing events (ones corresponding to vacant events). Notice by construction that the events $A_+ \cap B_+$ and $A_- \cap B_-$ are also conditionally independent, given the configuration in $\overline W_{r-1}=W_{r-1} + B_1(0)$. Thus,
	\beaa
	\P_{\lambda_c}(\widetilde A_4(s,r+1) | \eta_{\overline W_{r-1}}) &\ge& \P_{\lambda_c}(A_+ \cap A_- \cap B_+ \cap B_- | \eta_{\overline W_{r-1}}) \\
	&=& \P_{\lambda_c}(A_+\cap B_+ | \eta_{\overline W_{r-1}})\P_{\lambda_c}(A_- \cap B_- | \eta_{\overline W_{r-1}})\\
	& \ge& \P_{\lambda_c}(A_+| \eta_{\overline W_{r-1}})\P_{\lambda_c}(A_-| \eta_{\overline W_{r-1}}) \P_{\lambda_c}(B_+) \P_{\lambda_c}(B_-)\\
	& \ge & C_0 \P_{\lambda_c}(A_+ \cap A_-| \eta_{\overline W_{r-1}}),
	\eeaa
	where we have used the FKG inequality for the second inequality and in the final inequality, we have used that $B_+$ and $B_-$ are independent of $\eta_{\overline W_{r-1}}$ as well as RSW bounds from \eqref{eq:RSW}. The result now follows by taking expectation on both sides.
\end{proof}

\subsubsection{Proofs of Theorems \ref{thm1} and \ref{thm2}} 
\label{ss:proof_thms12}
We will first prove some lemmas to prepare for the proof of Theorem \ref{thm1}, and at the end of the subsection, prove the theorem. Since the proof of Theorem \ref{thm2} is similar, we will only quickly mention the differences. 

We shall now define the notions of a {\bf fence} and of a {\bf lowest crossing}, borrowing similar concepts from \cite[p.\ 121-123]{Kes87}. We shall avoid formal definitions for the sake of readability and also not justify measurability of the various events. These can be justified as in \cite{Kes87}.

The definition below is illustrated in Figure \ref{fig:path}. Fix $\theta \in (0,1/4)$, and an integer $k > 1- \log_2 \theta$ which ensures that $\sqrt{\theta} 2^{k-1}>\theta 2^k>2$. Let $\gamma$ be a (occupied or vacant) path from $\partial W_{k-1}$ to $\partial W_k$ in $W_k \setminus W_{k-1}$. Assuming that $\gamma$ hits the right boundary of $W_k$, denote by $\gamma'$ its {\em crosscut} in the right strip $\mathcal{S}_R = [2^{k-1} + 2,2^k] \times [-2^k,2^k]$ i.e., the piece of $\gamma$ from its last intersection of the line $x = 2^{k-1} + 2$ to the right edge of $\mathcal{S}_R$. Let $\mathcal{C} = \mathcal{C}_{\gamma',k}$ (or interchangeably $\mathcal{C}_{\gamma,k}$) be the component (occupied or vacant depending on the type of $\gamma$) of $\gamma'$ in $\mathcal{S}_R$ and let the lowest point of the component be $a=a(\mathcal{C}_{\gamma,k})=(a(1),a(2))$ on the right side of $\partial W_k$.  Observe that $a$ is the rightmost end-point of the lowest horizontal crossing of $\mathcal{C}_{\gamma,k}$. The path $\gamma$, or rather its component $\mathcal{C}$ ($\mathcal{C}_{\gamma,k}$ or $\mathcal{C}_{\gamma',k}$) is said to have {\em a $(\theta, k)$-fence}, if 
\begin{enumerate}[(i)]
    \item \noindent for another path (occupied or vacant) $\gamma_0$ from $\partial W_{k-1}$ to $\partial W_k$ with $\mathcal{C}_{\gamma,k} \cap \mathcal{C}_{\gamma_0,k} = \emptyset$ (this holds trivially if $\gamma$ and $\gamma_0$ are of different types), 
    \begin{equation}\label{e:separation}
    |a(\mathcal{C}_{\gamma,k}) - a(\mathcal{C}_{\gamma_0,k})|>2\sqrt{\theta} 2^k,
    \end{equation}
    \item \noindent \label{i:fence} there is a vertical crossing of the same type as $\gamma$ of the rectangle
    \begin{equation}\label{e:fence}
        R(a):=[a(1) +2,a(1) + \sqrt{\theta} 2^{k-1}] \times [a(2) - \theta 2^{k-1}, a(2) + \theta 2^{k-1}],
    \end{equation}
    which is connected to $\mathcal{C}_{\gamma,k}$ by a path of the same type in 
    $$S(a,\sqrt{\theta}2^{k-1}):=a+[-\sqrt{\theta}2^{k-1},\sqrt{\theta}2^{k-1}]^2.
    $$
\end{enumerate}
Finally, note that we can also define all the above for the similarly defined strips $\mathcal{S}_L$, $\mathcal{S}_T$, or $\mathcal{S}_B$ on the left, top and bottom, respectively.

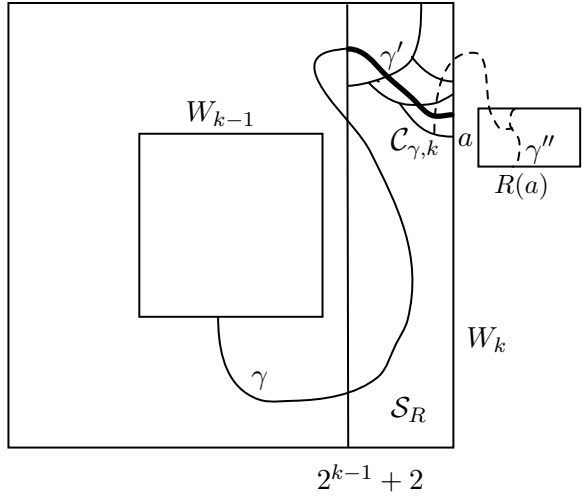
\begin{figure}
    \tikzset{every picture/.style={line width=0.75pt}} %set default line width to 0.75pt        

\tikzset{every picture/.style={line width=0.75pt}} %set default line width to 0.75pt        

\begin{tikzpicture}[x=0.55pt,y=0.55pt,yscale=-1,xscale=1]
\path (160,446); %set diagram left start at 0, and has height of 446

%Shape: Square [id:dp6284913256692546] 
\draw   (398,67) -- (704,67) -- (704,373) -- (398,373) -- cycle ;
%Shape: Square [id:dp1511615860028872] 
\draw   (488,157) -- (614,157) -- (614,283) -- (488,283) -- cycle ;
%Straight Lines [id:da9387989488681594] 
\draw    (631,68) -- (631.67,372.67) ;
%Shape: Free Drawing [id:dp7428723754933837] 
\draw  [line width=0.75] [line join = round][line cap = round] (631,98.4) .. controls (628.85,98.4) and (626.74,98.96) .. (624.6,99.2) .. controls (587.16,103.36) and (624.64,139.27) .. (636.6,154.4) .. controls (640.41,159.22) and (644.25,164.11) .. (647,169.6) .. controls (665.67,206.93) and (683.26,240.97) .. (672.6,284.8) .. controls (671.13,290.83) and (667.4,296.25) .. (664.6,301.6) .. controls (661.48,307.56) and (658.29,319.12) .. (652.6,324) .. controls (637.23,337.17) and (613.63,340.33) .. (594.2,340.8) .. controls (579.68,341.15) and (573.46,343.24) .. (562.2,336) .. controls (545.21,325.07) and (542.2,301.34) .. (542.2,283.2) ;
%Shape: Free Drawing [id:dp18966497031959295] 
\draw  [line width=0.75] [line join = round][line cap = round] (681.86,67) .. controls (681.86,80.51) and (681.66,93.39) .. (675.69,105) .. controls (669.82,116.43) and (641.74,124) .. (631.47,124) ;
%Shape: Free Drawing [id:dp4892197316032414] 
\draw  [line width=0.75] [line join = round][line cap = round] (675.42,104) .. controls (680.56,111.12) and (690.9,121) .. (703.58,121) ;
%Shape: Free Drawing [id:dp015621543900202983] 
\draw  [line width=0.75] [line join = round][line cap = round] (652,121) .. controls (652,121) and (652,121) .. (652,121) ;
%Shape: Free Drawing [id:dp9373470965041135] 
\draw  [line width=0.75] [line join = round][line cap = round] (645.86,122) .. controls (652.1,128.43) and (661.31,136.24) .. (673.04,137) .. controls (679.37,137.41) and (688.45,137.45) .. (694.41,135) .. controls (696.72,134.05) and (703.46,130) .. (703.5,130) ;
%Shape: Free Drawing [id:dp2756502925945876] 
\draw  [line width=0.75] [line join = round][line cap = round] (666.93,136) .. controls (677.53,150.34) and (683.84,159) .. (703.4,159) ;
%Shape: Rectangle [id:dp5087152534866322] 
\draw   (721.2,139.6) -- (791.2,139.6) -- (791.2,179.6) -- (721.2,179.6) -- cycle ;
%Shape: Free Drawing [id:dp9556808636570009] 
\draw  [dashed, line width=0.75] [line join = round][line cap = round] (746.2,140.14) .. controls (743.62,140.14) and (742.44,147.97) .. (743,150.85) .. controls (743.96,155.8) and (748.92,156.8) .. (749.4,163.19) .. controls (749.72,167.5) and (749.4,179.66) .. (744.6,179.66) ;
%Shape: Free Drawing [id:dp5918225434683924] 
\draw  [line width=2] [line join = round][line cap = round] (631.64,98.65) .. controls (642.6,98.65) and (649.23,108.1) .. (655.67,114.82) .. controls (661.72,121.13) and (671.31,127.9) .. (678.15,135.04) .. controls (681.46,138.49) and (685.46,143.67) .. (690.55,144.74) .. controls (694.64,145.59) and (700.2,143.93) .. (702.96,143.93) ;
%Shape: Free Drawing [id:dp9119576548048662] 
\draw  [dashed, line width=0.75] [line join = round][line cap = round] (690.91,156.18) .. controls (690.91,135.88) and (694.68,98.31) .. (719.28,101.42) .. controls (729.52,102.72) and (725.34,131.67) .. (729.52,142.24) .. controls (730.55,144.84) and (733.84,151.94) .. (735.82,153.19) .. controls (738.35,154.79) and (745.59,153.19) .. (744.49,153.19) ;

% Text Node
\draw (608,382.4) node [anchor=north west][inner sep=0.75pt]  [font=\small]  {$2^{k-1} +2$};
% Text Node
\draw (711,288.4) node [anchor=north west][inner sep=0.75pt]  [font=\small]  {$W_{k}$};
% Text Node
\draw (518,132.4) node [anchor=north west][inner sep=0.75pt]  [font=\small]  {$W_{k-1}$};
% Text Node
\draw (659,336.4) node [anchor=north west][inner sep=0.75pt]  [font=\small]  {$\mathcal{S}_{R}$};
% Text Node
\draw (563,318.4) node [anchor=north west][inner sep=0.75pt]  [font=\small]  {$\gamma$};
% Text Node
\draw (651.6,91.6) node [anchor=north west][inner sep=0.75pt]  [font=\small]  {$\gamma'$};
% Text Node
\draw (705,155) node [anchor=north west][inner sep=0.75pt]  [font=\small]  {$a$};
% Text Node
\draw (659,151.4) node [anchor=north west][inner sep=0.75pt]    {$\mathcal{C}_{\gamma,k}$};
% Text Node
\draw (751.6,156) node [anchor=north west][inner sep=0.75pt]  [font=\small]  {$\gamma''$};
\draw (730,182) node [anchor=north west][inner sep=0.75pt]  [font=\footnotesize]  {$R(a)$};

\end{tikzpicture}
\vspace{-.9cm}
\caption{$\gamma$ is a path from $\partial W_{k-1}$ to $\partial W_k$, with its crosscut in $\mathcal{S}_R$ denoted by $\gamma'$, the thickened black path. The component of $\gamma'$ in $\mathcal{S}_R$ is denoted by $\mathcal{C}_{\gamma,k}$, with its lowest point $a$. The rectangle on the right is $[a(1) +2,a(1) + \sqrt{\theta} 2^{k-1}] \times [a(2) - \theta 2^{k-1}, a(2) + \theta 2^{k-1}]$, with a top to bottom crossing $\gamma''$, which is connected to $\mathcal{C}_{\gamma,k}$ by a path inside $S(a,\sqrt{\theta}2^{k-1})$. If in addition \eqref{e:separation} is satisfied, we say $\mathcal{C}_{\gamma,k}$ has a $(\theta,k)$-fence.}
\label{fig:path}
\end{figure}

\begin{remark}[Comparison to Kesten's definitions]{\rm
Note that, compared to \cite{Kes87}, we replace $2^{k-1} + 1$ with $2^{k-1}+2$ in the definition of $\mathcal{S}_R$, to have independence of the crossing events in $W_{k-1}$, which are measurable with respect to the restriction of the Poisson measure to $\overline W_{k-1}$.
Also, note that condition \eqref{i:fence} is a slight modification of equation (2.28) in \cite{Kes87}. In particular, we replace $S(a, \sqrt{\theta} 2^k)$ by $S(a, \sqrt{\theta} 2^{k-1})$. This ensures that for two different endpoints $a_1,a_2$ satisfying \eqref{e:separation}, the corresponding squares $S(a_1,\sqrt{\theta} 2^{k-1})$ and $S(a_2,\sqrt{\theta} 2^{k-1})$ are disjoint and are at least $\sqrt{\theta} 2^k$ apart. Since $\theta \in (0,1/4)$, and $k > 1- \log_2 \theta$, we have $\sqrt{\theta} 2^k> 2$, making crossing events in these squares stochastically independent.}
\end{remark}

To prove Theorem~\ref{thm1}, we first need a lemma (Lemma \ref{lem:del}), showing that the probability that there is a crossing of $\mathcal{S}_R$ without a $(\theta,k)$ fence can be made arbitrarily small for all $k$ large enough, by choosing $\theta$ appropriately. Then we show that the $4$-arm event $A_4(s,r)$, when additionally satisfying that all the 4 arms have a $(\theta,r)$-fence (the event $A_4^{\theta}(s,r)$ defined before Lemma \ref{lem:etaext}), can be extended to the event $\tilde{A}_4(s,r+2)$; see Lemma \ref{lem:etaext}. Finally, we use these two lemmas along with Lemma \ref{lem:extend} to complete the proof of Theorem~\ref{thm1} by deriving some recursive inequalities.
\begin{lemma}\label{lem:del}
	For each $\delta>0$, there exists $\theta=\theta(\delta) \in (0,1/4)$ such that for all integer $k %\ge k(\theta(\delta)) 
 >1- \log_2 \theta$ (equivalently $\theta 2^k > 2$), 
	\begin{multline*}
	\P_{\lambda_c}(\exists \text{ an occupied horizontal crossing $\gamma$ of $\mathcal{S}_R$ }\\
		\text{whose occupied component $\mathcal{C}$ does not have an $(\theta,k)$-fence}) \le \delta.
	\end{multline*}
 The above claim also holds for `occupied' replaced by `vacant' and $\mathcal{S}_R$ replaced by one of the strips $\mathcal{S}_L$, $\mathcal{S}_T$, or $\mathcal{S}_B$.
\end{lemma}

\begin{proof} The proof is very similar to that of \cite[Lem.\ 2]{Kes87}, so we will only give an outline of the proof and point out the relevant differences. We start by assuming that
\begin{multline}\label{e:ineqlowcross}
\P_{\lambda_c}\left(\exists \text{ a lowest occupied L-R crossing $\gamma_1$ of $\mathcal{S}_R$ not satisfying} \right.   \\
\left. \text{condition (\ref{i:fence}) in the definition of a $(\theta,k)$-fence}\right) \; \le \; (1-c_0^4)^{-C_0 \log \theta},
\end{multline}
for some constant $C_0 \in (0,\infty)$, where $c_0 = c_0(4)$ is defined at \eqref{eq:RSW}. Deferring the proof sketch of this inequality to the end, we will complete rest of the proof. 

Assume that $\gamma_i$, $1 \le i \le \sigma$, are given ordered occupied horizontal crossings of $\mathcal{S}_R$, $\gamma_1$ being the lowest, and $\gamma_\sigma$ begin the topmost. For a crossing $\gamma$, let $\mathcal{S}^+(\gamma)$ denote the region above $\gamma$ in $\mathcal{S}_R$. Repeating the argument in the derivation of \eqref{e:ineqlowcross} by successive conditioning, one can show that
\begin{multline*}
\P_{\lambda_c}\left( \exists \text{ a lowest occupied crossing $\gamma_{\sigma+1}$ of $\mathcal{S}_R$ in $\mathcal{S}^+(\gamma_\sigma)$ not satisfying} \right. \\ 
\left. \qquad \text{condition (\ref{i:fence}) in the definition of a $(\theta,k)$-fence} \, \mid \, \gamma_i, 1 \le i \le \sigma \right) \le (1-c_0(4)^4)^{-C_0 \log \theta}.
\end{multline*}
Also, by the Poisson BK inequality \cite[Theorem 4.1]{Gup1999}, there cannot be too many occupied left to right horizontal crossings of $\mathcal{S}_R$. Combining these, one can choose $\theta$ small depending on $\delta$ such that
\begin{multline*}
\P_{\lambda_c}\left( \exists \text{ an occupied L-R crossing of $\mathcal{S}_R$  not satisfying} \right. \\ 
\left. \qquad \text{condition (\ref{i:fence}) in the definition of a $(\theta,k)$-fence} \right) \le \delta/4.
\end{multline*}
Finally, we strengthen the above bound by refining the arguments involved in the derivation of \eqref{e:ineqlowcross}. This can be done as in the proof of \cite[Lem.\ 2]{Kes87} -- see the arguments below (2.33) on p.\ 125 of \cite{Kes87}.
  
\noindent To conclude, one needs to take $\theta = \theta(\delta) \in (0,1/4)$ small enough. Since, in our case, we additionally need certain spacial independence (see proof of \eqref{e:ineqlowcross} below) which occurs for $2^{-k} < \theta$, the conclusion follows for $k > - \log_2 \theta$.
\medskip

\noindent \underline{\textit{Proof of \eqref{e:ineqlowcross}}:}
This proof requires a more detailed knowledge of the proof of \cite[Lem.\ 2]{Kes87}. In particular, assuming that there exists an occupied horizontal crossing of $\mathcal{S}_R$, let $\gamma_1$ be the lowest such crossing with its endpoint on the right edge of $W_k$ being $a$, and define the events $E_j\equiv E_j(\eta, k, \gamma_1)$, $j \ge 1$, as in the proof of \cite[Lem.\ 2]{Kes87} (with a straightforward adaptation from the discrete to the continuum setting, see Figure \ref{fig:ej} for an illustration of the event). Unlike in \cite[Lem.\ 2]{Kes87}, since we have to take care of spatial dependence, we only consider the events $E_{2j-1}$ for $j \in \N$ as long as $2^{2j-1}\le \theta^{-1/2}/2$. 
This difference will only result in a variation of some constants. By taking alternative $E_j$'s, we make sure that they are independent. For this we need that the distance between the annuli $E_1$ and $E_3$ is larger than $2$ (this ensures independence of all other pairs of events $(E_{2j-1}, E_{2j+1})$ with $j \ge 2$), requiring $\theta 2^{k+1} > 2$, which is true by our assumption.

\begin{figure}
    \tikzset{every picture/.style={line width=0.75pt}} %set default line width to 0.75pt        

\begin{tikzpicture}[x=0.6pt,y=0.6pt,yscale=-1,xscale=1]
\path (-50,300); %set diagram left start at 0, and has height of 300

%Shape: Square [id:dp14488337429008835] 
\draw   (193,34) -- (468.5,34) -- (468.5,309.5) -- (193,309.5) -- cycle ;
%Shape: Square [id:dp16621669096238723] 
\draw   (264.5,105.5) -- (397,105.5) -- (397,238) -- (264.5,238) -- cycle ;
%Straight Lines [id:da0004969712168689799] 
\draw    (329.8,20) -- (331.15,326.87) ;
%Shape: Circle [id:dp08549829112963869] 
\draw  [fill={rgb, 255:red, 0; green, 0; blue, 0 }  ,fill opacity=1 ] (333.08,174.07) .. controls (333.08,172.79) and (332.03,171.75) .. (330.75,171.75) .. controls (329.47,171.75) and (328.43,172.79) .. (328.43,174.07) .. controls (328.43,175.36) and (329.47,176.4) .. (330.75,176.4) .. controls (332.03,176.4) and (333.08,175.36) .. (333.08,174.07) -- cycle ;
%Shape: Free Drawing [id:dp8730656578371963] 
\draw  [line width=0.75] [line join = round][line cap = round] (225.8,160) .. controls (225.8,156.07) and (220.94,150.42) .. (221.8,138.4) .. controls (223.47,115.07) and (231.28,101.01) .. (229,73.6) .. controls (228.25,64.6) and (224.64,51.72) .. (221.8,43.2) .. controls (220.69,39.88) and (216.2,36.76) .. (216.2,34.4) ;
%Shape: Free Drawing [id:dp27145052356845367] 
\draw  [line width=0.75] [line join = round][line cap = round] (194.6,57.6) .. controls (206.96,69.96) and (232.61,67.7) .. (249,67.2) .. controls (302.79,65.57) and (354.72,54.45) .. (409,56) .. controls (425.47,56.47) and (442.32,59.37) .. (456.2,64) .. controls (459.79,65.2) and (466.43,65.43) .. (469,68) ;
%Shape: Free Drawing [id:dp5602516831938456] 
\draw  [line width=0.75] [line join = round][line cap = round] (443.4,34.4) .. controls (443.4,34.44) and (434.05,57.19) .. (433,61.6) .. controls (430.08,73.87) and (430.11,88.53) .. (429.8,100.8) .. controls (429.65,106.82) and (429.65,139.34) .. (430.6,148.8) .. controls (433.6,178.76) and (441.61,207.13) .. (442.6,236.8) .. controls (443.22,255.36) and (441.71,280.89) .. (427.4,295.2) .. controls (424.1,298.5) and (419.69,304.25) .. (415.4,306.4) .. controls (413.85,307.17) and (409.6,308.8) .. (410.6,308.8) ;
%Shape: Free Drawing [id:dp30424080034730694] 
\draw  [line width=0.75] [line join = round][line cap = round] (469,261.6) .. controls (462.82,261.6) and (456.51,271.35) .. (449,273.6) .. controls (435.17,277.75) and (416.67,277.75) .. (401.8,275.2) .. controls (382.53,271.9) and (369.45,266.59) .. (350.6,265.6) .. controls (344.74,265.29) and (331.4,265.26) .. (331.4,268.8) ;
%Shape: Free Drawing [id:dp48639914415347096] 
\draw  [line width=0.75] [line join = round][line cap = round] (329.99,173.96) .. controls (324.03,172.43) and (312.39,192.43) .. (303.86,195.55) .. controls (288.15,201.3) and (270.3,196.72) .. (257.39,187.38) .. controls (240.67,175.29) and (231.07,160.3) .. (213.44,153.43) .. controls (207.96,151.29) and (195.1,147.92) .. (193.08,155.75) ;
%Shape: Boxed Bezier Curve [id:dp2720293381028658] 
%\draw [color={rgb, 255:red, 63; green, 114; blue, 214 }  ,draw opacity=1 ][line width=0.75]    (226.44,160.4) .. controls (226.44,165.08) and (228.01,169.74) .. (227.57,174.4) .. controls (226.74,183.21) and (219,187.7) .. (218.48,196.4) .. controls (218.12,202.39) and (218.12,208.41) .. (218.48,214.4) .. controls (218.97,222.54) and (231.42,229.51) .. (232.12,239.4) .. controls (232.47,244.39) and (232.5,249.41) .. (232.12,254.4) .. controls (231.84,258.13) and (225.73,262.86) .. (225.3,267.4) .. controls (224.95,271.05) and (224.92,274.75) .. (225.3,278.4) .. controls (225.43,279.67) and (228.02,309.4) .. (232.12,309.4) ;
%Shape: Free Drawing [id:dp07655744016001087] 
%\draw  [color={rgb, 255:red, 63; green, 114; blue, 214 }  ,draw opacity=1 ][line width=0.75] [line join = round][line cap = round] (193.8,270.07) .. controls (199.82,270.07) and (208.03,268.19) .. (214.8,270.07) .. controls (220.4,271.62) and (225.25,277.58) .. (230.8,277.84) .. controls (238.13,278.2) and (245.47,278.17) .. (252.8,277.84) .. controls (266,277.26) and (275.22,270.12) .. (287.8,268.96) .. controls (296.46,268.15) and (305.8,270.81) .. (313.8,272.29) .. controls (320.05,273.45) and (326.5,268.96) .. (330.8,268.96) ;
%Curve Lines [id:da6477301386171943] 
\draw    (495.8,44.4) .. controls (479.96,1.83) and (393.55,22.97) .. (337.49,26.3) ;
\draw [shift={(335.8,26.4)}, rotate = 356.93] [color={rgb, 255:red, 0; green, 0; blue, 0 }  ][line width=0.75]    (10.93,-3.29) .. controls (6.95,-1.4) and (3.31,-0.3) .. (0,0) .. controls (3.31,0.3) and (6.95,1.4) .. (10.93,3.29)   ;

% Text Node
\draw (236,182) node [anchor=north west][inner sep=0.75pt]    {$\gamma _{1}$};
% Text Node
\draw (337,169.4) node [anchor=north west][inner sep=0.75pt]    {$a$};
% Text Node
\draw (477,54) node [anchor=north west][inner sep=0.75pt]  [font=\small] [align=left] {right edge\\of $\mathcal{S}_R$};
% Text Node
\draw (330,240) node [anchor=north west][inner sep=0.75pt]  [font=\footnotesize]  {$S\left( a,\theta 2^{k+2j-2}\right)$};
% Text Node
\draw (200,314) node [anchor=north west][inner sep=0.75pt]  [font=\footnotesize]  {$S\left( a,\theta 2^{k+2j-1}\right)$};

\end{tikzpicture}
%\vspace{-.9cm}
\caption{Illustration of the event $E_{2j-1}$ for $j \in \N$. $\gamma_1$ is the lowest occupied horizontal crossing of $\mathcal{S}_R$. Given $\gamma_1$, the event $E_{2j-1}$ requires further occupied paths as shown here. If $E_{2j-1}$ occurs for some $j \in \N$ with $2^{2j-1}\le \theta^{-1/2}/2$, then condition \eqref{i:fence} in the definition of $(\theta,k)$-fence is satisfied.}
\label{fig:ej} 
\end{figure}
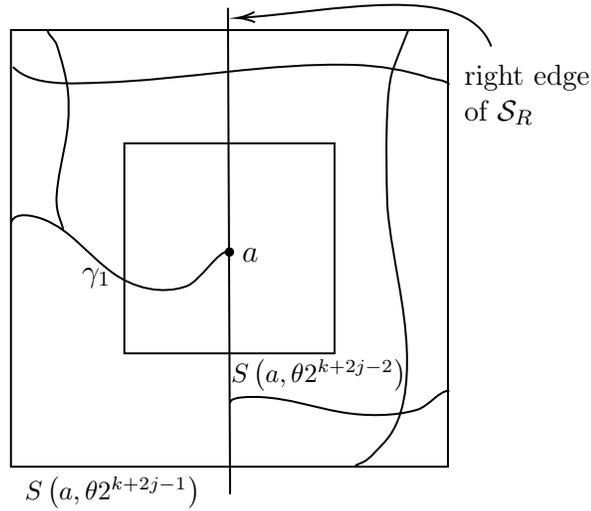

In particular we have that, conditionally on the event that the lowest occupied horizontal crossing in $\mathcal{S}_R$ is $\gamma_1$ (which provides no information on the configuration in the region in $\mathcal{S}_R$ which is unit distance above $\gamma_1$, and is positively associated with any occupied event in that unit strip), the probability that $E_j$ occurs is at least $c_0(4)^4$. We note here that the lowest crossing can be defined in a measurable way, as the lower boundary of the lowest connected component in the occupied region in $\mathcal{S}_R$ that connects its left and right boundaries. We need to be a bit careful here, since we are conditioning on an event with probability zero, but one can define it concretely by considering the Hausdorff metric on the space of piecewise smooth curves crossing $\mathcal{S}_R$ from left to right, which defines a Borel $\sigma$-algebra on this space. This makes the event that the lowest crossing is in the ball (in Hausdorff metric) of radius $\epsilon>0$ around a given fixed curve $\gamma_1$ measurable in our probability space. Then one can let $\epsilon$ go to zero to make sense of such a conditioning. Also, note that if $E_{2j-1}$ is true for some $j \in \N$ with $2^{2j-1}\le \theta^{-1/2}/2$, then \eqref{i:fence} in the definition of $(\theta,k)$-fence holds. The above upper bound now follows by arguing as in \cite[Lem.\ 2]{Kes87}.
\end{proof}

For the proof of Theorem~\ref{thm2}, we also need an inward version of Lemma~\ref{lem:del}. Again, for $\theta \in (0,1/4)$, and $k > 1- \log_2 \theta$, let $\gamma$ be a path from $\partial W_{k+1}$ to $\partial W_k$ (see also \cite[p.\ 134]{Kes87}). Assuming it hits the right boundary of $\partial W_k$, recall the definitions of its component $\mathcal{C}_{\gamma,k}$ in the $C$-shaped region $\mathcal{\widetilde S}_R := [0,2^{k+1}-2] \times [-2^{k+1}+2,2^{k+1}-2] \setminus W_k$ (again with $2^{k+1} - 1$ replaced by $2^{k+1} - 2$ in the definition of $ \mathcal{\widetilde S}_R$ in \cite{Kes87}) with outer and inner boundaries denoted by $B_1$ and $B_2$ respectively (See \cite[Figure 13]{Kes87}). Note that $B_2 = \{2^k\} \times [-2^k,2^k]$ is also the right boundary of $\partial W_k$. We denote the lowest point of  $\mathcal{C}_{\gamma,k}$ on $B_2$ by $a=a(\mathcal{C})$.

The component $\mathcal{C}$ of the path $\gamma$ is said to have an {\it inward $(\theta, k)$-fence}, if
\begin{enumerate}[(i)]
    \item for another path (occupied or vacant) $\gamma_0$ from $\partial W_{k+1}$ to $\partial W_k$ with $\mathcal{C}_{\gamma,k} \cap \mathcal{C}_{\gamma_0,k} = \emptyset$ (again holds trivially if $\gamma_0$ is of opposite type), 
    \begin{equation*}\label{e:inseparation}
    |a(\mathcal{C}_{\gamma,k}) - a(\mathcal{C}_{\gamma_0,k})|>2\sqrt{\theta} 2^k,
    \end{equation*}
    \item there is a vertical crossing of the same type as $\gamma$ of the rectangle
    \begin{equation*}\label{e:infence}
        [a(1) - \sqrt{\theta} 2^{k-1},a(1) -2] \times [a(2) - \theta 2^{k-1}, a(2) + \theta 2^{k-1}],
    \end{equation*}
    which is connected to $\mathcal{C}_{\gamma,k}$ by a path of the same type in $S(a,\sqrt{\theta}2^{k-1})$.
\end{enumerate}
As before, in contrast to the definition of inward $(\theta, k)$-fence in \cite{Kes87}, in condition (2.54) therein, we replace $S(a, \sqrt{\theta} 2^k)$ by $S(a, \sqrt{\theta} 2^{k-1})$. We can also similarly define the above events with respect to vacant component and left, top and bottom boundaries of $\partial W_k$.

Now we state an inward version of Lemma~\ref{lem:del}. Since its proof is very similar, we skip the proof here; see also the proof of \cite[Lem.\ 5]{Kes87}.
\begin{lemma}\label{lem:del'}
	For each $\delta>0$, there exists $\theta=\theta(\delta) \in (0,1/4)$ such that for all $k \ge k(\theta) >1-\log_2 \theta$,
	\beaa
	&&\P_{\lambda_c}(\exists \text{ an occupied horizontal crossing $\gamma$ from $B_1$ to $B_2$ in $\mathcal{\widetilde S_R}$ }\\
	&& \qquad \text{whose occupied component $\mathcal{C}$ does not have an inward $(\theta,k)$-fence}) \le \delta.
	\eeaa
 The above statement also holds for `occupied' replaced by `vacant', and $\mathcal{\widetilde S}_R$ replaced by one of the similarly defined strips $\mathcal{\widetilde S}_L$, $\mathcal{\widetilde S}_T$, or $\mathcal{\widetilde S}_B$.
\end{lemma}

For the proof of Theorem~\ref{thm1}, we also need the following result from \cite[Lem.\ 3]{Kes87}, adapted to the Poisson setting. The proof uses the FKG inequality and complete independence property of the Poisson process.
\begin{lemma}\label{lem:FKG}
Let $A$ and $D$ be increasing events, and $B$ and $E$ decreasing events and $\P$ is a Poisson measure. Assume also that the events $A,B,D,E$ depend on the configuration in the sets $\mathcal{A}, \mathcal{B}, \mathcal{D}, \mathcal{E}$ respectively, with 
\begin{equation}\label{eq:FKGcond}
\mathcal{A} \cap \mathcal{B} = \mathcal{A} \cap \mathcal{E} = \mathcal{B} \cap \mathcal{D} = \emptyset.
\end{equation}
Then
	\begin{equation}\label{e:FKG}
	\P(A \cap B \mid D \cap E) \ge \P(A \cap B).
	\end{equation}
\end{lemma}

Now we are ready to prove Theorem~\ref{thm1} following the ideas in the proof of \cite[Lem.\ 4]{Kes87}. To make the proof easier to follow, we break it into two parts, first showing the following lemma which is equivalent to (2.38) in \cite{Kes87}, and then proving Theorem~\ref{thm1}. For $\theta \in (0,1)$, denote $\alpha_4^\theta(s,r)=\P_{\lambda_c}(A_4^\theta(s,r))$, where we define the event $A_4^\theta(s,r)$ as
$$
A_4^\theta(s,r) = A_4(s,r) \cap \{\text{all 4 components can be chosen to have an $(\theta,r)$-fence}\}.
$$
\begin{lemma}\label{lem:etaext} For $\theta \in (0,1/4)$ there exists a constant $C_{\theta} \in (0,\infty)$, such that for $r$ large enough so that $\theta 2^r > 2$, and $s\le r$,
	$$\alpha_4^\theta(s,r) \le C_\theta \widetilde \alpha_4 (s,r+2).$$
\end{lemma}
\begin{proof} The proof again is similar as in \cite[Lem.\ 4]{Kes87}, and follows from the following two steps:
\medskip

\underline{\textsc{STEP 1}:} Assume 
$A_4(s,r)$ occurs so that there are occupied paths $\gamma_1,\gamma_3$ and vacant paths $\gamma_2^*,\gamma_4^*$ with corresponding components $\mathcal{C}_1, \mathcal{C}_3$ and $\mathcal{C}_2^*, \mathcal{C}_4^*$, and terminal points $a_1,a_3,a_2^*,a_4^*$, respectively. Partitioning $\partial W_r$ into disjoint intervals of lengths $\le \theta 2^r$, by the pigeonhole principle, we can pick intervals (landing sequence) $I_1,I_3,I_2^*, I_4^*$ of lengths $\le \theta 2^r$ such that there exists $C_0 \in (0,\infty)$ with
\begin{equation}\label{e:land}
\P_{\lambda_c}(A_4^\theta(s,r,I)) \ge C_0 \theta^4 \P_{\lambda_c}(A_4^\theta(s,r)),
\end{equation}
where 
$$
A_4^\theta(s,r,I):=A_4^\theta(s,r) \text{ with $a_i \in I_i$, and $a_{i+1} \in I_{i+1}^*$}, \, i=1,3.
$$
We fix this landing sequence $I$ in the rest of the proof. Also define
\beaa
\Theta(s,r,I) &:=& A_4(s,r) \text{ with $a_i \in I_i$, and $a_{i+1} \in I_{i+1}^*$, $i=1,3$,}\\
&& \text{ and the components satisfy \eqref{e:separation} with $k$ replaced by $r$}.
\eeaa
Then notice that under $\Theta$, the intervals $I_1,I_3,I_2^*, I_4^*$ are all different as $2\theta 2^r \le 2\sqrt{\theta} 2^r$ and distinct $a_i,a_i^*$'s are at least $2\sqrt{\theta} 2^r$ apart by definition. Moreover, $S(a_i,\sqrt{\theta} 2^{r-1})$ and $S(a_{i+1}^*,\sqrt{\theta} 2^{r-1})$ for $i=1,2$ are not only disjoint, but also at least at a distance $2$ away, since their distances are at least $2\sqrt{\theta}2^r - 2\sqrt{\theta} 2^{r-1} >2$ by our assumption. Thus, events in these squares are indeed independent.
\medskip

\underline{\textsc{STEP 2}:} Next, we use four disjoint corridors $K_1,\ldots,K_4$, that always stay at least at a distance $2$ apart, to connect the paths hitting the intervals $I_1,I_2,I_3^*$ and $I_4^*$ to $\partial W_{r+2}$, so that the parts of the corridors outside $W_{r+1}$ stay in the regions $\mathcal{A}(1,r+2)$ and $\mathcal{A}(3,r+2)$ and $\mathcal{B}(2,r+2)$ and $\mathcal{B}(4,r+2)$. Further, for $i=1,3$, recall from \eqref{e:fence} that $R(a_i)$ is the $(\sqrt{\theta} 2^{r-1} -2)  \times (2\theta 2^{r-1})$ rectangle outside $\overline W_r$, with $a_i$ shifted by $2$ and in the same direction as the side of $\partial W_r$ containing $a_i$, in the center of its short side; see also Figure \ref{fig:path}. Similarly, we have rectangles $R(a_{i+1}^*)$ with respect to $a_{i+1}^*$ for $i=1,3$.

Note that given that $\Theta(s,r,I)$ occurs, the event $A_4^\theta(s,r,I)$ occurs only if, additionally, condition (\ref{i:fence}) in the definition of $(\eta,k)$-fence is satisfied for the four paths, i.e., the following occurs for $i=1,3$ (we take $\theta<1/4$ so that $2\theta<\sqrt{\theta}$
):
\beaa
		D_i &=& \exists\text{ occupied crossing in the short direction of the rectangle $R(a_i)$ and} \\
  && \text{this crossing is connected to $\mathcal{C}_i$ by an occupied path in $S(a_i,\sqrt{\theta}2^{r-1})$,}
	\eeaa
 and
 \beaa
		E_{i+1} &=& \exists\text{ vacant crossing in the short direction of the rectangle $R(a_{i+1}^*)$ and} \\
			&&\text{this crossing is connected by a vacant path in $S(a_{i+1}^*,\sqrt{\theta}2^{r-1})$ to $\mathcal{C}_{i+1}^*$.}
	\eeaa

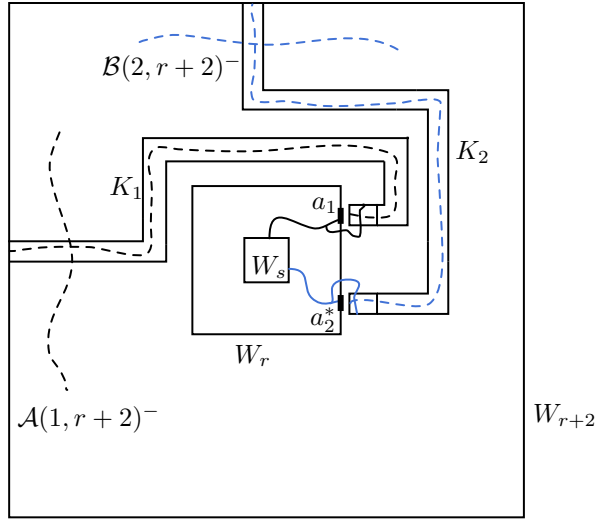
\begin{figure}
\tikzset{every picture/.style={line width=0.75pt}} %set default line width to 0.75pt        

\begin{tikzpicture}[x=0.55pt,y=0.55pt,yscale=-1,xscale=1]
\path (110,443); %set diagram left start at 0, and has height of 443

%Shape: Square [id:dp13986294871193272] 
\draw   (363,55) -- (717,55) -- (717,409) -- (363,409) -- cycle ;
%Shape: Square [id:dp4886365372940009] 
\draw   (489,181) -- (591,181) -- (591,283) -- (489,283) -- cycle ;
%Straight Lines [id:da9633170370832116] 
\draw [line width=2.25]    (591,196) -- (591,207) ;
%Shape: Square [id:dp5765794410766389] 
\draw   (524.75,216.75) -- (555.25,216.75) -- (555.25,247.25) -- (524.75,247.25) -- cycle ;
%Shape: Free Drawing [id:dp2171134890183879] 
\draw  [line width=0.75] [line join = round][line cap = round] (542,217) .. controls (542,191.56) and (557.51,207.1) .. (572,210) .. controls (579.33,211.47) and (585.63,204) .. (592,204) ;
%Shape: Free Drawing [id:dp11085267873560412] 
\draw  [color={rgb, 255:red, 63; green, 114; blue, 214 }  ,draw opacity=1 ][line width=0.75] [line join = round][line cap = round] (555.52,238) .. controls (571.8,238) and (563.94,251.12) .. (571.07,258) .. controls (578.76,265.41) and (586.16,260) .. (591.81,260) ;
%Shape: Rectangle [id:dp4232342834009766] 
\draw   (597,194) -- (616,194) -- (616,208) -- (597,208) -- cycle ;
%Straight Lines [id:da4420427308006618] 
\draw [line width=2.25]    (591,256) -- (591,267) ;
%Shape: Rectangle [id:dp8324618887695499] 
\draw   (597,255) -- (616,255) -- (616,269) -- (597,269) -- cycle ;
%Shape: Free Drawing [id:dp026557161907676763] 
\draw  [color={rgb, 255:red, 0; green, 0; blue, 0 }  ,draw opacity=1 ][line width=0.75] [line join = round][line cap = round] (581.16,208.73) .. controls (581.16,216.38) and (591.53,210.43) .. (595.43,211.55) .. controls (596.55,211.86) and (598.27,214.97) .. (599.82,215.3) .. controls (601.24,215.61) and (603.91,216.53) .. (604.22,215.3) .. controls (605.65,209.48) and (603.85,203.4) .. (604.22,197.45) .. controls (604.22,197.39) and (609.71,193.7) .. (606.41,193.7) ;
%Shape: Free Drawing [id:dp17503502030610174] 
\draw  [color={rgb, 255:red, 65; green, 114; blue, 214 }  ,draw opacity=1 ][line width=0.75] [line join = round][line cap = round] (583.67,261.67) .. controls (587.03,261.67) and (584.58,255) .. (585,251.67) .. controls (585.79,245.36) and (595.29,244.79) .. (599.67,245.67) .. controls (608.7,247.47) and (599.66,253.96) .. (600.33,258.33) .. controls (600.64,260.35) and (602.33,269.1) .. (602.33,269) ;
%Shape: Right Angle [id:dp10172574856613048] 
\draw   (616,208) -- (637,208) -- (637,147) ;
%Shape: Right Angle [id:dp8000307988499851] 
\draw   (615,194) -- (621,194) -- (621,164) ;
%Shape: Right Angle [id:dp4260958601017608] 
\draw   (621,164) -- (471,164) -- (471,221) ;
%Shape: Right Angle [id:dp11088921199824164] 
\draw   (638,148) -- (455,148) -- (455,203) ;
%Shape: Right Angle [id:dp3043378788813609] 
\draw   (363,233) -- (471,233) -- (471,221) ;
%Shape: Right Angle [id:dp9583253146145065] 
\draw   (362.33,219) -- (455,219) -- (455,203) ;
%Shape: Right Angle [id:dp2810575957032655] 
\draw   (616,269) -- (665,269) -- (665,233) ;
%Shape: Right Angle [id:dp9693969219277412] 
\draw   (616,255) -- (651,255) -- (651,219) ;
%Shape: Right Angle [id:dp06774056914119031] 
\draw   (643.33,115) -- (665,115) -- (665,233) ;
%Shape: Right Angle [id:dp1519991604629829] 
\draw   (629.33,128.33) -- (651,128.33) -- (651,219) ;
%Shape: Right Angle [id:dp19645746279669973] 
\draw   (643.33,115) -- (537.67,115) -- (537.67,54.33) ;
%Shape: Right Angle [id:dp7680138198861626] 
\draw   (629.33,128.33) -- (523.67,128.33) -- (523.67,54.33) ;
%Shape: Free Drawing [id:dp2193904678116858] 
\draw  [dashed, color={rgb, 255:red, 65; green, 114; blue, 214 }  ,draw opacity=1 ][line width=0.75] [line join = round][line cap = round] (597.67,261.73) .. controls (597.67,258.13) and (610.42,259.65) .. (611.67,259.75) .. controls (623.68,260.74) and (645.9,266.82) .. (652.33,261.73) .. controls (660.85,254.99) and (661.3,235.16) .. (661,225.45) .. controls (660.34,204.55) and (658.97,183.67) .. (657.67,162.79) .. controls (657.25,156.15) and (660.17,123.96) .. (659.67,123.22) .. controls (657.28,119.68) and (646.41,121.94) .. (645.67,121.9) .. controls (640.93,121.67) and (635.37,123.62) .. (630.33,123.88) .. controls (613.79,124.74) and (596.18,126.52) .. (579.67,125.86) .. controls (570.28,125.48) and (561.05,123.14) .. (551.67,122.56) .. controls (548.48,122.36) and (531.94,125.34) .. (530.33,120.58) .. controls (527.95,113.5) and (531.32,107.96) .. (531.67,100.8) .. controls (532.14,90.91) and (532.45,80.98) .. (531.67,71.12) .. controls (531.27,66.06) and (528.04,60.49) .. (530.33,55.95) ;
%Shape: Free Drawing [id:dp06209407360536545] 
\draw  [dashed, color={rgb, 255:red, 0; green, 0; blue, 0 }  ,draw opacity=1 ][line width=0.75] [line join = round][line cap = round] (597.53,200.33) .. controls (597.81,200.36) and (607.48,202.3) .. (608.24,202.33) .. controls (608.47,202.34) and (621.79,204.06) .. (625.63,201) .. controls (631.6,196.25) and (627.86,185.93) .. (628.31,178.33) .. controls (628.92,168.02) and (633.51,161.17) .. (622.96,157.67) .. controls (593.81,147.99) and (560.17,156.32) .. (531.29,157) .. controls (513.74,157.42) and (496.54,156.63) .. (479.11,155.67) .. controls (477.34,155.57) and (468.17,153.49) .. (465.05,154.33) .. controls (457.89,156.28) and (458.3,166.74) .. (459.7,172.33) .. controls (462.59,183.86) and (460.37,194.14) .. (460.37,208.33) .. controls (460.37,213.45) and (462.64,222.74) .. (458.36,227) .. controls (453.03,232.31) and (425.81,223.18) .. (418.22,223) .. controls (409.08,222.78) and (399.92,222.65) .. (390.79,223) .. controls (382.71,223.31) and (374.67,224.34) .. (366.7,225.67) .. controls (365.6,225.85) and (362.24,225.67) .. (363.35,225.67) ;
%Shape: Free Drawing [id:dp12215177428929835] 
\draw  [dashed, color={rgb, 255:red, 0; green, 0; blue, 0 }  ,draw opacity=1 ][line width=0.75] [line join = round][line cap = round] (397.67,144.01) .. controls (393.91,153.44) and (390.03,161.06) .. (391,170.67) .. controls (393.03,190.64) and (406.33,213.27) .. (406.33,232.09) .. controls (406.33,255.7) and (386.53,268.23) .. (390.33,294.67) .. controls (391.24,300.94) and (397.01,305.12) .. (399,309.73) .. controls (399.77,311.51) and (401.89,315.54) .. (402.33,317.26) .. controls (402.67,318.6) and (401.43,321.32) .. (403,321.32) ;
%Shape: Free Drawing [id:dp17649030744982608] 
\draw  [dashed, color={rgb, 255:red, 65; green, 114; blue, 214 }  ,draw opacity=1 ][line width=0.75] [line join = round][line cap = round] (630.14,87) .. controls (605.53,72.31) and (571.5,83.01) .. (545.25,83.67) .. controls (530.69,84.03) and (516.16,81.62) .. (501.69,79.67) .. controls (490.94,78.21) and (479.98,75.7) .. (469.3,77.67) .. controls (467.16,78.06) and (449.19,81.33) .. (449.19,84.33) ;

% Text Node
\draw (366.67,329.4) node [anchor=north west][inner sep=0.75pt]  [font=\footnotesize]  {$\mathcal{A}( 1,r+2)^{-}$};
% Text Node
\draw (425,88.07) node [anchor=north west][inner sep=0.75pt]  [font=\footnotesize]  {$\mathcal{B}(2,r+2)^{-}$};
% Text Node
\draw (569,187.07) node [anchor=north west][inner sep=0.75pt]  [font=\footnotesize]  {$a_{1}$};
% Text Node
\draw (568,262.4) node [anchor=north west][inner sep=0.75pt]  [font=\footnotesize]  {$a_2^*$};
% Text Node
%\draw (595,275) node [anchor=north west][inner sep=0.75pt]  [font=\footnotesize]  {$R(a_2^*)$};
% Text Node
\draw (527.33,227) node [anchor=north west][inner sep=0.75pt]  [font=\footnotesize]  {$W_{s}$};
% Text Node
\draw (516,288.07) node [anchor=north west][inner sep=0.75pt]  [font=\footnotesize]  {$W_{r}$};
% Text Node
\draw (721.33,328.73) node [anchor=north west][inner sep=0.75pt]  [font=\footnotesize]  {$W_{r+2}$};
% Text Node
\draw (430,172.73) node [anchor=north west][inner sep=0.75pt]  [font=\footnotesize]  {$K_{1}$};
% Text Node
\draw (668,148.07) node [anchor=north west][inner sep=0.75pt]  [font=\footnotesize]  {$K_{2}$};

\end{tikzpicture}
\vspace{-.7cm}
\caption{An illustration of the extension from $W_r$ to $W_{r+2}$ using an $(\theta,r)$-fence. The dashed paths constitute the events $A_1$ and $B_2$ here. The corridors $K_1$ and $K_2$ are of width $\theta 2^r$ and they always stay at a distance of at least 2 units apart.}
\label{fig:corridors}
\end{figure}

 Then, we apply \eqref{e:FKG} to the complement of $\overline W_r$, given a configuration inside $\overline W_r$ such that $\Theta(s,r,I)$ occurs. Without loss of generality, assume that $a_1,a_2^*,a_3,a_4^*$ appear on $\partial W_r$ in clockwise order. Define
  \beaa
		A_1 &=& \exists\text{ occupied path $\tilde \gamma_1$ outside $\overline W_r$ starting next to the side of $R(a_1)$ containing $a_1$,} \\
			&&\text{and ending on the left edge of $W_{r+2}$ inside a given corridor $K_1$ of width $\theta 2^{r}$;}\\
			&&\text{the part of $K_1$ outside of $W_{r+1}$ is in $\mathcal{A}(1,r+2)$; in addition, there exists an}\\
			&&\text{occupied vertical crossing of $\mathcal{A}(1,r+2)^-$.}
	\eeaa
We can similarly define $A_3$ for $a_3$ with corridor $K_3$, and let $A=A_1 \cap A_3$. Also, we can define $B_2$ and $B_4$ in very much the same way with $a_2^*,K_2$ and $a_4^*,K_4$ respectively, replacing occupied by vacant, see proof of \cite[Lem.\ 3]{Kes87} for more details. Let $B=B_2 \cap B_4$. Finally, we take $D=D_1 \cap D_3$ and $E=E_2\cap E_4$. 

Now given a configuration $\eta_r$ in $\overline W_{r}$ such that $\Theta(s,r,I)$ occurs (note that this specifies $a_i$ and $a_{i+1}^*$ for $i=1,3$), we will apply Lemma \ref{lem:FKG} to the Poisson measure $\P_{\lambda_c}^{\overline W_r^c}$ on the complement of $\overline W_r$. Note that, given $\eta_r$, our events $A,D$ are increasing, while $B,E$ are decreasing, as a function of the configuration in $\overline W_r^c$. In addition, conditionally, we also have that the assumption in \eqref{eq:FKGcond} is satisfied. Thus, by Lemma \ref{lem:FKG}, 
\beaa
\P_{\lambda_c}^{\overline W_r^c} (A \cap B \mid \{\eta(\overline W_r) = \eta_r\} \cap D \cap E) &\ge& \P_{\lambda_c}^{\overline W_r^c}(A \cap B \mid 
\{\eta(\overline W_r) = \eta_r\}) \\
&=& \P_{\lambda_c}^{\overline W_r^c}(A \mid a_1,a_3)\P_{\lambda_c}^{\overline W_r^c}(B \mid a_2^*, a_4^*) \ge C_1
\eeaa
for some constant $C_1 \in (0,\infty)$ depending only on $\theta$, where in the penultimate step, we have used that the event $A \cap B$, given $\eta(\overline W_r) = \eta_r$, depends only on $a_1,a_3,a_2^*$ and $a_4^*$ ($A,B$ being events measurable w.r.t.\ $\overline W_{r}^c$), and the two events are independent due to the disjointness of their corridors given $a_1,a_3,a_2^*$ and $a_4^*$. Indeed, the events $A,B$ depend on the Poisson process in the $1$-neighbourhood of corridors which are at least distance $2$ apart. The final step is due to the RSW bound \eqref{eq:RSW}.

Now notice that $\{\eta(\overline W_r) = \eta_r\} \cap A \cap B \cap D \cap E \subseteq \widetilde A_4 (s,r+2)$. The result now follows upon integrating the above inequality over $\Theta$ w.r.t.\ the Poisson measure in $\overline W_r$ and noting also that $A_4^\theta(s,r,I) \subseteq \Theta \cap D \cap E$, along with \eqref{e:land}.
\end{proof}

\begin{proof}[Proof of Theorem~\ref{thm1}] Note that one trivially has that $\tilde{\alpha}_4(s,r) \leq \alpha_4(s,r)$ and hence we only need to show that $\alpha_4(s,r) \lesssim \tilde{\alpha}_4(s,r)$ .

Fix $\delta \in (0,1/8 \wedge (32C^2)^{-1})$ for $C>1$ as in Lemma~\ref{lem:extend}, and let $\theta=\theta(\delta)$ and $k$ be as in Lemma~\ref{lem:del}. Also, fix $r \ge 2 (k \vee 4)$ and $s \le r$. We have
	\beaa
	\alpha_4(s,r)&\le& \alpha_4(s,r-2)\\
	&=& \alpha_4^\theta(s,r-2) + \P_{\lambda_c}(A_4(s,r-2) \setminus A_4^\theta (s,r-2)).
	\eeaa
	Observe (see also \cite[p.\ 131]{Kes87}) that the occurrence of $A_4(s,r-2) \setminus A_4^\theta (s,r-2)$ implies that $A_4(s,r-3)$ (determined by $\overline W_{r-3}$) occurs and, the following event - determined by $\overline W_{r-3}^c$ - occurs: in one of the right/left/top/bottom sides, there exists a occupied/vacant crossing of $\mathcal{S}_R$/$\mathcal{S}_L$/$\mathcal{S}_T$/$\mathcal{S}_B$ whose component does not have an $(\theta,k)$-fence. Each of these eight events (choices of occupied or vacant for each of the $4$ sides)  has probability at most $\delta$ due to Lemma~\ref{lem:del} and hence, the union has probability at most $8\delta$. Thus, using independence of the events, we obtain
	$$
	\alpha_4(s,r-2)\le \alpha_4^\theta(s,r-2) + 8\delta \,\alpha_4(s,r-3).
	$$
	We can iterate this argument for $\alpha_4(s,j)$ as long as 
	$j > s \vee k \vee 4=:s_1$.
	Thus, with $l = \frac{r}{2} \vee s_1$, we obtain that
	\beaa
	\alpha_4(s,r)&\le& \sum_{j=l+1}^{r-2} (8\delta)^{r-2-j} \alpha_4^\theta(s,j) + (8\delta)^{r-l-2} \alpha_4(s,l) \\
    &\le& \sum_{j=l+1}^{r-2} (8\delta)^{r-2-j} \alpha_4^\theta(s,j) + (8\delta)^{(\frac{r}{2} \wedge (r-s))-3} \alpha_4(s,s \vee 4) \\
	&\le&  \sum_{j=l+1}^{r-2} (8\delta)^{r-2-j} \alpha_4^\theta(s,j) + C_0 (8\delta)^{(\frac{r}{2} \wedge (r-s))-3}\widetilde \alpha_4(s,s \vee 4),
	\eeaa
 where in the final step, we have used the trivial fact that there exists a constant $C_0 \in (0,\infty)$ such that $\alpha_4(s,s \vee 4) \le C_0 \widetilde \alpha_4(s,s \vee 4)$.
	Now using Lemmas~\ref{lem:etaext} and \ref{lem:extend}, we obtain
	$$
	\alpha_4(s,r) \le \sum_{j=l+1}^{r-2} (8\delta)^{r-2-j} C_\theta C^{r-2-j} \widetilde \alpha_4(s,r) + (8\delta)^{(\frac{r}{2} \wedge (r-s))-3} C_0 C^{r-(s \vee 4)} \widetilde \alpha_4(s,r),
	$$
	with $C_{\theta}$ and $C$ as in  Lemmas~\ref{lem:etaext} and \ref{lem:extend}. The result now follows by our choice of $\delta$.
\end{proof}

\begin{proof}[Proof of Theorem~\ref{thm2}] Using Lemmas~\ref{lem:del'} and \ref{lem:extendin}, and an appropriate inward version of Lemma~\ref{lem:etaext}, arguing similarly as in the proof of Theorem \ref{thm1} above yields the result.
\end{proof}

\subsubsection{Proof of Theorem \ref{thm3}}
\label{ss:proof_thm3}
Finally we prove Theorem \ref{thm3} which is the final ingredient for the proof of Theorem \ref{t:quasimPBM} presented above.
\begin{proof}[Proof of Theorem~\ref{thm3}] First note that since $r_1 = (s_0+2) \vee r_0$, the result follows if $s=r$, or $t=r$ from Theorems~\ref{thm2} and \ref{thm1}, respectively. Similarly, for $i=1,2$ and $s=r-i \ge s_0 \ge 4$, we have by Lemma~\ref{lem:extendin} and Theorem~\ref{thm2} that
$$
\dbtilde \alpha_4(s+i,t) \le C^2 \dbtilde \alpha_4(s,t) \asymp \alpha_4(s,t).
$$
yielding the result. So we may assume that $t>r > s+2$, and we will show that for all $t > r > s+2$,
	\begin{equation*}\label{eq:qmtil}
		\alpha_4(s,t) \ge C_0 \widetilde \alpha_4(s,r-2) \dbtilde \alpha_4(r,t)
	\end{equation*}
	for some $C_0 \in (0,\infty)$, which will yield the desired result. 
 
	Notice that $\widetilde A_4(s,r-2)$ and $\dbtilde A_4(r,t)$ are independent, since they are events in annuli that are at least at a distance of two units apart. Define a new event $B(r-2,r)$ as (see Figure \ref{Brr})
	\beaa
	B(r-2,r) &=&\text{there are occupied horizontal crossings of }\\
	&&\text{$[2^{r-3}+2, 2^{r+1}-2] \times [-2^{r-3}+2, 2^{r-3}-2]$ } \\
	&&\text{and $[-2^{r+1}+2,-2^{r-3}-2] \times [-2^{r-3}+2, 2^{r-3}-2]$, and similarly }\\
	&&\text{there are vacant vertical crossings of }\\
	&&\text{$[-2^{r-3}+2, 2^{r-3}-2] \times [2^{r-3}+2, 2^{r+1}-2]$ and }\\
	&&\text{$[-2^{r-3}+2, 2^{r-3}-2] \times [-2^{r+1}+2,-2^{r-3}-2]$}.
	\eeaa

\begin{figure}
    \tikzset{every picture/.style={line width=0.75pt}} %set default line width to 0.75pt        
\begin{tikzpicture}[x=0.5pt,y=0.5pt,yscale=-1,xscale=1]
\path (-45,494); %set diagram left start at 0, and has height of 494

%Shape: Square [id:dp7509078966622402] 
\draw   (265,86) -- (588,86) -- (588,409) -- (265,409) -- cycle ;
%Shape: Square [id:dp8393076050459758] 
\draw   (327.73,148.73) -- (523,148.73) -- (523,344) -- (327.73,344) -- cycle ;
%Shape: Square [id:dp9981814740011492] 
\draw   (357.83,178.83) -- (495.17,178.83) -- (495.17,316.17) -- (357.83,316.17) -- cycle ;
%Shape: Square [id:dp012976075553980948] 
\draw   (288.34,109.34) -- (562.39,109.34) -- (562.39,383.39) -- (288.34,383.39) -- cycle ;
%Straight Lines [id:da3553362510969611] 
\draw [color={rgb, 255:red, 155; green, 155; blue, 155 }  ,draw opacity=1 ]   (495,86.33) -- (495.17,178.83) ;
%Straight Lines [id:da06497101919979831] 
\draw [color={rgb, 255:red, 155; green, 155; blue, 155 }  ,draw opacity=1 ]   (495.17,178.83) -- (587.67,179) ;
%Straight Lines [id:da3144181418465346] 
\draw [color={rgb, 255:red, 155; green, 155; blue, 155 }  ,draw opacity=1 ]   (357.67,86.33) -- (357.83,178.83) ;
%Straight Lines [id:da46349680908908697] 
\draw [color={rgb, 255:red, 155; green, 155; blue, 155 }  ,draw opacity=1 ]   (495.17,316.17) -- (587.67,316.33) ;
%Straight Lines [id:da1843812730175738] 
\draw [color={rgb, 255:red, 155; green, 155; blue, 155 }  ,draw opacity=1 ]   (265.33,316) -- (357.83,316.17) ;
%Straight Lines [id:da767176534370557] 
\draw [color={rgb, 255:red, 155; green, 155; blue, 155 }  ,draw opacity=1 ]   (265.33,178.67) -- (357.83,178.83) ;
%Straight Lines [id:da1392099576247945] 
\draw [color={rgb, 255:red, 155; green, 155; blue, 155 }  ,draw opacity=1 ]   (357.83,316.17) -- (358,408.67) ;
%Straight Lines [id:da2875556406558495] 
\draw [color={rgb, 255:red, 155; green, 155; blue, 155 }  ,draw opacity=1 ]   (495.17,316.17) -- (495.33,408.67) ;
%Shape: Rectangle [id:dp1932746682682691] 
\draw   (372.33,98.33) -- (481,98.33) -- (481,165.67) -- (372.33,165.67) -- cycle ;
%Shape: Rectangle [id:dp6704905853429088] 
\draw   (372.33,329) -- (481,329) -- (481,396.33) -- (372.33,396.33) -- cycle ;
%Shape: Rectangle [id:dp017829763852980296] 
\draw   (575,193) -- (575,301.67) -- (507.67,301.67) -- (507.67,193) -- cycle ;
%Shape: Rectangle [id:dp6301464853530716] 
\draw   (345,193) -- (345,301.67) -- (277.67,301.67) -- (277.67,193) -- cycle ;
%Straight Lines [id:da456626043465608] 
\draw    (579.67,253.67) -- (585.67,253.67) ;
\draw [shift={(587.67,253.67)}, rotate = 180] [color={rgb, 255:red, 0; green, 0; blue, 0 }  ][line width=0.75]    (4.37,-1.96) .. controls (2.78,-0.92) and (1.32,-0.27) .. (0,0) .. controls (1.32,0.27) and (2.78,0.92) .. (4.37,1.96)   ;
%Straight Lines [id:da45119640246845805] 
\draw  [dash pattern={on 4.5pt off 4.5pt}]  (582.67,253.67) -- (577,253.67) ;
\draw [shift={(575,253.67)}, rotate = 360] [color={rgb, 255:red, 0; green, 0; blue, 0 }  ][line width=0.75]    (4.37,-1.96) .. controls (2.78,-0.92) and (1.32,-0.27) .. (0,0) .. controls (1.32,0.27) and (2.78,0.92) .. (4.37,1.96)   ;
%Straight Lines [id:da16632967107105423] 
\draw  [dash pattern={on 4.5pt off 4.5pt}]  (538.67,309) -- (538.67,303.67) ;
\draw [shift={(538.67,301.67)}, rotate = 90] [color={rgb, 255:red, 0; green, 0; blue, 0 }  ][line width=0.75]    (4.37,-1.96) .. controls (2.78,-0.92) and (1.32,-0.27) .. (0,0) .. controls (1.32,0.27) and (2.78,0.92) .. (4.37,1.96)   ;
%Straight Lines [id:da4399452087486009] 
\draw    (538.67,309) -- (538.67,313) ;
\draw [shift={(538.67,315)}, rotate = 270] [color={rgb, 255:red, 0; green, 0; blue, 0 }  ][line width=0.75]    (4.37,-1.96) .. controls (2.78,-0.92) and (1.32,-0.27) .. (0,0) .. controls (1.32,0.27) and (2.78,0.92) .. (4.37,1.96)   ;
%Curve Lines [id:da9833517669211487] 
\draw    (614.33,295) .. controls (597.26,291.06) and (587.3,294.88) .. (581.27,261.23) ;
\draw [shift={(581,259.67)}, rotate = 80.36] [color={rgb, 255:red, 0; green, 0; blue, 0 }  ][line width=0.75]    (10.93,-3.29) .. controls (6.95,-1.4) and (3.31,-0.3) .. (0,0) .. controls (3.31,0.3) and (6.95,1.4) .. (10.93,3.29)   ;
%Curve Lines [id:da8869874631682546] 
\draw    (615,299.67) .. controls (601.14,302.31) and (587.28,317.36) .. (544.31,308.61) ;
\draw [shift={(543,308.33)}, rotate = 11.98] [color={rgb, 255:red, 0; green, 0; blue, 0 }  ][line width=0.75]    (10.93,-3.29) .. controls (6.95,-1.4) and (3.31,-0.3) .. (0,0) .. controls (3.31,0.3) and (6.95,1.4) .. (10.93,3.29)   ;
%Shape: Free Drawing [id:dp668806887817398] 
\draw  [color={rgb, 255:red, 63; green, 114; blue, 214 }  ,draw opacity=1 ][line width=0.75] [line join = round][line cap = round] (428.33,98.13) .. controls (404.93,115.29) and (424.57,126.14) .. (431.67,141.92) .. controls (434.62,148.49) and (432.26,156.96) .. (428.33,162.45) .. controls (426.95,164.38) and (424.33,165.99) .. (424.33,165.18) ;
%Shape: Free Drawing [id:dp15648769028731002] 
\draw  [color={rgb, 255:red, 63; green, 114; blue, 214 }  ,draw opacity=1 ][line width=0.75] [line join = round][line cap = round] (437.67,329.67) .. controls (437.67,341.87) and (416.13,340.74) .. (411.67,349.67) .. controls (403.6,365.81) and (411.95,383.86) .. (392.33,393.67) .. controls (391.91,393.88) and (390.58,395.67) .. (389.67,395.67) ;
%Shape: Free Drawing [id:dp5855556306198761] 
\draw  [color={rgb, 255:red, 0; green, 0; blue, 0 }  ,draw opacity=1 ][line width=0.75] [line join = round][line cap = round] (277.57,252.5) .. controls (277.57,246.01) and (292.55,242.68) .. (296.31,243) .. controls (305.92,243.81) and (308.21,248.95) .. (316.57,254) .. controls (327.37,260.53) and (339.77,252.42) .. (344.93,243.5) ;
%Shape: Free Drawing [id:dp9284196754487037] 
\draw  [color={rgb, 255:red, 0; green, 0; blue, 0 }  ,draw opacity=1 ][line width=0.75] [line join = round][line cap = round] (508.44,235) .. controls (528.64,209.42) and (556.71,248.05) .. (575.06,215.5) ;

% Text Node
\draw (620,288.07) node [anchor=north west][inner sep=0.75pt]  [font=\small]  {$2$};
% Text Node
\draw (549.5,387) node [anchor=north west][inner sep=0.75pt]  [font=\small]  {$W_r$};
% Text Node
\draw (210,212) node [anchor=north west][inner sep=0.75pt]  [font=\small]  {$W_{r+1}$};
% Text Node
\draw (497.5,346) node [anchor=north west][inner sep=0.75pt]  [font=\small]  {$W_{r-2}$};
% Text Node
\draw (414,295) node [anchor=north west][inner sep=0.75pt]  [font=\small]  {$W_{r-3}$};

\end{tikzpicture}
\vspace{-1.6cm}
\caption{Illustration of the event $B(r-2,r)$.}
\label{Brr}
\end{figure}
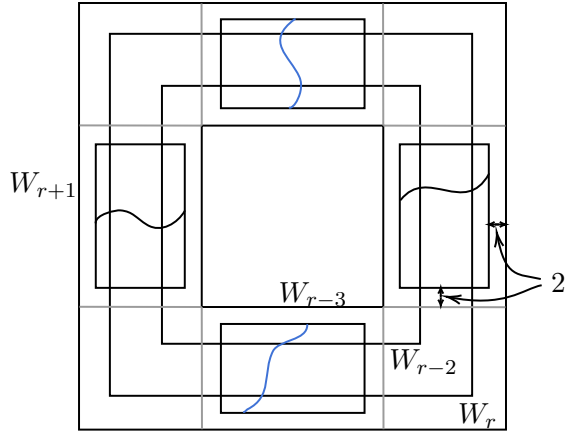

\noindent Now note that 
	$$
	A_4(s,t) \supseteq \widetilde A_4(s,r-2) \cap B(r-2,r) \cap \dbtilde A_4(r,t).
	$$
	On the other hand, the event $B(r-2,r)$ is positively associated with $\widetilde A_4(s,r-2) \cap \dbtilde A_4(r,t)$. To see this, first notice that the 4 different arm events in $B(r-2,r)$ are independent. Consider one of them, say the horizontal occupied crossing on the right. By construction, it is at least $2$ distance away from vacant events in $\widetilde A_4(s,r-2) \cap \dbtilde A_4(r,t)$ and hence is independent of them, while it is positively associated with the occupied events in $\widetilde A_4(s,r-2) \cap \dbtilde A_4(r,t)$. The same argument applies to the other 3 crossings in $B(r-2,r)$, implying the positive association. Thus we obtain,
	\beaa
	\alpha_4(s,t) &\ge& \P_{\lambda_c}(B(r-2,r)) \times \P_{\lambda_c}(\widetilde A_4(s,r-2) \cap \dbtilde A_4(r,t))\\
	& \ge & C_0 \widetilde \alpha_4(s,r-2) \dbtilde \alpha_4(r,t),
	\eeaa
where, for the final step, we have used independence as well as the fact that, by the RSW bounds \eqref{eq:RSW}, there exists a constant $C_0 \in (0,\infty)$ such that $\P_{\lambda_c}(B(r-2,r)) \ge C_0$, concluding the proof.
\end{proof}
 \subsection{Sharp noise instability for Boolean percolation ---  Theorem~\ref{t:sharpNS_Boolean}}\label{ss:proofBoolean}

 In this section, we prove Theorem~\ref{t:sharpNS_Boolean}, i.e., sharp noise instability for critical Boolean percolation, by checking the conditions \ref{A1} -- \ref{A3} from Section~\ref{ss.gen_sns_cont}. Recall from Section \ref{sec:quasim_boolean} that, we consider a critical Boolean percolation model on $\R^2$ with intensity measure $\lambda_c \;\md x$ and balls of unit radius. Note that this model fits in the framework in Section ~\ref{ss.gen_sns_cont} by simply taking the mark space $\M=\{1\}$ with a Dirac measure at $1$. 

\begin{proof}[Proof of Theorem~\ref{t:sharpNS_Boolean}]
By Proposition~\ref{p:sharpNStab}, it suffices to verify the conditions \ref{A1} -- \ref{A3} for the Boolean model. From \cite[Eq.\ (2.16)]{MS22}, we have that there exists a universal constant $\epsilon \in (0,2)$ such that
	$$
	\alpha_4(r,R)= \P_{\lambda_c}(A_4(r,R)) \ge \epsilon(r/R)^{2-\epsilon}
	$$
	for any $0<r<R<\infty$, proving \ref{A1}. 
 
Condition \ref{A2} follows from Theorem \ref{t:quasimPBM} which was proven above in Section \ref{sec:quasim_boolean}. The proof of \ref{A3} (with $\rho=1/2$) is provided in the next subsection \ref{sec:A3A4_boolean}. The following Lemma proves \ref{A4} and thereby completes the proof of the theorem. 
\end{proof}
\begin{lemma}
	\label{e:2corrpivotalvoronoi}
	Let $\eps \in (0,2)$ be as in \ref{A1}. For $x,y \in W_{L/2}$ and $|x-y| \geq L^{\eps/3}$,
	\begin{equation*}
		\E[|D_{\tx}D_{\ty}f_L( \eta)|^2] \lesssim \alpha_4^2\left(\frac{|x-y|}{4}\right).
	\end{equation*}
\end{lemma}
\begin{proof}
	Notice that it is enough to consider $\E[|D_{\tx}D_{\ty}f_L( \eta)|]$ and since $f_L$ is increasing, \eqref{e:4sum} applies.
	Fix $x,y$ satisfying the assumptions. In this case, one can see that the event $D_{\tx} f_L(\eta) \neq 0, D_{\ty} f_L(\eta) \neq 0$ (or any of the other three events on the r.h.s.\ of \eqref{e:4sum}) implies that there are four alternating arms in the rectangular annuli $A^x(1,|x-y|/4)$ and $A^y(1,|x-y|/4)$. %and $A^{\frac{x+y}{2}}(|x-y|, (1-\rho)L)$.
 Also, since the radius of balls is $1$, for $L$ large enough, these three annuli are disjoint and hence the three events are independent. Using this, we derive
	$$
	\P(D_{\tx} f_L(\eta) \neq 0, D_{\ty} f_L(\eta) \neq 0)
	\le \P(A_4^x(R,|x-y|/4)\cap A_4^y(R,|x-y|/4))\lesssim \alpha_4^2(|x-y|/4). %\alpha_4(|x-y|, (1-\rho)L).
	$$
	Since this is true for all the four probabilities in the r.h.s.\ of \eqref{e:4sum}, the result follows. % by \ref{A2}.
\end{proof}

\subsubsection{Pivotal sample estimates for Boolean percolation --- Condition \ref{A3}} 
\label{sec:A3A4_boolean}

To prove \ref{A3}, we need bounds on certain $2$-arm and $3$-arm events; see \eqref{e:2++} and Lemma \ref{l:weak3+}. For $0<r <R$, recall that $\alpha_1(r,R)$ denotes the probability that there exists an occupied path connecting $\partial W_r$ to $\partial W_R$. Similarly, let $\alpha_2^{++}(r,R):= \P_{\lambda_c}(A_2^{++}(r,R))$, where $A_2^{++}(r,R) \equiv A_2^{0,++}(r,R)$ denotes the two-arm event in the annulus $W_R \setminus W_r$ in the first quadrant i.e., existence of an occupied and vacant path from $\partial W_r$ to $\partial W_R$ inside $(W_R \setminus W_r) \cap \R_+^2$. Finally, denote $\alpha_3^{+}(r,R):= \P_{\lambda_c}(A_3^{+}(r,R))$, where $A_3^{+}(r,R) \equiv A_3^{0,+}(r,R)$ denotes the three-arm event in the part of the annulus $W_R \setminus W_r$ in the upper half-plane i.e., existence of three paths (occupied-vacant-occupied or vacant-occupied-vacant) from $\partial W_r$ to $\partial W_R$ inside $(W_R \setminus W_r) \cap (\R \times \R_+)$.

First, we have that there exists $c>0$ such that for all $1 \le r < R$,
\begin{equation}\label{e:2++}
\alpha_2^{++}(r,R) \le \alpha_2(r,R) \le \alpha_1(r,R) \le \frac{1}{c}\left(\frac{r}{R}\right)^c,
\end{equation}
where the last inequality is from \cite[Cor.\ 4.4(iii)]{Ahlbergsharp18}. 
Next, we state a bound determining the upper-half plane $3$-arm probabilities and comparing it with $4$-arm probabilities.

\begin{lemma}\label{l:weak3+}
For all $2 \le r \le R$,
\begin{equation*}
\alpha_3^{+}(r,R) \asymp \left(\frac{r}{R}\right)^2 \lesssim \alpha_4(r, R).
\end{equation*}
\end{lemma}
\begin{proof} 
First note that the second inequality is immediate from assumption \ref{A1}, while by quasi-multiplicativity of $\alpha_4$ and $\alpha_3^+$ in Theorem \ref{t:quasimPBM}, it is enough to show the first assertion for $r=2$, i.e.,
\begin{equation*}\label{eq:a3+}    \alpha_3^{+}(2,R) \asymp \left(\frac{1}{R}\right)^2. % \lesssim \alpha_4(R).
\end{equation*}
The proof follows three main steps as in the first exercise sheet in \cite{Wer07}. In STEP 1, we show that for a certain event $F_{S_0}(R)$ (see inside STEP 1 for definition), one has $\mathbb{P}(F_{S_0}(2R)) \asymp \alpha_3^{+}(2,R)$. In STEP 2, we show that $\mathbb{P}(F_{S_0}(2R))\asymp \mathbb{P}(F_{S_0}(R))\asymp R^{-2}$ which completes the proof. %since by quasi-multiplicativity of $\alpha_3^+$ and RSW bound \eqref{eq:RSW}, $\alpha_3^{+}(4,R) \asymp \alpha_3^{+}(2,R)$. 
One of the estimates needed in STEP 2 is proven in STEP 3.
\medskip

\underline{\textsc{STEP 1}:} For a $8\times 8$ square $S$ of the grid $8\mathbb{Z}^2 + (4,4)$, write $F_S(n)$ to denote the event that there exists
$y \in S$ that is an $n$-Good point, i.e., it is the lowest point in $W_n^{(y)}=y+W_n$ of an occupied component that intersects $\partial W_n^{(y)}$. Note that such a lowest point is almost surely unique in our model. Denote by $\partial W^+_R$ the upper-half of $\partial W_R$. Set  $S_0 = [-4,4]^2$. We first argue that $\alpha_3^+(2,R) \asymp \mathbb{P}(F_{S_0}(2R))$.
\medskip

By coupling the events $A_3^+(8,R)$ and $F_{S_0}(2R)$ appropriately, it is straightforward to see that the latter implies the first so that $\mathbb{P}(F_{S_0}(2R)) \lesssim \alpha_3^+(8,R) \asymp \alpha_3^+(2,R)$, where the final step is due to the quasi-multiplicativity of $\alpha_3^+$ and RSW bound \eqref{eq:RSW}. That $\alpha_3^+(2,R) \asymp \alpha_3^+(8,3R) \lesssim \mathbb{P}(F_{S_0}(2R))$ can be shown again by `extending the configuration by hand'. This involves first showing (using the RSW bound and spatial independence) that the probability $\alpha_3^+(8,3R)$ is of the same order as the 3-arm probability where the landing on $\partial S_0$ in the upper-half plane happens within certain specified parts (for example, we can take $\{-8\} \times [3,5]$, $ [-3,3] \times \{8\}$, and $\{8\}\times[3,5]$, respectively) of the boundary (see Figure \ref{extbyhand}). Given this event (say the arms are in the order vacant-occupied-vacant, from left to right), note that $F_{S_0}(2R)$ is implied if
\begin{enumerate}[(i)]
    \item the specified landing areas on the three sides are part of the vacant, occupied and vacant set, respectively,
    \item there is a top-down occupied crossing of $[-.5,.5] \times [0,4]$ (the dotted black path in Figure \ref{extbyhand}), %from the landing area on to top part to slightly inside $S_0$, staying above the $x$-axis
and 
\item there exists a `blocking' vacant path through a suitably chosen corridor (blue dotted path in Figure \ref{extbyhand} in the corridor $K$) of width 1 traversing left to right around this occupied path, staying 2 distance apart, ensuring independence and that the lowest point of the occupied component does not fall outside of $S_0$.
\end{enumerate}

By the RSW bound, all these events happen with probability bounded away from zero, yielding the claim.
\medskip

\begin{figure}
    \tikzset{every picture/.style={line width=0.75pt}} %set default line width to 0.75pt        

\begin{tikzpicture}[x=0.5pt,y=0.5pt,yscale=-1,xscale=1]
\path (-120,494); %set diagram left start at 0, and has height of 494

%Shape: Square [id:dp4607910111665712] 
\draw   (189.6,90.6) -- (513.2,90.6) -- (513.2,414.2) -- (189.6,414.2) -- cycle ;
%Shape: Square [id:dp004369427471050402] 
\draw   (312.4,376.6) -- (389.4,376.6) -- (389.4,453.6) -- (312.4,453.6) -- cycle ;
%Shape: Free Drawing [id:dp21313386254290823] 
\draw  [color={rgb, 255:red, 0; green, 0; blue, 0 }  ,draw opacity=1 ][line width=0.75] [line join = round][line cap = round] (190.14,107.77) .. controls (194.76,107.77) and (199.4,106.28) .. (203.97,106.88) .. controls (214.43,108.23) and (221.97,126.17) .. (225.71,134.63) .. controls (227.99,139.8) and (228.92,148.86) .. (234.6,152.53) .. controls (258.46,167.97) and (277.83,166.84) .. (301.78,176.7) .. controls (311.78,180.82) and (319.82,190.36) .. (327.47,197.29) .. controls (332.16,201.54) and (337.43,205.5) .. (340.31,210.71) .. controls (347.49,223.72) and (330.03,245.63) .. (327.47,257.26) .. controls (324.69,269.81) and (342.75,272.76) .. (344.26,282.33) .. controls (346.72,297.96) and (330.43,303.22) .. (330.43,319.92) ;
%Straight Lines [id:da983501594724858] 
\draw [line width=2.25]    (322,320.8) -- (381.2,321.4) ;
%Shape: Free Drawing [id:dp11173442055677629] 
\draw  [color={rgb, 255:red, 65; green, 114; blue, 214 }  ,draw opacity=1 ][line width=0.75] [line join = round][line cap = round] (514.34,320.13) .. controls (507.23,320.13) and (502.82,319.24) .. (498.09,321.07) .. controls (484.8,326.22) and (481.99,340.68) .. (474.52,352.26) .. controls (464.12,368.41) and (450.34,379.67) .. (431.46,379.67) ;
%Shape: Free Drawing [id:dp47800763211192443] 
\draw  [color={rgb, 255:red, 65; green, 114; blue, 214 }  ,draw opacity=1 ][line width=0.75] [line join = round][line cap = round] (190,398.18) .. controls (200.77,374.25) and (213.48,378.03) .. (232.32,381.83) .. controls (235.37,382.45) and (238.47,382.45) .. (241.52,381.83) .. controls (245.43,381.05) and (247.72,372.59) .. (251.64,371.62) .. controls (264.31,368.49) and (264.81,377.75) .. (272.8,377.75) ;
%Straight Lines [id:da7580205291423494] 
\draw [color={rgb, 255:red, 65; green, 114; blue, 214 }  ,draw opacity=1 ][line width=2.25]    (272.2,359) -- (272.2,394.2) ;
%Shape: Rectangle [id:dp35770114422592636] 
\draw   (273.4,320.8) -- (428.4,320.8) -- (428.4,414.4) -- (273.4,414.4) -- cycle ;
%Shape: Right Angle [id:dp6002971248168036] 
\draw   (272.2,376.6) -- (324.4,376.6) -- (324.4,440.6) ;
%Shape: Right Angle [id:dp4681751396776146] 
\draw   (429.2,376.6) -- (377.4,376.6) -- (377.4,440.6) ;
%Straight Lines [id:da662411143363365] 
\draw    (324.4,440.6) -- (377.4,440.6) ;
%Straight Lines [id:da04267102328201999] 
\draw    (272.4,388.6) -- (312.4,388.6) ;
%Straight Lines [id:da9068607751838402] 
\draw [color={rgb, 255:red, 65; green, 114; blue, 214 }  ,draw opacity=1 ][line width=2.25]    (429.2,359) -- (429.2,394.2) ;
%Straight Lines [id:da8343503852334622] 
\draw    (389.4,388.6) -- (429.4,388.6) ;
%Shape: Free Drawing [id:dp13565456406993337] 
\draw  [color={rgb, 255:red, 65; green, 114; blue, 214 }  ,draw opacity=1 ][dash pattern={on 4.5pt off 4.5pt}][line width=0.75] [line join = round][line cap = round] (273.4,382.6) .. controls (281.77,382.6) and (282.33,380.11) .. (287.4,379.6) .. controls (292.42,379.1) and (297.37,381.24) .. (302.4,381.6) .. controls (307.97,382) and (318.95,380.7) .. (319.4,386.6) .. controls (320.43,400.04) and (317.63,409.48) .. (318.4,422.6) .. controls (318.76,428.8) and (317.32,442.06) .. (322.4,444.6) .. controls (333.95,450.37) and (356.32,442.66) .. (370.4,443.6) .. controls (374.74,443.89) and (379.11,445.31) .. (383.4,444.6) .. controls (384.06,444.49) and (383.19,443.23) .. (383.4,442.6) .. controls (386.41,433.57) and (381.65,431.64) .. (381.4,423.6) .. controls (381.21,417.66) and (380.85,395.43) .. (382.4,384.6) .. controls (382.55,383.54) and (386.82,383.75) .. (387.4,383.6) .. controls (391.62,382.55) and (396.06,382.91) .. (400.4,382.6) .. controls (406.42,382.17) and (412.38,380.2) .. (418.4,380.6) .. controls (423.16,380.92) and (422.65,383.6) .. (429.4,383.6) ;
%Shape: Rectangle [id:dp7433064973909664] 
\draw   (340.65,321.3) -- (361.15,321.3) -- (361.15,413.9) -- (340.65,413.9) -- cycle ;
%Shape: Free Drawing [id:dp32617902387525] 
\draw  [color={rgb, 255:red, 0; green, 0; blue, 0 }  ,draw opacity=1 ][dash pattern={on 4.5pt off 4.5pt}][line width=0.75] [line join = round][line cap = round] (351.4,321.6) .. controls (351.4,324.84) and (347.73,328.92) .. (347.4,332.6) .. controls (346.35,344.2) and (353.74,348.35) .. (354.4,357.6) .. controls (354.71,361.92) and (354.94,366.3) .. (354.4,370.6) .. controls (353.74,375.86) and (348.16,378.55) .. (347.4,384.6) .. controls (346.16,394.5) and (352.33,395.11) .. (353.4,402.6) .. controls (354.56,410.74) and (350.4,409.69) .. (350.4,413.6) ;

% Text Node
\draw (520.2,203.4) node [anchor=north west][inner sep=0.75pt]  [font=\small]  {$W_{3R}$};
% Text Node
\draw (391.2,287) node [anchor=north west][inner sep=0.75pt]  [font=\small]  {$\partial W_{8}^+$};
% Text Node
\draw (393.2,428.4) node [anchor=north west][inner sep=0.75pt]  [font=\small]  {$S_0$};
\draw (280,390) node [anchor=north west][inner sep=0.75pt]  [font=\small]  {$K$};
\end{tikzpicture}
\vspace{-.8cm}
\caption{Illustration of the extension of the event $A_3^+(8,2R)$ to $F_{S_0}(R)$.}
\label{extbyhand}
\end{figure}
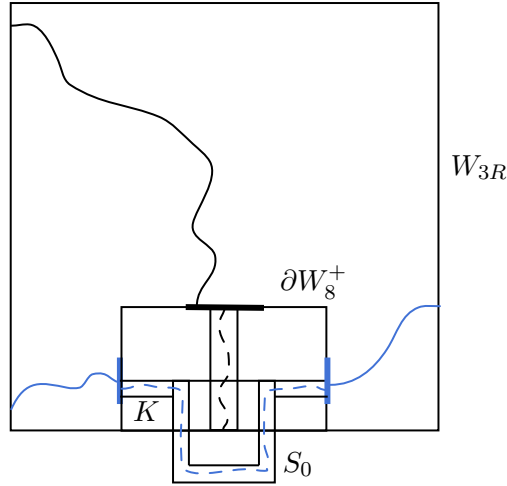

\underline{\textsc{STEP 2}:} Next, we show that for any $R \ge 2$,
\begin{equation}
\label{eq:FS0_bound}
\mathbb{P}(F_{S_0}(2R)) \asymp \frac{1}{R^2}.
\end{equation}
Let $K_1$ be the number of disjoint occupied components connecting  $\partial W^+_{R/2}$ to $\partial W^+_R$ in the upper half-plane. Assume that there exists $C \in [1,\infty)$ such that for all $R \ge 2$, one has $\E K_1 \leq C$. We show this in the STEP 3 below.

We can then derive that for all $R \geq 2$,
$$
\sum_{S: S \subset W_{R/2}} \mathbb{P}(F_S(2R)) = \E \sum_{S: S \subset W_{R/2}} \mathds{1}(F_S(2R)) \le \E K_1 \le C,
$$
since the occurrence of $F_S(2R)$ implies the existence of an occupied component connecting $\partial W^+_{R/2}$ and $\partial W^+_R$, while for two disjoint $S_1$ and $S_2$ if the events $F_{S_1}(2R)$ and $F_{S_2}(2R)$ both occur, then they necessarily correspond to two disjoint such occupied components. On the other hand, if we show there exists $C_0 \in (0,1]$ such that for all $R \geq 2$,
\begin{equation}\label{eq:lbco}
\sum_{S: S \subset W_{R/2}} \mathbb{P}(F_S(2R)) \ge \mathbb{P}(F_S(2R) \text{ holds for some }S \subset W_{R/2}) \ge C_0, 
\end{equation}
then \eqref{eq:FS0_bound} follows from the above two inequalities.
\smallskip

Now we argue the final lower bound in \eqref{eq:lbco}, similarly to STEP 1, as follows. Note that the event in question (i.e., $F_S(2R) \text{ holds for some }S \subset W_{R/2}$) is implied by the existence of a top-down occupied path in the rectangle $[-R/8,R/8] \times [3R/8, 5R/2]$
along with a vacant path through a suitable corridor (with both length and width being some constant multiples of $R$) below and around this rectangle, traversing left to right `blocking' the lowest point of the occupied component from falling outside of $W_{R/2}$. 
This ensures that $F_S(2R) \text{ holds for some }S \subset W_{R/2}$. 
That such a top-down occupied path and `blocking' vacant path exist with positive probability can be guaranteed by RSW bounds \eqref{eq:RSW} and their independence can be guaranteed by choosing $R$ large enough. Since for small $R$, one can simply consider some configuration with positive probability for which the event holds, this yields the lower bound.
\medskip

\underline{\textsc{STEP 3}:}  Now we show that $\E K_1 \leq C$ as claimed above in STEP 2 and thus completing the proof. Let $K_2$ be the number of disjoint occupied paths (i.e., they use disjoint sets of centers of balls) connecting $\partial W^+_{R/2}$ to $\partial W^+_R$ in the upper half-plane. Also let
$$
F_R:=\{\partial W^+_{R/2} \leftrightarrow \partial W^+_R \text{ in the upper half-plane}\}.
$$
Then, notice that there exists some $a \in [0,1)$ such that for all $R \ge 2$,
\begin{equation}\label{eq:crbd}
\mathbb{P}(F_R) \le a.
\end{equation}
Indeed, if this is not the case, then we must have that for any $\epsilon>0$, there exists $R \ge 2$ such that
$
\mathbb{P}(F_R) \ge 1-\epsilon.
$
But $F_R \subseteq A_1 \cup A_2 \cup A_3$, where $A_1$ and $A_3$ are the events that there is a L-R crossing of $[-R,-R/2] \times [0,R]$, and its reflection along the $y$-axis, respectively, while $A_3$ denotes a top to bottom crossing of $[-R,R] \times [R/2,R]$. The `square root trick' (see e.g., \cite[Prop.\ 4.1]{Tas16}) now implies that (noting $\max_{i \in \{1,2,3\}} \mathbb{P}(A_i) = \mathbb{P}(A_3)$)
$$
\mathbb{P}(A_3) \ge 1-\epsilon^{1/3}.
$$
This violates the RSW bound \eqref{eq:RSW} for $\epsilon$ small enough, proving \eqref{eq:crbd}. Now by the Poisson BK inequality (\cite[Theorem 4.1]{Gup1999}), we obtain
$$
\mathbb{P}(K_2 \ge k) \le a^k,
$$
which implies that there exists $C \in [1,\infty)$ such that $\mathbb{E} K_2 \le C$. Noting that $K_1 \le K_2$, this yields our claim.
\end{proof}

\noindent \textbf{\underline{Proof of \ref{A3} for critical Boolean percolation}:}
\medskip

\noindent {\underline{\textit{Upper bound}:} We need to show that $\E |\gamma_L| \lesssim L^2\alpha_4(L)$. We argue as in \cite[Lem.\ 4.9]{Vann2019} or \cite[Section 6.1.3]{GS}. First note that since adding a point outside of $W_{L+1}$ cannot affect a crossing in $W_L$, we have
\begin{equation}\label{e:A3ub}
\E |\gamma_L| = \int_{\R^2}\E |D_{(x,1)} f_L(\eta)|^2 \md x = \int_{W_{L/2}}\E |D_{(x,1)} f_L(\eta)|^2 \md x + \int_{W_{L+1} \setminus W_{L/2}}\E |D_{(x,1)} f_L(\eta)|^2 \md x.
\end{equation}
Also, for $x \in W_{L/2}$, by quasi-multiplicativity of 4-arm probabilities guaranteed by \ref{A2}, we have
\begin{equation}\label{e:ub1}
\P(D_{(x,1)} f_L(\eta) \neq 0) \lesssim \alpha_4(L),
\end{equation}
which yields that the first integral in \eqref{e:A3ub} is $\text{O}(L^2 \alpha_4(L))$.
The second integral, by rotational invariance, is bounded by (up to a constant factor)
\begin{equation}\label{e:ub2}
\int_{L/2}^{L+1} \int_{0}^u \P(D_{((u,v),1)} f_L(\eta) \neq 0) \md v \md u.
\end{equation}
First consider the case when $u \in (L/2,L-1)$. Let $d_0=d_0(u)$ be the distance from $(u,v)$ to 
its closest side of $W_L$, which means $d_0= L-u$. Also, let $d_1=d_1(v)$ be the distance between $x_1=(L,v)$ (projection of $x$ on the closest side) and the closest corner of $W_L$, say $x_2$, i.e., $d_1=(L-v)$. Then, notice that $D_{((u,v),1)} f_L(\eta) \neq 0$ implies that the event 
$$
A_4^{(u,v)}(1,d_0) \cap A_3^{x_1,+}(2d_0+2,d_1) \cap A_2^{x_2,++}(2d_1 +2,L)
$$
occurs. Since our radius of influence is $1$, by independence, Lemma \ref{l:weak3+} and quasi-multiplicativity of $\alpha_4$, we have
$$
\P(D_{((u,v),1)} f_L(\eta) \neq 0) \lesssim \alpha_2^{++}(2d_1+2,L) \alpha_3^{+}(2d_0+2,d_1)\alpha_4(d_0) \lesssim \alpha_2^{++}(2d_1+2,L) \alpha_4(d_1).
$$
Now again by quasi-multiplicativity of $\alpha_4$ and \ref{A1}, we have
$$
\alpha_4(d_1) \lesssim \alpha_4(L)\left(\frac{L}{d_1}\right)^{2-\epsilon}.
$$
Thus, using \eqref{e:2++} we obtain
\bea\label{e:ub4'}
&&\int_{L/2}^{L-1} \int_{0}^{u} \P(D_{((u,v),1)} f_L(\eta) \neq 0) \md v \md u \nonumber\\
&&\lesssim \alpha_4(L) \int_{L/2}^{L-1} \int_{0}^{u} \left(\frac{2d_1(v)+2}{L}\right)^{c}\left(\frac{L}{d_1(v)}\right)^{2-\epsilon} \md v \md u \nonumber\\
&&= \alpha_4(L) \int_{L/2}^{L-1} \int_{0}^{u} \left(\frac{2(L-v)+2}{L}\right)^{c}\left(\frac{L}{2(L-v)}\right)^{2-\epsilon} \md v \md u\nonumber\\
&& \lesssim  L^{2-\epsilon-c}\alpha_4(L) \int_{0}^{L-1} (L-v)^{c+\epsilon -1} \md v \asymp L^2\alpha_4(L).
\eea

Finally, consider the case when $u \in [L-1,L+1]$. Since $L^2 \alpha_4(L) \to \infty$ as $L \to \infty$ by \ref{A1}, it is enough to consider $v \in [0,u \wedge (L-1)]$. Let $d_1$ be as before. Notice that $D_{((u,v),1)} f_L(\eta) \neq 0$ implies that the event $A_3^{x_1,+}(2,d_1) \cap A_2^{x_2,++}(d_1 +2,L)$ occurs.
Then for $L$ large enough, we have from Lemma \ref{l:weak3+} that
$$
\P(D_{((u,v),1)} f_L(\eta) \neq 0) \lesssim \alpha_2^{++}(d_1+2,L) \alpha_3^{+}(2,d_1) \lesssim \alpha_2^{++}(d_1+2,L) \alpha_4(d_1),
$$
and thereby \eqref{e:2++} yields that
\bea\label{e:ub4}
&&\int_{L-1}^{L+1} \int_{0}^{u \wedge (L-1)} \P(D_{((u,v),1)} f_L(\eta) \neq 0) \md v \md u \nonumber\\
&& \lesssim \alpha_4(L) \int_{L-1}^{L+1} \int_{0}^{u \wedge (L-1)} \left(\frac{d_1(v)+2}{L}\right)^{c}\left(\frac{L}{d_1(v)}\right)^{2-\epsilon} \md v \md u\nonumber\\
&& \lesssim \alpha_4(L) \int_{L-1}^{L+1} \int_{0}^{u \wedge (L-1)} \left(\frac{L-v+2}{L}\right)^{c}\left(\frac{L}{L-v}\right)^{2-\epsilon} \md v \md u\nonumber\\
&& \lesssim L^{2-\epsilon-c}\alpha_4(L) \int_{0}^{L-1} (L-v)^{c+\epsilon -2} \md v \asymp L^2\alpha_4(L).
\eea
The upper bound in \ref{A3} now follows upon combining \eqref{e:A3ub}, \eqref{e:ub1}, \eqref{e:ub2}, \eqref{e:ub4'} and \eqref{e:ub4}. \\

\noindent {\underline{\textit{Lower bound}:} We need to show $\E |\gamma_L \cap \widehat W_{L/2}| \gtrsim L^2\alpha_4(L)$, which follows if we show that for $x \in W_{L/2}$, we have $\P(D_{(x,1)} f_L(\eta) \neq 0) \gtrsim \alpha_4(L)$. 

\noindent For simplicity and without loss of generality, assume that $L=2^l$ for some $l \ge 10$. Fix $x \in W_{L/2}=W_{2^{l-1}}$. For $s \ge 4$ and $r \ge s+2$, define the event $\overline A_4^{(x)}(2^s,2^r)$ as
\beaa
\overline A_4^{(x)}(2^s,2^r) &=&\text{there are 4 paths of alternating type from $\partial W_{2^r}^{(x)}$ to $\partial W_{2^s}^{(x)}$ such that all } \\
&&\text{the parts of the two occupied paths outside of $W_{2^{r-1}}^{(x)} \setminus W_{2^{s+1}}^{(x)}$ lie in }\\
&&\text{$x+\mathcal{A}(1,s+1) \cup \mathcal{A}(1,r)$ and $x+\mathcal{A}(3,s+1) \cup \mathcal{A}(3,r)$ respectively, and }\\
&&\text{similar event occurs for the vacant paths with $x+\mathcal{B}(2,s+1) \cup \mathcal{B}(2,r)$ }\\
&&\text{and $x+\mathcal{B}(4,s+1)\cup \mathcal{B}(4,r)$. In addition there are occupied top to }\\
&&\text{bottom crossings in $x+\mathcal{A}(1,s+1)^{-}$, $x+\mathcal{A}(3,s+1)^{-}$, $x+\mathcal{A}(1,r)^{-}$ }\\
&&\text{and $x+\mathcal{A}(3,r)^{-}$, and vacant horizontal crossings of $x+\mathcal{B}(2,s+1)^{-}$, }\\
&&\text{$x+\mathcal{B}(4,s+1)^{-}$, $x+\mathcal{B}(2,r)^{-}$  and $x+\mathcal{B}(4,r)^{-}$},
\eeaa
Note that the argument in the proof of Theorem~\ref{thm1} can be applied to $\dbtilde \alpha_4(2^s,2^r)$, instead of $\alpha_4(2^s,2^r)$, to ensure that the paths at the outer square $W_{2^r}$ hit the boundary inside the regions $\mathcal{A}(i,r)$ and $\mathcal{B}(i+1,r)$ for $i=1,3$ respectively. Arguing this way, one can obtain that for $s \ge 4$ and $r$ large enough,
\begin{equation}
\label{eq:al4barA4}
    \alpha_4(2^s,2^r) \asymp \dbtilde \alpha_4(2^s,2^r) \lesssim \P_{\lambda_c}(\overline A_4^{(x)}(2^s,2^r)).
\end{equation}
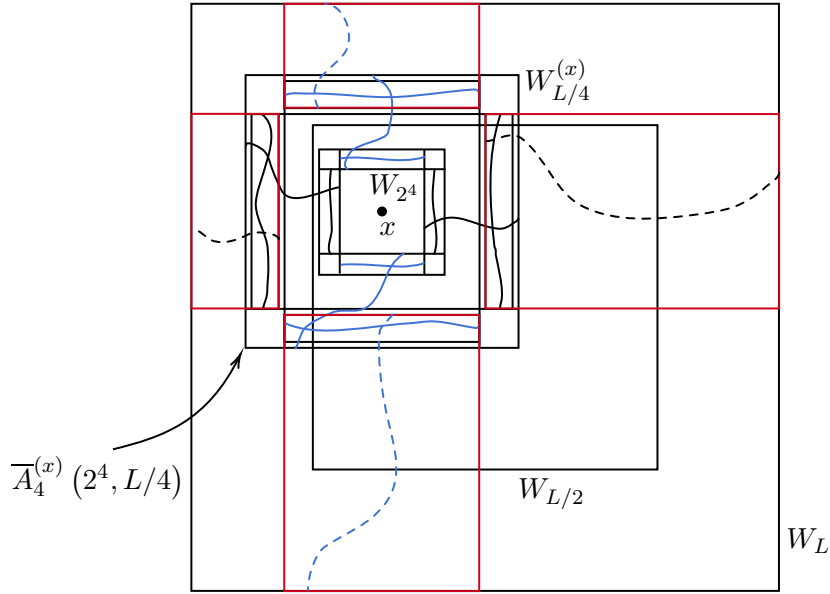
\begin{figure}
\vspace{-1cm}
    \tikzset{every picture/.style={line width=0.75pt}} %set default line width to 0.75pt        

\begin{tikzpicture}[x=0.5pt,y=0.5pt,yscale=-1,xscale=1]
\path (73,517); %set diagram left start at 0, and has height of 517

%Shape: Square [id:dp9526349941479804] 
\draw   (318,37) -- (762.5,37) -- (762.5,481.5) -- (318,481.5) -- cycle ;
%Shape: Square [id:dp8128338897644454] 
\draw   (409.89,128.89) -- (670.61,128.89) -- (670.61,389.61) -- (409.89,389.61) -- cycle ;
%Shape: Circle [id:dp804261400450869] 
\draw  [fill={rgb, 255:red, 0; green, 0; blue, 0 }  ,fill opacity=1 ] (459.25,194.25) .. controls (459.25,192.59) and (460.59,191.25) .. (462.25,191.25) .. controls (463.91,191.25) and (465.25,192.59) .. (465.25,194.25) .. controls (465.25,195.91) and (463.91,197.25) .. (462.25,197.25) .. controls (460.59,197.25) and (459.25,195.91) .. (459.25,194.25) -- cycle ;
%Shape: Square [id:dp9235726942696609] 
\draw   (359,91) -- (565.5,91) -- (565.5,297.5) -- (359,297.5) -- cycle ;
%Shape: Square [id:dp7410726498743578] 
\draw   (388.5,120.5) -- (536,120.5) -- (536,268) -- (388.5,268) -- cycle ;
%Shape: Right Angle [id:dp4788945946465528] 
\draw   (388.5,91.5) -- (388.5,120.5) -- (359.5,120.5) ;
%Shape: Right Angle [id:dp5469972001979728] 
\draw   (565.75,120.5) -- (536,120.5) -- (536,91) ;
%Shape: Right Angle [id:dp035607295675056605] 
\draw   (536,297.75) -- (536,268) -- (565.5,268) ;
%Shape: Right Angle [id:dp22686898908425812] 
\draw   (358.75,268) -- (388.5,268) -- (388.5,297.5) ;
%Shape: Rectangle [id:dp9282365992824892] 
\draw   (388.25,95.5) -- (535.75,95.5) -- (535.75,116) -- (388.25,116) -- cycle ;
%Shape: Rectangle [id:dp8672708917788858] 
\draw   (388.25,272.5) -- (535.75,272.5) -- (535.75,293) -- (388.25,293) -- cycle ;
%Shape: Rectangle [id:dp24349173955591819] 
\draw   (363.5,268) -- (363.5,120.5) -- (384,120.5) -- (384,268) -- cycle ;
%Shape: Rectangle [id:dp899335751692603] 
\draw   (540.5,268) -- (540.5,120.5) -- (561,120.5) -- (561,268) -- cycle ;
%Shape: Square [id:dp5941240281574685] 
\draw   (430.25,162.25) -- (494.25,162.25) -- (494.25,226.25) -- (430.25,226.25) -- cycle ;
%Shape: Square [id:dp39726445345036443] 
\draw   (414.62,147.28) -- (509.56,147.28) -- (509.56,242.21) -- (414.62,242.21) -- cycle ;
    %Shape: Rectangle [id:dp20881378758859737] 
    %\draw   (505.75,162) -- (505.75,226) -- (498.25,226) -- (498.25,162) -- cycle ;
%Shape: Right Angle [id:dp37374101856475095] 
\draw   (415.33,162.25) -- (430.25,162.25) -- (430.25,147) ;
%Shape: Right Angle [id:dp4087232306155759] 
\draw   (430.25,241.17) -- (430.25,226.25) -- (415,226.25) ;
%Shape: Right Angle [id:dp43096046666347965] 
\draw   (509.56,226.25) -- (494.25,226.25) -- (494.25,242.21) ;
%Shape: Right Angle [id:dp07106657781241799] 
\draw   (494.25,146.94) -- (494.25,162.25) -- (510.21,162.25) ;
    %Shape: Rectangle [id:dp5662854988474868] 
    %\draw   (430.5,150.25) -- (494.5,150.25) -- (494.5,157.75) -- (430.5,157.75) -- cycle ;
    %Shape: Rectangle [id:dp18275622741812425] 
    %\draw   (426.25,162) -- (426.25,226) -- (418.75,226) -- (418.75,162) -- cycle ;
    %Shape: Rectangle [id:dp3108929759098893] 
    %\draw [loosely dashed]  (430.5,230.25) -- (494.5,230.25) -- (494.5,237.75) -- (430.5,237.75) -- cycle ;
%Shape: Free Drawing [id:dp5063128765789731] 
\draw  [color={rgb, 255:red, 0; green, 0; blue, 0 }  ,draw opacity=1 ][line width=0.75] [line join = round][line cap = round] (551.75,121) .. controls (541.3,152.35) and (543.79,188.34) .. (547.25,220.5) .. controls (547.36,221.55) and (548.09,222.46) .. (548.25,223.5) .. controls (548.43,224.65) and (548.08,225.85) .. (548.25,227) .. controls (548.8,230.69) and (555.17,254.25) .. (552.25,263) .. controls (551.7,264.65) and (551.98,266.77) .. (550.75,268) ;
%Shape: Free Drawing [id:dp9629627693154086] 
\draw  [color={rgb, 255:red, 0; green, 0; blue, 0 }  ,draw opacity=1 ][line width=0.75] [line join = round][line cap = round] (371.25,120.53) .. controls (390.44,135.02) and (365.61,172.39) .. (366.75,189.97) .. controls (367.51,201.76) and (374.97,211.31) .. (375.75,222.68) .. controls (376.3,230.71) and (376.55,238.83) .. (375.75,246.84) .. controls (375.28,251.54) and (375.93,256.52) .. (374.25,260.93) .. controls (372.2,266.29) and (372.25,269.85) .. (372.25,266.96) ;
%Shape: Free Drawing [id:dp09299134327505953] 
\draw  [color={rgb, 255:red, 0; green, 0; blue, 0 }  ,draw opacity=1 ][line width=0.75] [line join = round][line cap = round] (501.92,162.5) .. controls (501.45,169.33) and (500.46,176.12) .. (500.53,183) .. controls (500.66,197.51) and (506.11,214.42) .. (501.64,226.5) ;
%Shape: Free Drawing [id:dp7014138212555419] 
\draw  [color={rgb, 255:red, 0; green, 0; blue, 0 }  ,draw opacity=1 ][line width=0.75] [line join = round][line cap = round] (423.75,162) .. controls (421.33,176.97) and (424.07,192.33) .. (424.25,207.5) .. controls (424.31,213.06) and (420.29,220.59) .. (422.75,225.5) .. controls (422.96,225.92) and (423.28,226.5) .. (423.75,226.5) ;
%Shape: Free Drawing [id:dp5645193606685528] 
\draw  [color={rgb, 255:red, 63; green, 114; blue, 214 }  ,draw opacity=1 ][line width=0.75] [line join = round][line cap = round] (431.18,154) .. controls (452.27,151.21) and (474.08,159.15) .. (494.32,153) ;
%Shape: Free Drawing [id:dp5715507713491368] 
\draw  [color={rgb, 255:red, 63; green, 114; blue, 214 }  ,draw opacity=1 ][line width=0.75] [line join = round][line cap = round] (431.68,234.5) .. controls (446.95,233.53) and (462.25,235.65) .. (477.54,236) .. controls (478.88,236.03) and (493.82,234.87) .. (493.82,233) ;
%Shape: Free Drawing [id:dp6532978411054691] 
\draw  [color={rgb, 255:red, 63; green, 114; blue, 214 }  ,draw opacity=1 ][line width=0.75] [line join = round][line cap = round] (389.2,106) .. controls (394.88,103.71) and (401.41,105.13) .. (407.53,105) .. controls (429.9,104.53) and (452.01,108.65) .. (474.38,107.5) .. controls (492.94,106.55) and (511.32,103.38) .. (529.85,102) .. controls (533.71,101.71) and (533.8,101.98) .. (535.3,103.5) ;
%Shape: Free Drawing [id:dp5147629979603068] 
\draw  [color={rgb, 255:red, 63; green, 114; blue, 214 }  ,draw opacity=1 ][line width=0.75] [line join = round][line cap = round] (389.25,279) .. controls (414.82,290.36) and (454.18,280.49) .. (481.75,279) .. controls (488.59,278.63) and (495.41,277.04) .. (502.25,277.5) .. controls (512.69,278.21) and (522.89,281) .. (533.25,282.5) .. controls (534.42,282.67) and (536.93,280) .. (535.75,280) ;
%Shape: Free Drawing [id:dp9870947152986009] 
\draw  [color={rgb, 255:red, 63; green, 114; blue, 214 }  ,draw opacity=1 ][line width=0.75] [line join = round][line cap = round] (436.25,162.09) .. controls (431.01,159.52) and (437.21,150.79) .. (438.75,148.84) .. controls (443.37,142.96) and (465.28,140.38) .. (466.75,134.6) .. controls (469.06,125.52) and (471.03,115.82) .. (469.25,106.63) .. controls (468.61,103.32) and (466.5,100.34) .. (464.25,97.79) .. controls (461.91,95.14) and (455.75,94.91) .. (455.75,91.41) ;
%Shape: Free Drawing [id:dp30040152825068467] 
\draw  [color={rgb, 255:red, 63; green, 114; blue, 214 }  ,draw opacity=1 ][line width=0.75] [line join = round][line cap = round] (478.75,226.06) .. controls (464.18,235.9) and (466.98,247.01) .. (456.25,261.5) .. controls (446.34,274.88) and (435.08,266) .. (423.75,271.11) .. controls (416.73,274.29) and (409.24,277.31) .. (403.75,282.76) .. controls (400.03,286.46) and (396.75,302.21) .. (396.75,296.93) ;
%Shape: Free Drawing [id:dp3862143215999707] 
\draw  [color={rgb, 255:red, 0; green, 0; blue, 0 }  ,draw opacity=1 ][line width=0.75] [line join = round][line cap = round] (494.25,207) .. controls (515.92,190.75) and (526.63,203.4) .. (547.75,207) .. controls (555.45,208.31) and (557.64,205.64) .. (564.75,201) .. controls (565.19,200.71) and (565.78,199.5) .. (565.25,199.5) ;
%Shape: Free Drawing [id:dp967527584778477] 
\draw  [color={rgb, 255:red, 0; green, 0; blue, 0 }  ,draw opacity=1 ][line width=0.75] [line join = round][line cap = round] (359.22,144.5) .. controls (359.22,139.28) and (371.29,146.34) .. (373.13,148.5) .. controls (378.31,154.58) and (375.38,156.52) .. (386.05,178) .. controls (393.23,192.45) and (419.32,176.96) .. (429.78,176) ;
%Shape: Rectangle [id:dp46749289494973545] 
\draw  [color={rgb, 255:red, 208; green, 2; blue, 27 }  ,draw opacity=1 ] (540.5,120.5) -- (762.67,120.5) -- (762.67,267.67) -- (540.5,267.67) -- cycle ;
%Shape: Free Drawing [id:dp2176653521309675] 
\draw  [dashed, color={rgb, 255:red, 0; green, 0; blue, 0 }  ,draw opacity=1 ][line width=0.75] [line join = round][line cap = round] (541.25,141) .. controls (548.36,141) and (559.48,132.5) .. (567.93,139) .. controls (586.23,153.07) and (592.93,179.83) .. (613.94,190.33) .. controls (638.87,202.79) and (669.88,200.34) .. (696.64,197) .. controls (723.47,193.65) and (734.97,191.31) .. (753.33,175) .. controls (755.45,173.12) and (762.66,168.75) .. (762.66,165) ;
%Shape: Rectangle [id:dp7683304023985125] 
\draw  [color={rgb, 255:red, 208; green, 2; blue, 27 }  ,draw opacity=1 ] (388.25,37) -- (535.75,37) -- (535.75,115.67) -- (388.25,115.67) -- cycle ;
%Shape: Rectangle [id:dp7125427845309678] 
\draw  [color={rgb, 255:red, 208; green, 2; blue, 27 }  ,draw opacity=1 ] (318,120.33) -- (384,120.33) -- (384,267.67) -- (318,267.67) -- cycle ;
%Shape: Free Drawing [id:dp44818720715831994] 
\draw  [dashed, color={rgb, 255:red, 0; green, 0; blue, 0 }  ,draw opacity=1 ][line width=0.75] [line join = round][line cap = round] (383.92,215) .. controls (385.58,215) and (382.26,212.85) .. (381.22,212.33) .. controls (377.68,210.59) and (366.96,208.7) .. (360.96,211.67) .. controls (353.03,215.58) and (343.44,219.46) .. (333.94,216.33) .. controls (328.96,214.69) and (324.72,207.67) .. (318.4,207.67) ;
%Shape: Free Drawing [id:dp24937423957807336] 
\draw  [dashed,color={rgb, 255:red, 63; green, 114; blue, 214 }  ,draw opacity=1 ][line width=0.75] [line join = round][line cap = round] (414,115.72) .. controls (401.99,91.97) and (430.53,85.56) .. (436,70.24) .. controls (440.54,57.53) and (432.54,37.28) .. (419.33,37.28) ;
%Shape: Rectangle [id:dp49696485681054536] 
\draw  [color={rgb, 255:red, 208; green, 2; blue, 27 }  ,draw opacity=1 ] (388.25,272.67) -- (535.75,272.67) -- (535.75,481.67) -- (388.25,481.67) -- cycle ;
%Shape: Free Drawing [id:dp5307326520743176] 
\draw  [dashed, color={rgb, 255:red, 63; green, 114; blue, 214 }  ,draw opacity=1 ][line width=0.75] [line join = round][line cap = round] (470,272.33) .. controls (470,273.48) and (461.22,281.49) .. (460.67,284.33) .. controls (459.43,290.66) and (459.72,297.22) .. (460,303.67) .. controls (460.68,319.43) and (462.98,336.23) .. (466.67,351.67) .. controls (471.75,372.94) and (479.2,397.04) .. (464,416.33) .. controls (451.33,432.41) and (435.92,436.71) .. (418.67,448.33) .. controls (412.64,452.39) and (408.34,465.61) .. (407.33,471) .. controls (406.28,476.64) and (406,481.62) .. (406,481) ;
%Curve Lines [id:da9634067200383265] 
\draw    (258.67,374.33) .. controls (296.29,365.09) and (329.99,353.89) .. (354.59,303.54) ;
\draw [shift={(355.33,302)}, rotate = 115.52] [color={rgb, 255:red, 0; green, 0; blue, 0 }  ][line width=0.75]    (10.93,-3.29) .. controls (6.95,-1.4) and (3.31,-0.3) .. (0,0) .. controls (3.31,0.3) and (6.95,1.4) .. (10.93,3.29)   ;

% Text Node
\draw (458,200) node [anchor=north west][inner sep=0.75pt]  [font=\small]  {$x$};
% Text Node
\draw (567.83,76.4) node [anchor=north west][inner sep=0.75pt]  [font=\small]  {$W_{L/4}^{( x)}$};
% Text Node
\draw (450,164) node [anchor=north west][inner sep=0.75pt]  [font=\small]  {$W_{2^{4}}$};
% Text Node
\draw (563,393.9) node [anchor=north west][inner sep=0.75pt]  [font=\small]  {$W_{L/2}$};
% Text Node
\draw (766,431.9) node [anchor=north west][inner sep=0.75pt]  [font=\small]  {$W_{L}$};
% Text Node
\draw (180,377) node [anchor=north west][inner sep=0.75pt]  [font=\small]  {$\overline A_{4}^{(x)}\left( 2^{4} ,L/4\right)$};

\end{tikzpicture}
\vspace{-.5cm}
\caption{Illustration of the event $\overline{A_4}^{(x)}(2^4,L/4)$ for $x \in W_{L/2}$ (all crossings inside the square $W_{L/4}^{(x)}$), and its extention to $W_L$ (with the 4 dashed crossings of alternating colors in the 4 red rectangles).}
\label{fig:barA4x}
\end{figure}

Now we extend by hand (the dashed crossings in Figure \ref{fig:barA4x}) the event $\overline{A_4}^{(x)}(2^4,L/4)$ to ensure that we have $D_{(x,1)} f_L(\eta) \neq 0$. To do this, we first enforce the event that there is a L-R occupied crossing of $[x_1+L/8+2,L] \times [x_2 -L/8,x_2+L/8]$. We also enforce similar occupied event to the left of $W_{L/4}^{(x)}$, and similar events, with occupied replaced by vacant, and L-R replaced by top to bottom, to the top and bottom of $W_{L/4}^{(x)}$. Note that all these events are positively associated with $\overline{A_4}^{(x)}(2^4,L/4)$ and using \eqref{eq:RSW}, they all have a probability bounded below by a positive constant. Also, together with $\overline{A_4}^{(x)}(2^4,L/4)$, these events imply that there are 4-arms of alternating types occupied, vacant, occupied, vacant, starting at the left, top, right and bottom sides of $W_{2^4}^{(x)}$ and hitting the left, top, right and bottom sides of $W_L$, respectively. Finally consider the event that there is a top to bottom vacant crossing of $[x_1-2^4,x_1+2^4] \times [x_2 - 2^5+2,x_2 + 2^5-2]$, while upon adding $B(x,1)$, there is a L-R occupied crossing of $[x_1-2^5+2,x_1+2^5-2] \times [x_2 - 2^4,x_2 + 2^4]$. This event has a probability strictly larger than zero, and it is positively associated with all the previous events. Putting all these events together, we obtain by quasi-multiplicativity, the RSW bound, and \eqref{eq:al4barA4} that
$$
\P(D_{(x,1)} f_L(\eta) \neq 0) \gtrsim \P_{\lambda_c}(\overline A_4^{(x)}(2^4,L/4)) \gtrsim \alpha_4(2^4,L/4) \asymp \alpha_4(L),
$$
which yields the lower bound.
 \qed

 \subsection{Sharp noise sensitivity for Voronoi percolation} 
\label{s:sharpNSvor}

In this section, we prove the sharp noise sensitivity result for crossing event in the case of planar Voronoi percolation stated in Theorem \ref{t:sharpNS_voronoi}. We adapt here the notations from Section~\ref{ss.gen_sns_cont}. While it is possible to confirm the assumptions \ref{A1} -- \ref{A3} in Section~\ref{ss.gen_sns_cont} to conclude sharp noise instability, here we manage to prove instead both sharp noise instability and sharp noise sensitivity for the critical Voronoi percolation under the OU dynamics, by making use of a comparisons with the so-called frozen dynamics in \cite{Vann2021}.
 
Let $\eta$ be a Poisson process on $\X = \R^2 \times \M$ with $\M=\{0,1\}$, and with intensity measure $\md x \times \nu(\md a)$ where $\nu$ is the Ber($1/2$) probability measure, and denote by $\eta^{pr}$ its projection on $\R^2$. Consider the Voronoi tessellation $\{C(x) := C(x,\eta^{pr})\}_{x \in \eta^{pr}}$ of $\R^2$ formed by the collection of points $\eta^{pr}$. Formally, $C(x) := \{y \in \R^2 : |y-x| \leq \min \{|y-z| : z \in \eta^{pr} \}$ is the set of points closer to $x$ than other points of $\eta^{pr}$. One can show that the these cells partition $\R^2$ and have disjoint interiors. Furthermore, 
the cells are convex polygons. 

The Voronoi percolation model is defined as follows. We color the Voronoi cell $C(x)$ of a point $x \in \eta^{pr}$ either white (corresponds to the mark $M(x) =0$) or black (corresponds to the mark $M(x) = 1$) according to the mark $M(x)$ of $x$ in $\eta$. We define the {\em black-colored/occupied region} $\Occ(\eta) = \cup_{x \in \eta^{pr}, M(x) = 1} C(x)$. As before, we are interested in the event of a left-right (LR) crossing of the box $W_L$ through the black-colored region $\Occ(\eta)$.  Let $f_L$ be the $\pm 1$-indicator of LR crossing of $\Occ(\eta) \cap W_L$, and $\alpha_4(L)$ denote the 4-arm probability of $4$ arms of alternating colors from $\partial W_1$ to $\partial W_L$ inside the annulus $W_L \setminus W_1$. We will often refer to $\eta$ as the colored point configuration while $\eta^{pr}$ will be called the uncolored configuration. We let $\Omega$ be the class of locally finite point configurations in $\X$, while $\eta^{pr}$ lives in $\Omega^{pr}$. It is known that $f_L$ is non-degenerate, indeed, a RSW-type result similar to \eqref{eq:RSW} holds for the Voronoi percolation model too; see \cite[Theorem 8.1]{Ahlbergsharp18}.

We prepare by proving two propositions. Below, we let $\gamma_L$ to be the spectral process as in Definition~\ref{def:spproc} associated with the crossing functional $f_L$.

\begin{prop} The spectral point process $\gamma_L$  satisfies
		\label{p:exppivotalsetvoronoi}
		\begin{equation*}
			\label{e:4armpivotalvoronoi}
			\E[|\gamma_L|] =  \int_{\X} \E[|D_{(x,a)}f_L(\eta)|^2]\md x \nu(\md a) \asymp L^2\alpha_4(L).  
		\end{equation*}
	\end{prop}
\begin{proof}
		Though it is also possible to establish the result directly,  we shall use comparison to annealed spectral sample as in Section \ref{ss:comparison_samples} and then use estimates for quenched pivotal sample from \cite{Vann2019,Vann2021}. By the local-finiteness of the Voronoi tessellation, we have that $f_L \in \SA$ and also one has that
  \begin{equation*}
  \int_{\X} \E[(D_{(x,a)} f_L)^2] \, {\rm d}x\nu({\rm d}a) \leq 4\int_{\R^2} \P 
  \big(C(x,\eta^{pr} \cup \{x\}) \cap W_L \neq \emptyset \big) {\rm d}x  < \infty,
  \end{equation*}
  where the last inequality uses that the probability in the integrand decays super-exponentially in $\|x\|$; see for example \cite[Section 5.1]{Penrose2007}. Thus, using \eqref{e:eqint} and Proposition \ref{p:comp_Pivotal}, we have that  
  \begin{equation*}
 \E[|\gamma_L|] = 4\E[|\mathcal{P}_L|]  = 2\E[|\mathcal{P}^q_L|],
  \end{equation*}
  where $\mathcal{P}_L^q$ denotes the quenched pivotal process in Definition \ref{def:qpiv} associated with $f_L$. The result now follows from the estimate $\E[|\mathcal{P}^q_L|] \asymp L^2\alpha_4(L)$; see \cite[footnote 6]{Vann2021}. In particular, the upper bound follows from \cite[(D.1) and (D.2)]{Vann2021} and the lower bound can be derived by using similar arguments as that for the Boolean model in the proof of \ref{A3}; also see the estimates in \cite[Section 4]{Vann2019}.
\end{proof}
It is also possible to use \eqref{e:onepointcomparison} and estimates on annealed spectral sample $\gamma_L^{an}$ associated with $f_L$ to obtain the above conclusion; see \cite[Section 2.3]{Vann2021}. Indeed, we have that all the five processes, namely, the spectral process, the projected spectral sample, the annealed spectral sample, the pivotal process and the quenched pivotal process, have the same expected total size of $L^2\alpha_4(L)$, up to constant factors.

 Next, we consider the {\em frozen dynamics} for the Voronoi percolation; see \cite{Vann2021}. In this dynamics, the points remain fixed (i.e., frozen) and the colour of each cell is resampled at rate $1$, independent of other cells. Let $\eta_{t}$ denote the point process after time $t>0$ in the frozen dynamics. Note that $(\eta_t)^{pr} = \eta^{pr}$ for all $t \geq 0$. Then we have the following covariance comparison inequality, which was communicated to us by H. Vanneuville. 
	
	\begin{prop}[\cite{Vann2022}]\label{p:covcomp} For any $t>0$, let $\eta^t$ and $\eta_t$ be the configurations after time $t$ in the OU and frozen dynamics, respectively. Then
		$$
		\mathrm{Cov}(f_L(\eta), f_L(\eta^{t}))  \le \mathrm{Cov}(f_L(\eta), f_L(\eta_{t})).
		$$
	\end{prop}
\begin{proof}
	We first look at the left-hand side. We have
	$$
	\mathrm{Cov}(f_L(\eta), f_L(\eta^{t})) = \E [f_L(\eta) f_L(\eta^{t})] - \E[f_L(\eta)]^2.
	$$
	Now we identify the process at time $t$ as $\eta^t = \eta_1 + \eta_2$ with $\eta_1$ obtained by thinning $\eta$ upon deleting each point independently with probability $(1-e^{-t})$ and $\eta_2$ an independent homogeneous Poisson process with intensity $(1-e^{-t})$. Notice that if we condition on $\eta_1$, then the processes $\eta$ and $\eta^t$ become independent. Hence, we obtain
	\begin{eqnarray}\label{eq:VarOU}
		\mathrm{Cov}(f_L(\eta), f_L(\eta^{t})) & =& \E \left[\E \left[f_L(\eta) f_L(\eta^{t}) \mid \eta_1\right]\right] - \E[f_L(\eta)]^2 \nonumber\\
	& =& \E \left[\E \left[f_L(\eta) \mid \eta_1\right]^2\right] - \E[f_L(\eta)]^2 \nonumber \\ 
 & = & \mathrm{Var}(\E \left[f_L(\eta) \mid \eta_1\right]),
	\end{eqnarray}
where in the penultimate step, we have used that the distribution of $\eta$ and $\eta^t$ are the same given $\eta_1$.

Let us now consider the right-hand side of the inequality in the proposition. We again have that $\mathrm{Cov}(f_L(\eta), f_L(\eta_{t})) = \E [f_L(\eta) f_L(\eta_{t})] - \E[f_L(\eta)]^2$. As before, we identify $\eta_t = \eta_1 + \eta_2'$, where $\eta_1$ now denotes the points that are untouched by time $t$, while $\eta_2'$ is the collection of points whose colors are resampled by time $t$ in the frozen dynamics, i.e., $\eta_2'^{pr} = (\eta\setminus \eta_1)^{pr}$ while their color sets are completely independent. Hence, if we condition on $\eta_1 \cup(\eta\setminus \eta_1)^{pr}$, then $\eta$ and $\eta_t$ become independent. Arguing similarly as above (in the case of the OU dynamics) now yields
	\begin{eqnarray}
	\mathrm{Cov}(f_L(\eta), f_L(\eta_{t})) & = & \E \left[\E \left[f_L(\eta) f_L(\eta_{t}) \mid \eta_1 \cup (\eta\setminus \eta_1)^{pr} \right] \right] - \E[f_L(\eta)]^2 \nonumber \\
	& = & \E \left[\E \left[f_L(\eta) \mid \eta_1 \cup (\eta \setminus \eta_1)^{pr} \right]^2 \right] - \E[f_L(\eta)]^2 \nonumber \\
\label{eq:VarFr} &  = & \mathrm{Var}(\E \left[ f_L(\eta) \mid \eta_1 \cup(\eta\setminus \eta_1)^{pr} \right] ). 
        \end{eqnarray}
Thus, by \eqref{eq:VarOU} and \eqref{eq:VarFr}, the covariances are variances of the conditional expectation of the same function $f_L(\eta)$, albeit in the case of OU dynamics, we condition on a smaller $\sigma$-algebra. Hence, the variance is also smaller in the case of OU dynamics. This yields the desired conclusion.
\end{proof}

	\begin{proof}[Proof of Theorem \ref{t:sharpNS_voronoi}] Note that Proposition~\ref{p:exppivotalsetvoronoi} and Proposition~\ref{p:NStabgen} imply that $$\lim_{L \to \infty} \mathbb{P}\{f_L(\eta) \neq f_L(\eta^{t_L})\}  =0$$ if $t_L L^2\alpha_4(L) \to 0$. Next, letting $\eta_{t},\,  t \geq 0$ be the process under the frozen dynamics, we know from \cite[Theorem 1.7]{Vann2021} that when $t_L L^2\alpha_4(L) \to\infty$ as $L \to \infty$,
	$$
	\lim_{L \to \infty} \mathrm{Cov}(f_L(\eta), f_L(\eta_{t_L})) = 0.
	$$
 Thus, Proposition~\ref{p:covcomp} yields noise sensitivity of $\eta^t$ when $t_L L^2\alpha_4(L) \to\infty$ as $L \to \infty$.
\end{proof}

\section*{Acknowledgements}	
The authors thank  G\"{u}nter Last for discussions during the initial stages of the project and Hugo Vanneuville for his comments on an earlier draft as well as for showing the authors the covariance inequality stated in Proposition \ref{p:covcomp}. CB was supported in part by the German Research Foundation (DFG) Project 531540467. Large parts of this work was done when CB was employed by the University of Luxembourg. GP was supported by the Luxembourg National Research Fund (Grant: 021/16236290/HDSA).  DY's research was partially supported by SERB-MATRICS Grant MTR/2020/000470 and CPDA from the Indian Statistical Institute. The work also benefited from his visits to the University of Luxembourg and he is thankful to the university for hosting him.

\bibliographystyle{abbrv}
\bibliography{Spectra}
\addcontentsline{toc}{section}{References}

\appendix

\section{A general approach to sharp noise sensitivity and instability}
\label{s:outline_approach}

In this section, we outline a the general approach to show sharp noise sensitivity and instability for Poisson functionals under the OU dynamics and then, using this general approach, we prove Proposition \ref{p:sharpNStab}. As demonstrated in the main text, such a result yields geometric conditions implying sharp noise instability in continuum percolation models. The proof of Proposition \ref{p:sharpNStab} is provided in Section \ref{s:pfSharpNStab}.
 
Recall notations from Sections \ref{ss:prem} and \ref{sec:Spec}. We say $f_L : \bN \to \{-1,1\},  L \geq 1$ is a family of non-degenerate Boolean functions if $\limsup_{L} \P(f_L(\eta)=1), \limsup_{L} \P(f_L(\eta)=-1)<1$. We shall abbreviate $\gamma_L := \gamma_{f_L}$ for convenience.  We shall assume that $\eta^t, t \geq 0$ is a Poisson point process evolving under the OU dynamics and with $\eta^0 = \eta$, which is a Poisson process with a given intensity measure $\lambda$ on $\X$.
	
	The cornerstones of our approach to sharp noise instability in \eqref{eq:sNStab} for Poisson functionals are the following two propositions.  This is analogous to the approach for Rademacher functionals in \cite[Chs.\ 9 and 10]{GS} and in \cite{GPS10}.  We shall state the two propositions now and discuss their implications for noise instability. The proofs are provided at the end of the section for completeness.  
	
	\begin{prop}{\rm (See \cite[Proposition 6.6]{LPY2021})}
		\label{p:NStabgen}
		If $t_L \E[|\gamma_L|] \to 0$ then $f_L$ is noise stable at time-scale $t_L$ i.e.,
		\begin{equation}
			\label{e:quantNStab}
			\lim_{L  \to \infty} \P(f_L(\eta) \neq f_L(\eta^{t_L})) = 0.
		\end{equation}
In particular, when $\{f_L : L \ge 1\}$ is asymptotically non-degenerate,  $f_L$ is not noise sensitive if $t_L \E[|\gamma_L|] \to 0$.
	\end{prop}
	\begin{prop}
		\label{p:noNStabgen}
		Suppose that for some $c_0,  c > 0$,  we have
		\begin{equation}
			\label{e:2momspectra}
			\liminf_{L \to \infty} \P\left( |\gamma_L| \geq c \E[|\gamma_L|] \right) \geq c_0.
		\end{equation}
		Then if $t_L\E[|\gamma_L|] \to \infty$, $f_L$ is not noise stable  at time-scale $t_L$ i.e.,
		\begin{equation}
			\label{e:quantnoNStab}
			\liminf_{L  \to \infty} \, \P(f_L(\eta) \neq f_L(\eta^{t_L})) \geq c_0/2.  
		\end{equation}
	\end{prop}
Note that \eqref{e:quantnoNStab} necessarily implies that $f_L$ is non-degenerate asymptotically.
%It is possible to obtain further quantitative results as in \cite[Examples 9.1 and 9.2]{GS} for Poisson functionals as well. 
Based on the above propositions,   we now outline the general approach to sharp noise instability and sensitivity. We remark here about a general approach for noise sensitivity as well, even though we do not use it directly in the present work. STEP 4 below can be justified via similar arguments used in the proofs of the propositions above, but we skip the details, as we do not use it. 
	\begin{remark} {\rm The following are the general steps to follow for showing sharp noise instability/sensitivity in a continuum percolation model based on a Poisson process under the OU dynamics:
		\label{rem:outlineNS} 
		\hspace*{0.2cm}
		\begin{itemize}
            \item[\underline{STEP 1:}] Find $A_L$ such that 
			$$ \E[|\gamma_L|] = \E[ \int_{\X} |D_xf_L(\eta)|^2 \lambda(\md x)] \asymp A_L.$$ 
			
			\item[\underline{STEP 2:}]  Deduce noise stability (i.e., show \eqref{e:quantNStab}), as well as no noise sensitivity at time-scale $t_L$ for $t_LA_L \to 0$ using Proposition \ref{p:NStabgen}.
			
			\item[\underline{STEP 3:}] Verify \eqref{e:2momspectra} and thus deduce no noise stability, or noise instability, at time-scale $t_L$ for $t_LA_L  \to \infty$. This together with the conclusion from STEP 2 yields sharp noise instability.
			
			\item[\underline{STEP 4:}] Further to show noise sensitivity at time-scale $t_L$ for $t_LA_L \to \infty$, it suffices to show that
   \begin{equation*}
			\label{e:sharp2momspectra}
			\lim_{c \to 0} \liminf_{L \to \infty} \P\left( |\gamma_L| \geq c \E[|\gamma_L|] \right) = 1.
		\end{equation*}
Note that this together with the conclusion in STEP 2 yields sharp noise sensitivity.
		\end{itemize}
  }
	\end{remark}
	
	We shall first make some basic observations regarding $f_L(\eta^t)$ and its relation to the spectral point process $\gamma_L$.  These basic observations are already present in \cite[Section 6]{LPY2021} but will be re-cast here in terms of the spectral point process.  
	
	Firstly,  from {\em Mehler's formula} (see \cite[formula (80)]{Last16} or \cite[formula (3.13)]{LPS}),   we obtain that for each $t \geq 0$,  
	\begin{equation*}
		\label{e:Ttfeta}
		T_tf_L(\eta) := \E[f_L(\eta^t) | \eta] = \sum_{k=0}^{\infty}e^{-kt} I_k(u_k).
	\end{equation*}
	Thus from the above representation,  orthogonality property of the Wiener-It\^o kernels and definition of spectral point process,  we obtain that
	\begin{equation}
		\label{e:Eprodspectra}
		\E[f_L(\eta)f_L(\eta^t)] =  \sum_{k=0}^{\infty}e^{-kt}\E[I_k(u_k)^2] = \E[e^{-t|\gamma_L|}].
	\end{equation}
	Now using the fact that $f_L \in \{-1,1\}$,  we also have that
	\begin{equation}
		\label{e:Eprodprobneq}
		\E[f_L(\eta)f_L(\eta^t)] =  \P(f_L(\eta) = f_L(\eta^t)) -  \P(f_L(\eta) \neq f_L(\eta^t)) = 1 - 2 \P(f_L(\eta) \neq f_L(\eta^t)) . 
	\end{equation}
	\begin{proof}[Proof of Proposition \ref{p:NStabgen}]
		Applying Jensen's inequality to \eqref{e:Eprodspectra} and using the assumption $t_L\E[|\gamma_L|] \to 0$,  we obtain that $\E[f_L(\eta)f_L(\eta^{t_L})] \to 1$ as $L \to \infty$.  Thus \eqref{e:quantNStab} follows from \eqref{e:Eprodprobneq}.  

Also, \eqref{e:quantNStab} and \eqref{e:Eprodprobneq} imply that
  $$
  \liminf_{L \to \infty} \operatorname{Cov}(f_L(\eta), f_L(\eta^t)) = 1- \limsup_{L \to \infty} \E[f_L(\eta)]^2>0,
  $$
  where the final step follows due to the non-degeneracy of $f_L$.
	\end{proof}
	\begin{proof}[Proof of Proposition \ref{p:noNStabgen}]
		Using \eqref{e:Eprodspectra} and assumption \eqref{e:2momspectra},   we derive that
		\begin{align*}
		\limsup_{L \to \infty} \E[f_L(\eta)f_L(\eta^{t_L})] & \leq \limsup_{L \to \infty} e^{-ct_L\E[|\gamma_L|]} \P(|\gamma_L| \geq c\E[|\gamma_L|]) + \limsup_{L \to \infty}  \P(|\gamma_L| < c \E[|\gamma_L|] ) \\
		& \leq \limsup_{L \to \infty} e^{-ct_L\E[|\gamma_L|]}  + (1-c_0). 
		\end{align*}
		Thus if $t_L\E[|\gamma_L|] \to \infty$,  we obtain that
		$$ \limsup_{L \to \infty} \E[f_L(\eta)f_L(\eta^{t_L})] \leq 1 - c_0$$
		and now using \eqref{e:Eprodprobneq},   \eqref{e:quantnoNStab} follows. 
	\end{proof}
\subsection{Proof of Proposition \ref{p:sharpNStab}}
\label{s:pfSharpNStab}

To prove Proposition \ref{p:sharpNStab}, we first need a lemma bounding the second moment of the size of the spectral point process. We will generally denote elements in $\X = \R^2 \times \M$ by $\tx$ and its projection in $\R^2$ simply by $x$.
	\begin{lemma}
 \label{l:2ndmomspectravoronoi}
 Let assumptions \ref{A1} -- \ref{A3} of Section \ref{ss.gen_sns_cont} be in force. Then, for $\rho \in (0,\infty)$ as in assumption \ref{A4},
		\begin{equation*}
			\label{e:2ndmomspectravoronoi}
			\E[|\gamma_L \cap \widehat W_{\rho L}|^2] \lesssim \E[|\gamma_L \cap \widehat W_{\rho L}|]^2. 
		\end{equation*}
	\end{lemma}
	
	\begin{proof} Note that by Remark \ref{e:bagatelle} (ii), the condition \eqref{e:integram} is trivially satisfied for $C=\widehat W_{\rho L}$. Thus, using Proposition~\ref{p:ell} together with \ref{A2} -- \ref{A3}, we derive
		\begin{eqnarray}
		&&\E[|\gamma_L \cap \widehat W_{\rho L}|^2] \lesssim \int_{\widehat W_{\rho L}}\E |D_{\tx} f_L(\eta)|^2 \lambda(\md \tx) +  \int_{\widehat W_{\rho L}^2}\E |D_{\tx}D_{\ty} f_L(\eta)|^2 \lambda^2(\md \tx,\md \ty) \nonumber \\
		&& \lesssim L^2 \alpha_4(L) + \int_{\widehat W_{\rho L}^2,|x-y|< L^{\eps/3}} \E |D_{\tx}D_{\ty} f_L(\eta)|^2 \lambda^2(\md \tx,\md \ty) \nonumber \\
        &&\qquad\qquad + \int_{W_{\rho L}^2,|x-y|\ge L^{\eps/3} }\left[\frac{\alpha_4^2(\rho L)}{\alpha_4^2(|x-y|/4,\rho L)}
        + \text{O}(L^2 \alpha_4(L)^2) \right]\md x\md y \nonumber \\
		&& \lesssim L^2 \alpha_4(L) +
		\int_{\widehat W_{\rho L}^2,|x-y|< L^{\eps/3}} \E |D_{\tx}D_{\ty} f_L(\eta)|^2 \lambda^2(\md \tx,\md \ty) \nonumber \\
        && \qquad \qquad + \int_{W_{\rho L}^2, |x-y| \ge 4} \frac{\alpha_4^2(\rho L)}{\alpha_4^2(|x-y|/4,\rho L)} \md x\md y + \text{O}((L^2 \alpha_4(L))^2).  \label{e:2ndspecmom_ub}
		\end{eqnarray}

		We shall now bound the two integrals above, starting with the second one. Firstly, for each $y \in W_{\rho L}$, using \ref{A1} and \ref{A2}, 
		\begin{multline*}
		\int_{W_{\rho L}, |x-y| \ge 4} \frac{1}{\alpha_4^2(|x-y|/4,\rho L)} \md x  \lesssim \sum_{k=0}^{\lfloor \log_2 \rho L \rfloor } \int_{W_{\rho L}} \frac{\mathds{1}(|x-y|/4 \in [2^k,2^{k+1}))}{\alpha_4^2(|x-y|/4,\rho L)} \md x\\
		 \lesssim \sum_{k=0}^{\lfloor \log_2 \rho L \rfloor}\frac{4^k}{\alpha_4^2(2^k,L)} \lesssim \eps^{-2}\sum_{k=0}^{\lfloor \log_2 \rho L \rfloor }\frac{L^{2-\eps}2^{2k}}{(2^k)^{2-\eps}} =\eps^{-2} L^{2-\eps} \sum_{k=0}^{\lfloor \log_2 \rho L \rfloor} 2^{k\eps} \lesssim \eps^{-2} L^2.
		\end{multline*}
        Thus, the second integral is $\text{O}((L^2 \alpha_4(L))^2)$ and it remains only to bound the first integral in \eqref{e:2ndspecmom_ub}. Note that $D_{\tx}D_{\ty} f_L(\eta) \neq 0$ implies that either $D_{\tx} f_L(\eta + \delta_{\ty}) \neq 0$ or $D_{\ty} f_L(\eta) \neq 0$. Thus we have
		\begin{multline}\label{e:DxDyfint_ub}
		 \int_{\widehat W_{\rho L}^2,|x-y|< L^{\eps/3}} \E |D_{\tx}D_{\ty} f_L(\eta)|^2 \lambda^2(\md \tx,\md \ty) \lesssim \int_{\widehat W_{\rho L}} \int_{B(y,L^{\eps/3}) \times \M} \P(D_{\ty} f_L(\eta) \neq 0)  \lambda^2(\md \tx,\md \ty)\\
		 +\int_{\widehat W_{\rho L}} \int_{B(x,L^{\eps/3}) \times \M} \P(D_{\tx} f_L(\eta + \delta_{\ty}) \neq 0)  \lambda^2(\md \ty,\md \tx).
		\end{multline}
		Using \ref{A3} and noting from \ref{A1} that $L^2 \alpha_4(L)  \gtrsim L^\eps$, the first integral in \eqref{e:DxDyfint_ub} is bounded by
		\begin{eqnarray}
			&&\int_{\widehat W_{\rho L}} \int_{B(y,L^{\eps/3}) \times \M} \P(D_{\ty} f_L(\eta) \neq 0)  \lambda^2(\md \tx,\md \ty) \nonumber \\
            &&\lesssim L^{2\eps/3 } \int_{\widehat W_{\rho L}}\E |D_{\ty} f_L(\eta)|^2  \lambda(\md \ty)
            \lesssim L^{2\eps/3 } L^2 \alpha_4(L) \lesssim (L^2 \alpha_4(L) )^2.   \label{e:DxDyfint_ub1} 
		\end{eqnarray}
		For the second integral in \eqref{e:DxDyfint_ub}, by the Mecke formula notice
		\beaa
		&&\int_{\widehat W_{\rho L}} \int_{B(x,L^{\eps/3}) \times \M} \P(D_{\tx} f_L(\eta + \delta_{\ty}) \neq 0)  \lambda^2(\md \ty,\md \tx)\\
		&&=\int_{\widehat W_{\rho L}}  \E \sum_{\ty \in \eta \cap B(x,L^{\eps/3}) \times \M} \mathds{1}(D_{\tx} f_L( \eta) \neq 0) \lambda(\md \tx)=\int_{\widehat W_{\rho L}}  \E  \left[X \mathds{1}(D_{\tx} f_L(\eta) \neq 0)\right] \lambda(\md \tx),
		\eeaa
		where $X:=|\eta \cap B(x,L^{\eps/3}) \times \M|$ which is distributed as a Poisson random variable with mean $\pi L^{2\eps/3}$.
		Notice by \ref{A3} and \ref{A1} that
		\begin{equation*}
			\int_{\widehat W_{\rho L}}  \E  \left[X \mathds{1}(X\le L^\eps)\mathds{1}(D_{\tx} f_L( \eta) \neq 0)\right] \lambda(\md \tx) \lesssim  L^{\eps }L^2 \alpha_4(L) \lesssim (L^2 \alpha_4(L) )^2.
		\end{equation*}
		On the other hand, applying the Cauchy-Schwarz inequality followed by Markov's inequality yields that
		\begin{multline*}
		\int_{\widehat W_{\rho L}}  \E  \left[X \mathds{1}(X >L^\eps)\mathds{1}(D_{\tx} f_L(\eta) \neq 0)\right] \lambda(\md \tx)
		\lesssim L^2 \sqrt{\E X^2 \P(e^X>e^{L^\eps})}\\
		\lesssim L^{2\eps/3} L^2 e^{-L^\eps/2} \sqrt{\E e^X} \lesssim L^{2+2\eps/3} e^{-L^\eps/2} e^{2 L^{2\eps/3}}=o(1).
		\end{multline*}
	Noting that $L^2 \alpha_4(L)  + o(1) \lesssim (L^2 \alpha_4(L) )^2$, we have that the second integral in \eqref{e:DxDyfint_ub} is also $\text{O}((L^2 \alpha_4(L))^2)$ and combined with \eqref{e:DxDyfint_ub1}, this gives that the integral in the left-hand side of \eqref{e:DxDyfint_ub} is also $\text{O}((L^2 \alpha_4(L))^2)$. Substituting this into \eqref{e:2ndspecmom_ub} along with the derived bounds for the second integral therein, we obtain
		\begin{equation*}
			\E[|\gamma_L \cap \widehat W_{\rho L}|^2] \lesssim (L^2 \alpha_4(L) )^2 \lesssim \E[|\gamma_L \cap  \widehat W_{\rho L}|]^2,
		\end{equation*}
		where the final step follows from the lower bound in \ref{A3}. 
	\end{proof}
	
	Lemma~\ref{l:2ndmomspectravoronoi} yields sharp noise instability in Proposition~\ref{p:sharpNStab} via Steps 1-3 of Remark \ref{rem:outlineNS}. Recall the Paley-Zygmund inequality: for a random variable $Z$ with finite variance and $s \in [0,1]$,
	\begin{equation}\label{e:PZ}
		\P(Z>s \E Z) \ge (1-s)^2 \frac{(\E Z)^2}{\E Z^2}.
	\end{equation}

 \begin{proof}[Proof of Proposition \ref{p:sharpNStab}]  
     By \eqref{e:PZ} with $s=1/2$ and Lemma~\ref{l:2ndmomspectravoronoi}, we have
		\begin{equation*}
			\P\left(|\gamma_L \cap  \widehat W_{\rho L}| \ge \frac{\E |\gamma_L \cap  \widehat W_{\rho L}|}{2}\right) \ge \frac{(\E |\gamma_L \cap  \widehat W_{\rho L}|)^2}{4 \E |\gamma_L \cap  \widehat W_{\rho L}|^2} \ge C
		\end{equation*}
		for some constant $C$ not depending on $L$. Thus, by the lower bound in \ref{A3}, there exist $C,C_0 > 0$ such that for $L$ large enough,
		$$ \P(|\gamma_L| \geq CL^2\alpha_4(L)) \geq \P(|\gamma_L \cap  \widehat W_{\rho L}| \geq CL^2\alpha_4(L)) \geq C_0,$$
		and consequently, sharp noise instability follows from Propositions~\ref{p:ell}, \ref{p:NStabgen} and \ref{p:noNStabgen} via steps 1-3 of Remark \ref{rem:outlineNS}. 
 \end{proof}

\section{Proof of Lemma \ref{lem:z1max}}
\label{s:pfz1max}
For $x,y \in \R^2$, define the random variables
	\begin{eqnarray*}
		Z_1(x,y,\eta)&= \E [\eps_1 \eps_2 D_{x,\eps_1} D_{y,\eps_2} F(\eta) | \eta]\\
		Z_2(x,y,\eta)&= \E [\eps_1 D_{x,\eps_1} D_{y,\eps_2} F(\eta) | \eta]\\
		Z_3(x,y,\eta)&= \E [\eps_2 D_{x,\eps_1} D_{y,\eps_2} F(\eta) | \eta]\\
		Z_4(x,y,\eta)&= \E [D_{x,\eps_1} D_{y,\eps_2} F(\eta) | \eta],
	\end{eqnarray*}
where $\eps_1$ and $\eps_2$ are independent Rademacher variables. Then from \eqref{e:psik} and \eqref{e:phik}, we have
$$
\Psi_2(x,y) = \E [Z_1^2(x,y,\eta)] \quad \text{and} \quad \Phi_2(x,y) = \sum_{i=1}^4 \E [Z_i^2(x,y,\eta)].
$$
Writing $Z_1$ explicitly, after cancellations, we have
	\begin{eqnarray*}
		4 Z_1(x,y,\eta)&=& [(F(\eta+x^1+y^1) - F(\eta+x^1) - F(\eta+y^1) + F(\eta))\\
		&&-[(F(\eta+x^1+y^{-1}) - F(\eta+x^1) - F(\eta+y^{-1}) + F(\eta))\\
		&&-[(F(\eta+x^{-1}+y^1) -F(\eta+x^{-1}) - F(\eta+y^1) + F(\eta))\\
		&&+[(F(\eta+x^{-1}+y^{-1}) - F(\eta+x^{-1}) - F(\eta+y^{-1}) + F(\eta))]\\
		&=&F(\eta+x^1+y^1) - F(\eta+x^1+y^{-1}) -F(\eta+x^{-1}+y^1) +F(\eta+x^{-1}+y^{-1}).
	\end{eqnarray*}
 First notice that on $B(x,y)\cup B(y,x)$, $Z_1=0$ so that we have $\sum_{i=1}^4 Z_i^2 \le 48$. 
 Next fix $\eta \in (B(x,y)\cup B(y,x))^c$, such that $Z_1 \neq 0$. This implies
	\begin{enumerate}[(i)]
		\item $f_L(\eta+x^1+y^1) = 1 = - f_L(\eta+x^{-1}+y^{-1})$, and
		\item $f_L(\eta+x^1+y^{-1}) = f_L(\eta+x^{-1}+y^1)$.
	\end{enumerate}
	By symmetry, without loss of generality, assume $f_L(\eta+x^1+y^{-1}) = f_L(\eta+x^{-1}+y^1)=-1$. In this case, $|Z_1|=1/2$. Writing $Z_2$ in the same manner, after cancellation and plugging in the values above, and noticing by increasingness that $f_L(\eta + x^{-1})=f_L(\eta + y^{-1})=-1$, one obtains that
	$$
	4Z_2(x,y,\eta)=-2f_L(\eta+x^1) \implies |Z_2| \le 1/2 = |Z_1|.
	$$
	By symmetry, the same holds for $Z_3$. Finally, writing $Z_4$ explicitly, and plugging in the values, we obtain
	$$
	4Z_4(x,y,\eta) = 2+ 4f_L(\eta) - 2f_L(\eta+x^1) -2f_L(\eta+y^1).
	$$
	Again, by increasingness, if either $f_L(\eta)=1$ (then all else in the r.h.s.\ is 1), or $f_L(\eta)=-1$ and $\min\{f_L(\eta+x^1) ,f_L(\eta+y^1)\}=-1$ (in this case, cancellation occurs), then $|Z_4| \le 1/2=|Z_1|$. Finally, when
	$$
	    f_L(\eta)=-1,\quad  \text{and} \quad f_L(\eta+x^1)=f_L(\eta+y^1)=1,
	$$
 then $|Z_4|= 3/2 = 3|Z_1|$. Hence, when $Z_1 \neq 0$ on $(B(x,y)\cup B(y,x))^c$, we have $|Z_2|,|Z_3|,|Z_4| \le 3 |Z_1|$. 
 
 Finally assume $Z_1=0$ for some $\eta \in (B(x,y)\cup B(y,x))^c$. By increasingness, this corresponds to the case when either
	\begin{enumerate}[(i)]
		\item $F(\eta+x^1 + y^1)=-1$ (all terms are zero in this case, that is $|Z_2|= |Z_3|= |Z_4|=0$), or
		\item $F(\eta+x^{-1} + y^{-1})=1$ (again, all terms are 1 in this case, so that that is $|Z_2|= |Z_3|= |Z_4|=0$), or
		\item $F(\eta+x^1+y^{-1}) \neq F(\eta+x^{-1}+y^1)$.
	\end{enumerate}
 We now just need to consider the case (iii) when $F(\eta+x^1 + y^1)=1=-F(\eta+x^{-1} + y^{-1})$. By symmetry, without loss of generality, assume $F(\eta+x^1+y^{-1})=1= - F(\eta+x^{-1}+y^1)$. It is straightforward to check %by writing out explicitly 
 that $Z_2=0$ in this case. Writing out $Z_3$ and $Z_4$ explicitly, we also have that
	$$
	4 Z_3 = 2(F(\eta+y^1) - F(\eta+y^{-1}))
	$$
	and
	$$
	4 Z_4(x,y,\eta) = 4 F(\eta) -2 F(\eta+y^1) - 2F(\eta+y^{-1}).
	$$
 Since we are on $(B(x,y)\cup B(y,x))^c$, we necessarily have $F(\eta+y^1) = F(\eta+y^{-1})$ so that $Z_3=0$. Also, their common value must also be equal to $F(\eta)$ which yields that $Z_4=0$. 
Combining all the above cases yields the result. \qed
%\end{proof}

\end{document}